\documentclass{amsart}
\usepackage{latexsym,amsxtra,amscd,ifthen}
\usepackage{amsfonts}
\usepackage{verbatim}
\usepackage{amsmath, mathrsfs}
\usepackage{amsthm}
\usepackage{amssymb}
\usepackage{url,color}
\usepackage[arrow,matrix]{xy}
\usepackage{hyperref}

\usepackage{anysize}\marginsize{25mm}{25mm}{30mm}{30mm}
\allowdisplaybreaks[4]
\numberwithin{equation}{section}
\numberwithin{table}{section}
\numberwithin{equation}{section}
\newtheorem{theorem}{Theorem}[section]
\newtheorem{lemma}[theorem]{Lemma}
\newtheorem{proposition}[theorem]{Proposition}

\newtheorem{corollary}[theorem]{Corollary}
\theoremstyle{definition}
\newtheorem{example}[theorem]{Example}
\newtheorem{remark}[theorem]{Remark}

\theoremstyle{definition}
\newtheorem{definition}[theorem]{Definition}

\usepackage{bbm}

\def\la{\lambda}
\def \ot{\otimes}

\def \vt{\vartheta}

\def \As{\mathcal{A}ss}
\def \tt{\mathcal{T}}
\def \Com{\mathcal{C}om}
\def \OP{\mathcal{O}p}

\def\Pois{\mathcal{P}ois}
\def\Id{{\rm Id}}

\newcommand{\up}[1]{{^{#1}{\it \!\Upsilon}}}
\def \Uni{\mathcal{U}ni}
\def \I{\mathcal{I}}
\def \To{\longrightarrow}
\def \ip{{\mathcal{P}}}
\def \iq{\mathcal {Q}}
\def \iu{\!{\it{\Upsilon}}}

\def \dim{\operatorname{dim}}
\def\gkdim{\operatorname{GKdim}}

\def \Hom{\operatorname{Hom}}
\def \End{\operatorname{End}}

\def \id{\operatorname{id}}
\def \ker{\operatorname{Ker}}

\def \1{\mathbbm 1}

\def \D{\Delta}
\def \d{\delta}

\def \k{\Bbbk}

\def \N{\mathbb{N}}
\def \S{\mathbb{S}}

\def \rad{\operatorname{rad}}

\def \n{[n]}
\def \d{[d]}

\def\M{\mathcal{M}}
\def\rMod{\mbox{\rm Mod-}}

\begin{document}
\title{Truncation of unitary Operads}

\subjclass[2010]{}

\keywords{}

\author[Y.-H. Bao]{Yan-Hong Bao}
\address{(Bao) School of Mathematical Sciences, Anhui University,
Hefei 230601, China}
\email{baoyh@ahu.edu.cn, yhbao@ustc.edu.cn}

\author[Y. Ye]{Yu Ye}
\address{(Ye) School of Mathematical Sciences,
University of Sciences and Technology of China, Hefei 230001, China}
\email{yeyu@ustc.edu.cn}

\author[J.J. Zhang]{James J. Zhang}
\address{(Zhang) Department of Mathematics, Box 354350,
University of Washington, Seattle, Washington 98195, USA}
\email{zhang@math.washington.edu}

\begin{abstract}
We introduce truncation ideals of a $\Bbbk$-linear unitary symmetric
operad and use them to study ideal structure, growth property and
to classify operads of low Gelfand-Kirillov dimension.
\end{abstract}

\subjclass[2000]{Primary 18D50, 55P48, 16W50, 16P90}



\keywords{operad, unitary, 2-unitary, operadic ideal, truncation,
Gelfand-Kirillov dimension, classification, generating function,
basis theorem}

\date{\today}

\maketitle


\setcounter{section}{-1}

\section{Introduction}

Operad theory originates from work of Boardman-Vogt \cite{BV} and May 
\cite{Ma} in homotopy theory in 1970s. Since then many applications of 
both topological and algebraic operads have been discovered in algebra, 
category theory, combinatorics, geometry, mathematical physics and 
topology \cite{Fr1, Fr2, LV, MSS}. In this paper we study operads from
the algebraic viewpoint.

Following tradition, let $\As$ denote the associative algebra operad 
that encodes the category of unital associative algebras. (In the book 
\cite{LV}, it is denoted by $uAs$.) Given an operadic ideal $\I$ of 
$\As$, one can define the quotient operad $\As/\I$. Quotient operads 
of $\As$ relate to polynomial identity algebras (PI-algebras) closely. 
In fact, a PI-algebra is equivalent to an algebra over $\As/\I$ for 
some nonzero operadic ideal of $\As$ \cite{Am, Kle}. It is worth 
mentioning that an operadic ideal is essentially equivalent to
so-called $T$-ideal. For an introduction to PI-algebras and $T$-ideals, 
we refer to \cite[Chapter 13]{McR}.

We are mainly interested in those operads that
have some common properties with $\As/\I$. Let $\Bbbk$ be a base field.
Let $\ip:=\{\ip(n)\}_{n\geq 0}$ denote a $\Bbbk$-linear operad. Recall
that $\ip$ is {\it unitary} if $\ip(0)=\Bbbk \1_0$ with a basis
element $\1_0$, see \cite[Section 2.2]{Fr1}. Operads in this paper
are usually unitary. We say $\ip$ is {\it 2-unitary},
if $\ip$ is unitary and there is an element $\1_2\in \ip(2)$ such that
\begin{equation}
\label{E0.0.1}\tag{E0.0.1}
\1_2 \circ (\1_0,\1)=\1=\1_2\circ (\1, \1_0),
\end{equation}
where $\1\in \ip(1)$ is the identity of the operad $\ip$ and $\circ$ is
composition in $\ip$. A 2-unitary operad
$\ip$ is called $2a$-unitary if, further,
\begin{equation}
\notag
\1_2 \circ (\1_2,\1)=\1_2\circ (\1, \1_2).
\end{equation}
Note that every quotient operad $\As/\I$ is $2a$-unitary and that
there are many other interesting 2-unitary (respectively, $2a$-unitary)
operads [Example \ref{xxex2.1} and Lemma \ref{xxlem2.2}].

All operads in this paper are $\Bbbk$-linear. An operad usually 
means a {\it symmetric} operad and the word symmetric could be 
omitted. Plain operads are used in a few places.

\subsection{Definition of truncations}
\label{xxsec0.1}
Given a unitary
operad $\ip$, one can define restriction operators \cite[Section 2.2.1]{Fr1}
as follows.  We are using different notation from \cite{Fr1} and
explanations concerning the restriction operators are given in
\cite[Section 2.2]{Fr1}. Let $\n$ be the set $\{1, \cdots, n\}$ and $I$ be a
subset of $\n$. Let $\chi_{I}$ be the characteristic function of $I$,
i.e. $\chi_{I}(x)=1$ for $x\in I$ and $\chi_{I}(x)=0$ otherwise. Then
one defines the {\it restriction operator} $\pi^I: \ip(n)\to \ip(s)$,
where $s=|I|$, by
$$ \pi^I(\theta)=
\theta\circ(\1_{\chi_{I}(1)}, \cdots, \1_{\chi_{I}(n)})
$$
for all $\theta\in \ip(n)$. The restriction operator also appeared in many
other papers, see for example, \cite{LP}. For $k\geq 1$, the {\it $k$-th truncation}
of $\ip$, denoted by $^k \iu$, is defined by
\begin{equation}
\label{E0.0.2}\tag{E0.0.2}
^k\iu(n) =\begin{cases}
\bigcap\limits_{I\subset \n,\, |I|= k-1} \ker\pi^I, & \text{if }n\geq k;\\
\quad\ \ 0, & \text{if } n<k.
\end{cases}
\end{equation}
By convention, ${^0 \iu}=\ip$. The {\it truncation} $\{^k \iu\}_{k\geq 1}$
of $\ip$ is a sequence of ideals that are naturally associated to $\ip$.
In the case of $\ip=\As$,
$${^1 \iu}={^2 \iu} =\ker (\As\to \Com)$$
where $\Com$ is the commutative algebra operad defined by
$\Com(n)=\Bbbk$ for all $n\geq 0$. We will use the
truncation to study the growth of operads, as well as their ideal structure
and classification of operads of low growth.

\subsection{Truncations and Gelfand-Kirillov dimension}
\label{xxsec0.2}
The first application of the truncations concerns the growth property.
The growth of a $T$-ideal (in the theory of PI algebras) has been studied
by many authors, see for instance \cite{KR, GZ1, GZ2, GZ3, GMZ}. This paper
deals with a similar question in the framework of operad theory. Next we 
define the Gelfand-Kirillov dimension of an operad. For the definition of
Gelfand-Kirillov dimension of an algebra, we refer to \cite{KL}. The
{\it Gelfand-Kirillov dimension} (or {\it GKdimension} for short) of
an operad $\ip$ is defined to be
\begin{equation}
\label{E0.0.3}\tag{E0.0.3}
\gkdim \ip:=\limsup_{n\to\infty} \left(\log_{n}
(\sum_{i=0}^{n} \dim_{\Bbbk} \ip(i))\right).
\end{equation}
The {\it exponent} of $\ip$ is defined to be
\begin{equation}
\label{E0.0.4}\tag{E0.0.4}
\exp(\ip):=\limsup_{n\to \infty} (\dim \ip(n))^{\frac{1}{n}}.
\end{equation}
When we talk about the GKdimension or the exponent of an
operad $\ip$, we usually implicitly assume that $\ip$ is
{\it locally finite}, namely,
$\dim_{\Bbbk} \ip(n)<\infty$ for all $n\geq 0$. We
say $\ip$ {\it has polynomial growth} if $\gkdim \ip<\infty$. It is
easy to see that $\gkdim \As=\infty$ and $\gkdim \Com=1$. The {\it
generating series} or {\it Hilbert series} of $\ip$ is defined to be
$$G_{\ip}(t)=\sum_{n=0}^{\infty}\dim_{\Bbbk}\ip(n)t^n\in {\mathbb Z}[[t]].$$
Also see Definition \ref{xxdef4.1}.
\footnote{In \cite[Section 5.1.10, p. 128]{LV}, the
\emph{generating series} of $\ip$ is defined to be
$E_{\ip}(x)\colon
=\sum\limits_{n\ge 0}\dfrac{\dim_{\Bbbk} {\ip}(n)}{n!}x^n$,
which is also called the \emph{Hilbert-Poincar\'{e} series} of
$\ip$.}

Our first result is to give a characterization of operads that have
finite GKdimension. 

\begin{theorem}
\label{xxthm0.1}
Let $\ip$ be a 2-unitary operad.
\begin{enumerate}
\item[(1)]
If $\ip$ has polynomial growth, then the generating series $G_{\ip}(t)$
is rational. As a consequence, $\gkdim \ip\in {\mathbb N}$.
\item[(2)]
$\ip$ has polynomial growth if and only if there is an integer $k$ such
that ${^k\iu}=0$. And
$$\gkdim \ip=\max\{ k\mid {^k \iu}\neq 0\}+1=\min\{ k\mid {^k \iu}= 0\}.$$
\end{enumerate}
\end{theorem}

Theorem \ref{xxthm0.1}(1) answers an open question (or rather fulfills an
expectation) of Khoroshkin-Piontkovski \cite[Expectation 3]{KP} for 2-unitary
symmetric operads. When $\ip$ has finite Gr{\"o}bner basis \cite{KP},
Theorem \ref{xxthm0.1}(2) is a consequence of a more general result
\cite[Theorem 0.1.5]{KP}. Our proof is not dependent on the Gr{\"o}bner
basis. It follows from Corollary \ref{xxcor6.8} that the GKdimension of
a unitary operad can be a non-integer.
There are some other results concerning the {\it exponent} of an
operad, see for example Theorem \ref{xxthm0.8}(2).
In the next corollary, let $\{ ^k \iu\}_{k\geq 0}$ be the
truncation of $\As$.

\begin{corollary}
\label{xxcor0.2}
Let $\I$ be an operadic ideal of $\As$ and $\ip$ be the
quotient operad $\As/\I$. Let $k$ be a positive integer.
Then $\gkdim \ip\le k$ if and only if
$\I\supseteq {^{k}\iu}$. In particular,
$$\gkdim (\As/{^k\iu})=
\begin{cases}
1, & k=1,2,\\
k, & k\geq 3.
\end{cases}$$
\end{corollary}

\subsection{Chain conditions on ideals of an operad}
\label{xxsec0.3}
The second application of the truncations concerns the ideal structure of
operads. We say an operad $\ip$ is \emph{artinian} (respectively, 
\emph{noetherian}) if the set of ideals of $\ip$ satisfies
the descending chain condition (respectively, ascending chain condition).

\begin{theorem}
\label{xxthm0.3}
Let $\ip$ be a 2-unitary operad that is locally finite.
\begin{enumerate}
\item[(1)]
If $\gkdim \ip<\infty$, then
$\ip$ is noetherian.
\item[(2)]
$\gkdim \ip<\infty$ if and only if $\ip$ is artinian.
\item[(3)]{\rm{[An operadic version of Hopkins' Theorem]}}
If $\ip$ is artinian, then it is noetherian.
\end{enumerate}
\end{theorem}

We have a version of Artin-Wedderburn Theorem
for operads. Similar to the definition given before Theorem \ref{xxthm0.3},
we can define left or right artinian operads [Definition \ref{xxdef1.8}(2,3)].
We say an operad $\ip$ is {\it semiprime}, if it does not contain an
ideal $\mathcal{N}\neq 0$ such that $\mathcal{N}^2=0$ [Definition \ref{xxdef1.10}(4)].
An operad $\ip$ is called {\it bounded above} if $\ip(n)=0$
for all $n\gg 0$.

\begin{theorem}[Operadic versions of Artin-Wedderburn Theorem]
\label{xxthm0.4}
Suppose $\ip$ is semiprime. In parts {\rm{(1)}} and {\rm{(2)}},
$\ip$ is either a plain operad or a symmetric
operad. In part {\rm{(3)}}, $\ip$ is a symmetric operad.
\begin{enumerate}
\item[(1)]
If $\ip$ is reduced and left or right artinian, then
$$\ip(n)=
\begin{cases}
0, & n\neq 1,\\
A, & n=1,\end{cases}
$$
where $A$ is a semisimple algebra.
\item[(2)]
If $\ip$ is unitary, bounded above, and left or right artinian, then
$$\ip(n)=\begin{cases}
0, & n\neq 0, 1,\\
\Bbbk, & n=0,\\
A, & n=1,\end{cases}
$$
where $A$ is an augmented semisimple algebra.
\item[(3)]
If $\ip$ is 2-unitary and left or right artinian, then
$\ip$ is as in Example {\rm{\ref{xxex2.3}(1)}} and $\ip(1)$ is an
augmented semisimple algebra.

If, further, $\ip(1)$ is finite dimensional over $\Bbbk$, then $\ip$
is locally finite, $\gkdim\ip=2$ or $\gkdim \ip=1$
{\rm{(}}and hence $\ip=\Com${\rm{)}}, and
$\ip(1)$ is a finite dimensional augmented semisimple algebra.
\end{enumerate}
\end{theorem}

Note that there are unitary and left (or right) artinian operads that
are not bounded above. Such examples are given in Example
\ref{xxex2.3}(2).

\subsection{Classifications of operads of low Gelfand-Kirillov
dimension}
\label{xxsec0.4}
The third application of truncations concerns classifications
of 2-unitary operads.

The classification of 2-unitary operads of GKdimension 1 is easy.

\begin{proposition}
\label{xxpro0.5} Let $\ip$ be a {\rm{(}}symmetric or plain{\rm{)}}
2-unitary operad. If $\gkdim (\ip)<2$, then $\ip\cong \Com$.
\end{proposition}

A 2-unitary operad consists of a triple
$(\ip, \1_0,\1_2)$ satisfying \eqref{E0.0.1}. A morphism between
two 2-unitary operads means a morphism of operads that preserves
$\1_0$ and $\1_2$. All 2-unitary operads form a category with
morphisms being defined as above.

\begin{theorem}
\label{xxthm0.6}
There are natural equivalences between
\begin{enumerate}
\item[(1)]
the category of finite dimensional, not necessarily unital, $\Bbbk$-algebras;
\item[(2)]
the category of $2$-unitary operads of GKdimension $\leq 2$;
\item[(3)]
the category of $2a$-unitary operads of GKdimension $\leq 2$.
\end{enumerate}
\end{theorem}

At this point we have not found any 2-unitary plain operad of GKdimension two
that is not a symmetric operad. It would be nice to show that every
2-unitary plain operad of GKdimension two is automatically symmetric.

Note that the category in Theorem \ref{xxthm0.6}(1) is equivalently
to the category of finite dimensional unital augmented $\Bbbk$-algebras.
The description of operads in the above theorem is given in Example
\ref{xxex2.3}(1).

For quotient operads of $\As$, we can classify a few more operads with small
GKdimension.

\begin{theorem}
\label{xxthm0.7}
Let $\ip$ be a quotient operad of $\As$ and $n$ be $\gkdim \ip$.
Let ${^k\iu}$ be the truncations of $\As$.
\begin{enumerate}
\item[(1)]{\rm{[Proposition \ref{xxpro0.5}]}}
If $n=1$, $\ip=\As/{^1\iu}\cong\Com$.
\item[(2)]{\rm{[Gap Theorem]}}
$\gkdim \ip$ can not be 2, {\rm{(}}so can not be strictly
between $1$ and $3${\rm{)}}.
\item[(3)]
If $n=3$, then $\ip=\As/{^3\iu}$.
\item[(4)]
If $n=4$, then $\ip=\As/{^4\iu}$.
\item[(5)]
There are at least two non-isomorphic quotient operads $\ip$
such that $\gkdim \ip=5$.
\end{enumerate}
\end{theorem}

\subsection{Other results related to truncations}
\label{xxsec0.5}

We list two other results related to the truncations indirectly.
In Theorem \ref{xxthm0.9}, operads $\ip$ need not be unitary.

Using the Hilbert series of an operad $\ip$, one can define
another numerical invariant, {\it signature} of $\ip$, denoted
by ${\mathcal S}(\ip)$ [Definition \ref{xxdef6.1}].
Let $\OP_{\Com}$ be the category of operads with morphism 
$\Com\to \ip$. (More precisely, objects in $\OP_{\Com}$ are $f: 
\Com \to \ip$ and morphisms are the obvious commutative triangles.) 
Every operad in $\OP_{\Com}$ is canonically 2a-unitary, inherited from
$\Com$. We prove the following

\begin{theorem}
\label{xxthm0.8} Let $\OP_{\Com}$ be defined as above.
\begin{enumerate}
\item[(1)]
For every sequence of non-negative integers ${\bf d}:=\{d_1, d_2, \cdots\}$,
there is an operad $\ip$ in $\OP_{\Com}$ such that
${\mathcal S}(\ip)={\bf d}$.
\item[(2)]
Exponent $\exp$ of \eqref{E0.0.4} is a surjective map from
$\OP_{\Com}$ or from the category
of 2-unitary operads to $\{1\} \cup [2,\infty]$.
\end{enumerate}
\end{theorem}

For a 2-unitary operad $\ip$ with infinite GKdimension,
we can show that $\exp(\ip)\geq 2$. This implies that
there are no 2-unitary operads that have subexponential
growth [Definition \ref{xxdef4.1}(5)]. On the other hand,
there are many unitary operads having subexponential
growth [Example \ref{xxex2.1}(3)]. Theorem \ref{xxthm0.8}(2)
says that $\exp$ of an 2-unitary operad can be any real number
larger than $2$. However, for 2-unitary Hopf operads, we don't
have any example that has non-integer $\exp$.

The next result is a connection between the GKdimension of an operad
and the GKdimension of finitely generated algebras over it.

\begin{theorem}
\label{xxthm0.9}
Let $\ip$ be an operad and $A$ be an algebra over $\ip$.
Suppose $A$ is generated by $g$ elements as an algebra over $\Bbbk$.
Then
\[\gkdim A\leq g-1+\gkdim \ip.\]
\end{theorem}

When $\ip$ is the commutative algebra operad $\Com$, then the above
theorem gives rise to a well-known fact  that the GKdimension
of a commutative algebra $A$ is bounded by the number of generators of
$A$ [Example \ref{xxex5.7}]. Note that every finitely generated PI-algebra
has finite GKdimension, see for instance  \cite{KR,Dr}.

%

The theory of operads provides a unified approach to several
different topics. Operads are also closely related to {\it 
clones} in universal algebra \cite{Sz, Cu} and {\it species} 
in combinatorics \cite{Jo, AM1, AM2}. Some ideas presented 
in this paper can be adapted to study both
clones and species.

The paper is organized as follows. We recall some basic 
concepts in Section \ref{xxsec1}. In Section \ref{xxsec2}, 
we study basic properties of 2-unitary operads, and prove 
some lemmas that are needed in later sections. One of the 
main examples is given in Example \ref{xxex2.3}. Proposition 
\ref{xxpro0.5} is proved in Section \ref{xxsec2}. The main 
object of this paper, the sequence of truncation ideals, is 
defined in Section \ref{xxsec3}. As an application of
truncations, a basis theorem is proved in Section \ref{xxsec4}.
Binomial transform of generating series is defined in Section 
\ref{xxsec5}. Theorems \ref{xxthm0.1}, \ref{xxthm0.9}, 
\ref{xxthm0.3} and Corollary \ref{xxcor0.2} are proved in 
Section \ref{xxsec5}. In Section \ref{xxsec6}, we study the 
signature of an operad. Theorems \ref{xxthm0.6}, \ref{xxthm0.7} 
and \ref{xxthm0.8} are proved in Section \ref{xxsec6}. Theorem 
\ref{xxthm0.4} is proved in Sections \ref{xxsec3} and 
\ref{xxsec6}. In Section \ref{xxsec7} we introduce the notion 
of a truncatified operad. Some basic material is reviewed in
Section \ref{xxsec8} (Appendix).

\section{Preliminaries}
\label{xxsec1}

Throughout let $\k$ be a fixed base field, and all unadorned 
$\ot$ will be $\ot_\k$. In this section, we recall some basic 
facts about operads from standard books such as \cite{LV} and 
\cite{Fr1, Fr2}. Also see Section \ref{xxsec8} for some extra 
material.

\subsection{Operads}
\label{xxsec1.1}
An algebraic structure of a certain type is usually defined by
generating operations and relations, see for instance, the
definition for associative algebras, commutative algebras,
Lie algebras and so on. Given a type of algebras, the set of operations
generated by the ones defining this algebra structure will
form an operad, and an algebra of this type is exactly given
by a set (or a vector space) together with an action of the operad
on it. Roughly speaking, an operad can be viewed as a set of
operations, each of which has a fixed number of inputs and one
output, satisfying a set of compatibility laws.

In this paper we consider operads over $\Bbbk$-vector spaces.
We now recall the \emph{classical definition} of an operad.
Usually the word ``symmetric'' is omitted in this paper.

\begin{definition}
\label{xxdef1.1}
Most of the following definitions are copied from \cite[Chapter 5]{LV}.
\begin{enumerate}
\item[(1)]
A \emph{plain operad} (sometimes called a
\emph{non-$\Sigma$} or \emph{non-symmetric} operad) consists of the
following data:
\begin{enumerate}
\item[(i)]
a sequence $(\ip(n))_{n\geq 0}$ of sets, whose elements are
called \emph{$n$-ary operations},
\item[(ii)]
an element $\1\in \ip(1)$ called the \emph{identity},
\item[(iii)]
for all integers $n\ge 1$, $k_1, \cdots, k_n \ge0$, a
\emph{composition map}
\begin{equation*}
\begin{array}{c}
\circ\colon \ip(n)\times \ip(k_1)\times\cdots \times \ip(k_n)\To \ip(k_1+\cdots +k_n)\\
(\theta,\theta_1,\cdots, \theta_n)\mapsto \theta\circ(\theta_1,\cdots, \theta_n),
\end{array}
\end{equation*}
\end{enumerate}
satisfying the following coherence axioms:
\begin{enumerate}
\item[(OP1)](Identity)
$$\theta\circ(\1,\1,\cdots,\1) = \theta =\1\circ \theta;$$
\item[(OP2)] (Associativity)
\begin{align*}
&\theta\circ(\theta_1\circ(\theta_{1,1},\cdots, \theta_{1,k_1} ),
\cdots, \theta_n\circ(\theta_{n,1},\cdots, \theta_{n,k_n}))\\
 &= (\theta\circ(\theta_1,\cdots,\theta_n))\circ
(\theta_{1,1},\cdots, \theta_{1,k_1},\cdots, \theta_{n,1},
\cdots, \theta_{n,k_n}),
\end{align*}
\end{enumerate}
where in the left hand side,
$\theta_i\circ(\theta_{i,1}, \cdots, \theta_{i,k_i})=\theta_i$
in case $k_i=0$.
\item[(2)]
A plain operad $\ip$ is called an
\emph{operad} (or a \emph{symmetric operad}), if there exists a right action $\ast$ of the
symmetric group $\S_n$ on  $\ip(n)$ for each $n$, satisfying
the following compatibility condition:
\begin{enumerate}
\item[(OP3)](Equivariance)
\begin{align*}
&(\theta\ast\sigma)\circ(\theta_1 \ast \sigma_{1},
\cdots, \theta_{n}\ast \sigma_{n})\\
=&(\theta\circ(\theta_{\sigma^{-1}(1)},\cdots, \theta_{\sigma^{-1}(n)}))
\ast \vartheta_{n; k_1, \cdots, k_{n}}
(\sigma, \sigma_{1}, \cdots, \sigma_{n}),
\end{align*}
where $\vartheta_{n; k_1, \cdots, k_{n}}$ is defined in 
Section \ref{xxsec8}
\end{enumerate}
\item[(3)]
An operad (respectively, a plain operad) is said to be \emph{$\k$-linear}
if $\ip(n)$ is a $\k\S_n$-module (respectively, a $\Bbbk$-module)
for each $n$ and all composition maps are $\k$-multilinear.
\item[(4)]
A $\Bbbk$-linear operad is called {\it unitary}
if $\ip(0)=\Bbbk \1_0\cong \Bbbk$, which is the unit object in
the symmetric monoidal category ${\text{Vect}}_{\Bbbk}$.
Here $\1_0$ is a basis for $\ip(0)$ and is called a \emph{$0$-unit}
of $\ip$.
\item[(5)]
Let $\ip$ be a unitary operad with a fixed $0$-unit $\1_0
\in \ip(0)$. An element $\1_2\in \ip(2)$ is called a {\it right 2-unit} if
\begin{equation}
\label{E1.1.1}\tag{E1.1.1}
\1_2\circ (\1, \1_0)=\1.
\end{equation}
An element $\1_2\in \ip(2)$ is called a {\it left 2-unit} if
\begin{equation}
\label{E1.1.2}\tag{E1.1.2}
\1_2\circ (\1_0,\1)=\1.
\end{equation}
If both \eqref{E1.1.1} and \eqref{E1.1.2} hold for the same $\1_2$,
then it is called a {\it 2-unit}. A unitary operad $\ip$ is called
{\it 2-unitary} (respectively, {\it right 2-unitary}, or
{\it left 2-unitary}) if it has a 2-unit (respectively, right 2-unit,
or left 2-unit) $\1_2$.
\item[(6)]
If $\ip(0)=0$, $\ip$ is called {\it reduced}.
\item[(7)]
If $\ip(1)=\Bbbk$, $\ip$ is called {\it connected}.
\end{enumerate}
\end{definition}

Note that a 2-unit (if exists) may not be unique. For example,
if $\1_2$ is a 2-unit, then so is $\1_2\ast (12)$, where
$(12)$ is the non-identity element in $\S_2$.

Suggested by \eqref{E1.1.1}-\eqref{E1.1.2}, sometimes
we denote $\1$ by $\1_1$.
It is easy to see that \eqref{E1.1.1} implies that
\begin{equation}
\label{E1.1.3}\tag{E1.1.3}
\1_2\circ (\theta, \1_0)=\theta
\end{equation}
for all $\theta\in \ip(n)$ and that \eqref{E1.1.2} implies that
\begin{equation}
\label{E1.1.4}\tag{E1.1.4}
\1_2\circ (\1_0, \theta)=\theta
\end{equation}
for all $\theta\in \ip(n)$.

Unless otherwise stated, all operads considered here will be
$\k$-linear. In some occasions, it will be more convenient
to use another definition, called the {\it partial definition}
of an operad.

\begin{definition}[{\cite[Section 2.1]{Fr1}, \cite[Section 5.3.4]{LV}}]
\label{xxdef1.2}
An \emph{operad} consists of the following data:
\begin{enumerate}
\item[(i)]
a sequence $(\ip(n))_{n\geq 0}$ of right $\Bbbk\S_n$-modules, whose elements are
called \emph{$n$-ary operations},
\item[(ii)]
an element $\1\in \ip(1)$ called the \emph{identity},
\item[(iii)]
for all integers $m\ge 1$, $n \ge0$, and $1\le i\le m$, a
\emph{partial composition map}
\[-\underset{i}{\circ}-\colon \ip(m) \otimes \ip(n) \to \ip(m+n-1) \ \ (1\le i\le m), \]
\end{enumerate}
satisfying the following axioms:
\begin{enumerate}
\item[(OP1$'$)] (Identity)

for $\theta\in \ip(n)$ and $1\leq i\leq n$,
\[\theta\underset{i}{\circ} \1 = \theta =\1\underset{1}{\circ} \theta;
\]

\item[(OP2$'$)] (Associativity)

for $\lambda \in \ip(l)$, $\mu\in \ip(m)$ and $\nu\in \ip(n)$,
\[\begin{cases}
(\la  \underset{i}{\circ} \mu) \underset{i-1+j}{\circ} \nu
=\la \underset{i}{\circ} (\mu \underset{j}{\circ} \nu),
& 1\le i\le l, 1\le j\le m,\\
(\la  \underset{i}{\circ} \mu) \underset{k-1+m}{\circ} \nu
=(\la \underset{k}{\circ}\nu) \underset{i}{\circ} \mu,
& 1\le i<k\le l;
\end{cases}
\]

\item[(OP3$'$)] (Equivariance)

for $\mu\in \ip(m)$, $\phi\in \S_m$, $\nu\in \ip(n)$ and $\sigma\in \S_{n}$,
\[
\begin{cases}
\mu  \underset{i}{\circ} (\nu \ast \sigma)=
&(\mu  \underset{i}{\circ} \nu)\ast \sigma',\\
(\mu\ast \sigma)  \underset{i}{\circ} \nu=
&(\mu  \underset{\sigma(i)}{\circ} \nu)\ast \sigma'',
\end{cases}
\]
where
\begin{equation}\label{E1.2.1}\tag{E1.2.1}
\begin{split}
\sigma'= \vartheta_{m; 1, \cdots, 1, \underset{i}{n}, 1, \cdots, 1}
(1_m, 1_1, \cdots, 1, \underset{i}{\sigma}, 1_1, \cdots, 1_1),\\
\sigma''= \vartheta_{m; 1, \cdots, 1, \underset{i}{n}, 1, \cdots, 1}
(\sigma, 1_1, \cdots,1_1, \underset{i}{1_n}, 1_1 \cdots,  1_1).
\end{split}
\end{equation}
(see \eqref{E8.1.3} for the definition of 
$\vartheta_{m; 1, \cdots, 1, \underset{i}{n}, 1, \cdots, 1}$).
\end{enumerate}
\end{definition}

\begin{remark}
\label{xxrem1.3}
The above two definitions for operads are equivalent
by \cite[Proposition 5.3.4]{LV}.
Let $\ip$ be an operad in the sense of Definition \ref{xxdef1.1}.
Then the partial compositions
\[-\underset{i}{\circ}-\colon \ip(m) \otimes \ip(n) \to
\ip(m+n-1) \ \ (1\le i\le m) \]
associated to $\ip$ are defined by
\[\mu \underset{i}{\circ} \nu=\mu \circ (\1_1, \cdots, \1_1,
\underset{i}{\nu}, \1_1, \cdots, \1_1).\]
Conversely, let $\ip$ be an operad in the sense of Definition
\ref{xxdef1.2}, then one can define composition maps by
\[\theta \circ (\theta_1, \cdots, \theta_n)=
(\cdots ((\theta \underset{n}{\circ} \theta_n)
\underset{n-1}{\circ} \theta_{n-1})\underset{n-2}{\circ}
\theta_{n-2}\cdots) \underset{1}{\circ}\theta_1.\]
One can show that the axioms (OP1)-(OP3) are equivalent to
the axioms (OP1$'$)-(OP3$'$) respectively.
\end{remark}

We will use the  partial definition in several examples in later
sections.

\begin{example}\cite[Section 5.2.11]{LV}
\label{xxex1.4}
For every $\k$-vector space $V$, the sequence $(\End_V(n))_{n\ge 0}$
together with the composition map defined as in \eqref{E8.1.6}
gives rise to an operad, which is denoted by $\End_V$. We call $\End_V$
the \emph{endomorphism operad} of $V$. It is easy to see that $\End_V$
is not unitary unless $V=\Bbbk$.

If ${\mathcal T}$ is a $\Bbbk$-linear symmetric monoidal category
with internal hom-bifunctor
$$\Hom_{\mathcal T}(-,-):
{\mathcal T}^{op}\times {\mathcal T}\to {\mathcal T},$$
then endomorphism operad $\End_V$ can be defined for any object
$V\in {\mathcal T}$. Some results in this paper can be extended from
${\text{Vect}}_{\k}$ to ${\mathcal T}$.
\end{example}

\subsection{Algebras and free algebras over an operad}
\label{xxsec1.2}
Given a type of algebras, there is a notion of ``free" algebras, which
can be constructed by using the associated operad.

\begin{definition}
\cite[Sections 5.2.1 and 5.2.3]{LV}
\label{xxdef1.5}
\begin{enumerate}
\item[(1)]
Let $\ip, \ip'$ be ($\k$-linear) operads. A \emph{morphism} from $\ip$ to
$\ip'$ is a sequence of $\S_n$-morphism
$\gamma=(\gamma_n\colon \ip(n)\to \ip'(n))_{n\geq 0}$,
satisfying
\[\gamma(\1_1)=\1'_1\]
where $\1_1$ and $\1'_1$ are identities of $\ip$ and $\ip'$, respectively, and
$$\gamma(\theta\circ(\theta_1,\cdots,\theta_n)) =
\gamma(\theta)\circ(\gamma(\theta_1),\cdots,\gamma(\theta_n))$$
for all $\theta, \theta_1,\cdots \theta_n$.
\item[(2)]
An \emph{algebra over $\ip$},
or a \emph{$\ip$-algebra} for short, is a $\k$-vector space $A$ equipped with
a morphism $\gamma\colon \ip\to \End_A$. Also see \cite[Proposition 1.1.15]{Fr1}.
\end{enumerate}
\end{definition}

Let $\ip$ be an operad and $V$ a $\k$-vector space. Set
$$\ip(V)_n= \ip(n)\ot_{\k\S_n} V^{\ot n},\qquad \ip(V)= \bigoplus_{n\ge 0}\ip(V)_n$$
where a pure tensor $\theta\ot x_1\ot \cdots \ot x_n$ in
$\ip(n)\ot_{\k\S_n} V^{\ot n}$ is denoted by
$[\theta, x_1,\cdots,x_n]$. Then we have
\begin{align*}\ip(V)^{\ot n}&= \bigoplus_{m\ge0}\bigoplus_{k_1+\cdots+ k_n =m}
\ip(V)_{k_1}\ot\cdots\ot \ip(V)_{k_n}.
\end{align*}
The composition in $\ip$ gives a linear map
\begin{equation*}
\begin{gathered}
\gamma_n\colon \ip(n)\to \Hom_\k\left(\bigoplus_{k_1+\cdots+ k_n =m}
\ip(V)_{k_1}\ot\cdots\ot \ip(V)_{k_n}, \ip(V)_m\right)\\
\begin{aligned}
\gamma_n(\theta)([\theta_1, x_{1,1}, \cdots, x_{1,k_1}]&\ot\cdots
\ot
[\theta_n, x_{n,1}, \cdots, x_{n,k_n}]) \\
&=[\theta\circ(\theta_1,\cdots,\theta_n),x_{1,1}, \cdots, x_{1,k_1},
\cdots, x_{n,1}, \cdots, x_{n,k_n}],
\end{aligned}
\end{gathered}
\end{equation*}
which extends to a linear map $\gamma_n\colon \ip(n)\to \End_{\ip(V)}(n)$.
One can check that $\gamma_n$ is well defined and the sequence
$\gamma=(\gamma_n)_{n\ge0}$ is a morphism of operads, i.e.,  $\ip(V)$
is a $\ip$-algebra. We mention that $\ip(V)$ is a \emph{free $\ip$-algebra}
in the following sense.

\begin{proposition} \cite[Proposition 5.2.1]{LV}
\label{xxpro1.6}
Let $A$ be a $\ip$-algebra, and $V$ a $\k$-vector space. Then every linear map
$f\colon V\to A$ extends uniquely to a morphism $f\colon \ip(V) \to A$ of
$\ip$-algebras.
\end{proposition}

\begin{remark}
\label{xxrem1.7}
The above proposition can be restated as follows. Given an operad $\ip$, the functor
$V\mapsto \ip(V)$ is a left adjoint to the forgetful functor from the category of
$\ip$-algebras to the category of $\k$-vector spaces ${\text{Vect}}_{\k}$.
\end{remark}

\subsection{Operadic ideals and quotient operads}
\label{xxsec1.3}
We denote by $\S$ the disjoint union of all $\S_n$, $n\ge0$.
We call a family
\[\M =(\M(0), \M(1), \cdots, \M(n),\cdots)\]
of right $\k\S_n$-modules $\M(n)$ a (right) \emph{$\S$-module over $\k$}.
Thus a $\k$-linear operad is an $\S$-module over $\k$ equipped with a family
of suitable composition maps.

An $\S$-submodule $\mathcal{N}$ of $\M$ is a sequence
$\mathcal{N}=(\mathcal{N}(n))_{n\ge0}$, where each $\mathcal{N}(n)$
is an $\S_n$-submodule of $\M(n)$. Given $\M, \mathcal{N}$, one defines
the quotient $\S$-module $\M/\mathcal{N}$ by setting
$(\M/\mathcal{N})(n) = \M(n)/\mathcal{N}(n)$.

\begin{definition}
\label{xxdef1.8}
Let $\ip$ be an operad and $\I$ is a $\S$-submodule of $\ip$.
\begin{enumerate}
\item[(1)]\cite[Section 5.2.14]{LV}.
We call $\I$ an \emph{operadic ideal} (or simply {\it ideal}) of $\ip$
if the operad structure on $\ip$ passes to $\ip/\I$. In this case,
$\ip/\I$ is called a \emph{quotient operad } of $\ip$. More explicitly,
$\I$ is an ideal if and only if
\[\I(n)\circ (\ip(k_1),\cdots, \ip(k_n))\subseteq \I(k_1+\cdots+k_n)\]
and
\[\ip(n)\circ (\ip(k_1),\cdots, \ip(k_{s-1}), \I(k_s),\ip(k_{s+1}),\cdots,
\ip(k_n))\subseteq \I(k_1+\cdots+k_n)\]
for all $n>0$, $k_1,\cdots, k_n\ge 0$. In other words, for any family of
operations $\theta, \theta_1,\cdots,\theta_n$,  if one of them is in $\I$,
then so is $\theta\circ(\theta_1,\cdots,\theta_n)$.
\item[(2)]
An $\S$-submodule $\I$ of $\ip$ is called a {\it right
ideal} of $\ip$, if for every $\lambda \in \I(m)$ and $\mu\in \ip(n)$,
$\lambda \underset{i}{\circ} \mu\in \I(m+n-1)$ for every
$1\leq i\leq m$. We say $\ip$ is {\it right artinian} if the
set of right ideals of $\ip$ satisfies the descending chain condition.
\item[(3)]
An $\S$-submodule $\I$ of $\ip$ is called a {\it left
ideal} of $\ip$, if for every $\lambda \in \ip(m)$ and $\mu\in \I(n)$,
$\lambda \underset{i}{\circ} \mu\in \I(m+n-1)$ for every
$1\leq i\leq m$. We say $\ip$ is {\it left artinian} if the
set of left ideals of $\ip$ satisfies the descending chain condition.
\end{enumerate}
It is easy to see that $\I$ is an ideal if and only if it is both
a left and a right ideal.
\end{definition}

Let $\{\I^j\}_{j\in J}$ be a family of ideals of $\ip$. Let
$\sum_{j\in J} \I^j$ and $\bigcap_{j\in J}\I^j$ be the $\S$-modules
given by
\[(\sum\nolimits_{j\in J} \I^j)(n)=\sum\nolimits_{j\in J} \I^j(n), \qquad
(\bigcap\nolimits_{j\in J} \I^j)(n)= \bigcap\nolimits_{j\in J} \I^j(n)\]
for all $n\ge0$. The following lemmas are easy and their proofs are
omitted.

\begin{lemma}
\label{xxlem1.9}
Let $\{\I^j\}_{j\in J}$ be a family of ideals {\rm{(}}respectively,
left or right ideals{\rm{)}} of an operad $\ip$. Then both
$\sum_{j\in J} \I^j$ and $\bigcap_{j\in J}\I^j$ are ideals
{\rm{(}}respectively, left or right ideals{\rm{)}} of $\ip$.
\end{lemma}


Let $\I$ and $\mathcal{J}$ be $\S$-submodules (or ideals) of $\ip$. The product
$\I \mathcal{J}$ is defined to be the $\S$-submodule of $\ip$ generated by
elements of the form
$\mu \underset{i}{\circ} \nu$
for all possible $\mu\in \I(m)$, $\nu\in \mathcal{J}(n)$ and $1\leq i\leq m$.

\begin{definition}
\label{xxdef1.10} Let $\ip$ be an operad.
\begin{enumerate}
\item[(1)]
Let $X$ be a property that is defined on operads (or a class of operads).
We define $X$-radical of $\ip$ to be
$$X\rad(\ip):=\bigcap\{ \I \mid \ip/\I {\text{ has property $X$}}\}.$$
\item[(2)]
For example, if $(GK\leq k)$ denotes the property that the $\gkdim$
of $\ip$ is no more than $k$, then
$$(GK\leq k)\rad(\ip):=\bigcap\{ \I \mid \gkdim(\ip/\I)\leq k \}.$$
\item[(3)]
We say $\ip$ is {\it semiprime} if $\ip$ does not contain
an ideal $\mathcal{N}\neq 0$ such that $\mathcal{N}^2=0$.
\item[(4)]
If $p.$ denotes the property of $\ip$ being semiprime, then
$$p.\rad(\ip):=\bigcap\{ \I \mid {\text{$\ip/\I$ does not contain
an ideal $\mathcal{N}\neq 0$ such that $\mathcal{N}^2=0$}}.\}$$
\end{enumerate}
\end{definition}

\begin{lemma}
\label{xxlem1.11} Let $\I$ and $\M$ be $\S$-submodules of
an operad $\ip$.
\begin{enumerate}
\item[(1)]
If $\I$ and $\mathcal{J}$ are right ideals of $\ip$, then so is $\I\mathcal{J}$.
\item[(2)]
If $\I$ is a left ideal of $\ip$, then so is $\I\mathcal{J}$.
\item[(3)]
If $\I$ is an ideal of $\ip$ and $\mathcal{J}$ is a right ideal of $\ip$,
then $\I\mathcal{J}$ is an ideal of $\ip$.
\end{enumerate}
\end{lemma}

We conclude this section with the following fact. Recall that $\Com$ denotes the operad
that encodes the category of unital commutative algebras, namely,
$\Com(n)=\Bbbk$ for all $n\geq 0$. Let $\Uni$ be the trivial unitary operad defined by
$$\Uni(n)=\begin{cases}
\Bbbk=\Bbbk\1_0, & n=0,\\
\Bbbk=\Bbbk\1_1, & n=1,\\
0, & n\geq 2.
\end{cases}$$

\begin{lemma}\label{xxlem1.12}
\begin{enumerate}
\item[(1)]\cite[Proposition 2.2.21]{Fr1}
The operad $\Com$ is the terminal object in the category of
unitary operads.
\item[(2)]
The operad $\Uni$ is the initial object in the category of
unitary operads.
\end{enumerate}
\end{lemma}

\section{Unitary and 2-unitary operads}
\label{xxsec2}

Recall from Definition \ref{xxdef1.1} that an operad $\ip$ is 
unitary if $\ip(0)\cong \Bbbk$, and a unitary operad $\ip$ is
2-unitary if there is a 2-unit $\1_2\in \ip(2)$ such that
$$\1_2\circ (\1_1, \1_0)=\1_1=\1_2\circ(\1_0, \1_1),$$
or equivalently, for any $\theta\in \ip(n)\ \ (n\ge 0)$,
\[\1_2\circ (\theta, \1_0)=\theta=\1_2\circ (\1_0, \theta).\]

\subsection{Examples of 2-unitary operads}
\label{xxsec2.1}
\begin{example}
\label{xxex2.1}
Parts (1) and (2) are examples of 2-unitary operads and part 
(3) is an example of unitary operads.
\begin{enumerate}
\item[(1)]
There are many commonly-used 2-unitary operads from textbooks, such as unitary
operads $\As$ and $\Com$, the unitary $A_{\infty}$-algebra operad
(denoted by ${\mathcal A}_{\infty}$),
the Poisson operad (denoted by $\mathcal{P}ois$), the unitary differential
graded algebra operad.
\item[(2)]
One can easily show that every quotient operad of a 2-unitary
operad is again 2-unitary.
\item[(3)]
Let $\M$ be an $\S$-module with $\M(0)=0$. Then $\Uni\oplus \M$ is an unitary
operad with partial composition defined by
\begin{align*}
\1_1 \underset{1}{\circ} \theta = \theta= \theta \underset{i}{\circ}\1_1,
\quad & \quad \forall \theta \in \M,\\
\theta \underset{i}{\circ} \1_0 =0, \quad & \quad \forall \theta \in \M,\\
\theta_1 \underset{i}{\circ} \theta_2 =0, \quad & \quad \forall \theta_1, \theta_2 \in \M.
\end{align*}
One can use the partial definition to check that this operad is unitary, but not
2-unitary.
\end{enumerate}
\end{example}

Of course, any non-unitary operads can not be 2-unitary. In the rest of this
subsection we give some examples of 2-unitary operads different from ones in
Example \ref{xxex2.1}. The following lemma is easy to prove.

\begin{lemma}
\label{xxlem2.2}
Let $\ip$ and $\iq$ be unitary operads.
\begin{enumerate}
\item[(1)]
If $\ip$ and $\iq$ are 2-unitary, then so is the {\it Hadamard product}
\cite[Section 5.3.2]{LV} {\rm{(}}also called {\it Segre product} or
{\it white product}{\rm{)}} of $\ip$ and $\iq$. In fact, the $2$-unit in
$\ip\underset{{\rm H}}\otimes \iq$ is just $\1_2^\ip \otimes \1_2^\iq$,
where $\1_2^\ip$ and $\1_2^\iq$ are $2$-units in $\ip$ and $\iq$,
respectively.
\item[(2)]
Suppose $\ip$ is 2-unitary with 2-unit $\1_2^{\ip}$ and $f: \ip\to \iq$
is a morphism of unitary operads. Then $\iq$ is 2-unitary with
2-unit $f(\1_2^{\ip})$.
\end{enumerate}
\end{lemma}

The next example will be used in the classification of 2-unitary
operads of GKdimension two.

\begin{example}
\label{xxex2.3}
Let $A=\Bbbk \1_1\oplus \bar A$ be an augmented algebra with augmentation
ideal $\bar A$. Let $\{\delta_i\mid i\in T\}$ be a $\Bbbk$-basis
for $\bar A$ where $T$ is an index set, and $\{\Omega_{kl}^v\mid k,l,v\in T\}$
the corresponding structural constants, namely,
\begin{equation}
\label{E2.3.1}\tag{E2.3.1}
\delta_i\delta_j=\sum_{k\in T}\Omega_{ij}^k\delta_k
\end{equation}
for all $i,j\in T$.  We assume that $0$ is not in $T$.

\begin{enumerate}
\item[(1)]
We define a 2-unitary operad $\mathcal{D}$ as follows. Set
$\mathcal{D}(0)=\Bbbk \1_0\cong \Bbbk$, $\mathcal{D}(1)=A=\Bbbk \1_1\oplus {\bar A}$,
and
\begin{equation}
\label{E2.3.2}\tag{E2.3.2}
\mathcal{D}(n)=\Bbbk \1_n \oplus \bigoplus
\limits_{i\in [n], j\in T}
\Bbbk \delta^n_{(i)j}
\end{equation}
for $n\ge 2$. For consistency of notations, we set $\delta^1_{(1)j}=\delta_j$
for each $j\in T$, and $\delta^n_{(i)0}=\1_n$ for all $i\in [n]$.

The action of $\S_n$ on $\mathcal{D}(n)$ is given by
$\1_n\ast \sigma = \1_n$ and
$\delta^n_{(i)j}\ast \sigma =  \delta^n_{(\sigma^{-1}(i))j}$
for all $\sigma\in \S_n$ and all $n$.

We use the partial
definition of an operad [Definition \ref{xxdef1.2}]. The
partial composition
\[-\underset{i}{\circ}-\colon \mathcal{D}(m) \otimes \mathcal{D}(n)
\to \mathcal{D}(m+n-1)\ \  (i\in [m])\]
is defined by
\begin{equation}
\label{E2.3.3}\tag{E2.3.3}
\delta^m_{(s)t}\underset{i}{\circ} \delta_{(k)l}^n
=\begin{cases}
\delta_{(k+i-1) l}^{m+n-1}, & t=0, l\ge 0, \\
\delta_{(s)t}^{m+n-1}, & t\ge 1, l=0, 1\le s\le i-1,\\
\sum\limits_{h=i}^{i+n-1} \delta_{(h)t}^{m+n-1},
& t\ge 1, l=0, s=i,\\
\delta_{(s+n-1) t}^{m+n-1}, & t\ge 1, l=0, i<s\le m,\\
\sum\limits_{v\in T}\Omega^v_{tl}\delta^{m+n-1}_{(i+k-1)v},
& t\ge 1, l\ge 1, s=i,\\
0, & t\ge 1, l\ge 1, s\neq i
\end{cases}
\end{equation}
for all $n\ge 1$, and $\1_1\underset{1}{\circ} \1_0=\1_0$,
$\delta_j\underset{1}{\circ} \1_0=0$ for all $j\in T$. If we separate
$\1_m$ from elements of the form $\delta^m_{(k)l}$ for $k\in [m]$
and $0\neq l\in T$, it
is easy to see that \eqref{E2.3.3} is equivalent to
$$\begin{aligned}
\1_{m} \underset{i}{\circ} \1_{n}& = \1_{m+n-1},\\
\1_{m} \underset{i}{\circ} \delta_{(k)l}^n&= \delta_{(k+i-1)l}^{m+n-1},\\
\delta^m_{(s)t} \underset{i}{\circ} \1_n&=
\begin{cases} \delta^{m+n-1}_{(s)t}, & \quad 1\leq s \leq i-1,\\
              \sum\limits_{h=i}^{i+n-1} \delta_{(h)t}^{m+n-1}, & \quad s=i,\\
              \delta_{(s+n-1) t}^{m+n-1}, & \quad i<s\le m,
\end{cases}\\
\delta^m_{(s)t}\underset{i}{\circ} \delta_{(k)l}^n&=
\begin{cases}
              \sum\limits_{v\in T}\Omega^v_{tl}\delta^{m+n-1}_{(i+k-1)v}, & s=i,\\
              0, & s\neq i.
\end{cases}
\end{aligned}
$$

Note that $-\underset{1}{\circ}-$ in $A$ is just the associative
multiplication of $A$. By the second relation on the above list,
we obtain
$$\delta_{(i)j}^n=\1_n\underset{i}{\circ} \delta_j$$
for all $i\in [n], j\in T$. One can now  directly
check via a tedious computation that $\mathcal{D}$ is a 2-unitary 
operad by partial composition.

An algebra $A$ over $\mathcal{D}$ means a unital commutative associative
algebra together with a set of derivations $\{\delta_i\}_{i\in T}$ satisfying
\begin{enumerate}
\item[(1)]
$\delta_i(x)\delta_j(y)=0$ for all $i, j\in T$ and all $x, y\in A$,
and
\item[(2)]
\eqref{E2.3.1}:
$\delta_i\delta_j=\sum_{k\in T}\Omega_{ij}^k\delta_k$.
\end{enumerate}
Note that a $\mathcal{D}$-algebra is a special kind of commutative differential
$\Bbbk$-algebra. Similar algebras have been studied by Goodearl in \cite[Section 1]{Go}.

A $\Bbbk$-linear basis of ${\mathcal D}$ is explicitly given in \eqref{E2.3.2}.
When $T$ is a finite set with $d$ elements, the generating function of
${\mathcal D}$ is
$$G_{\mathcal D}(t)=\sum_{n=0}^{\infty} (1+dn)t^n=\frac{1}{1-t}+\frac{d}{(1-t)^{2}}.$$
As a consequence, ${\mathcal D}$ has GKdimension two. We will see later that
every 2-unitary operad of GKdimension two is of this form.

\item[(2)]
Let $I:=\{\I_\alpha\}_{\alpha\geq 2}$ be a descending chain of
ideals of $A$ inside ${\bar{A}}$ such that $\I_{\alpha}\I_{\beta}
\subseteq \I_{\alpha+\beta-1}$ for all $\alpha$ and $\beta$.
We define a unitary operad, denoted by
${\mathcal D}^{I}$, associated $I$. For the sake
of using $\Bbbk$-linear bases, suppose we can choose
a descending chain of subsets $\{T_{\alpha}\}$ of $T$ such that
$\{\delta_i \mid i\in T_{\alpha}\}$ is a $\Bbbk$-linear basis of $\I_{\alpha}$
(this is not essential). Define
$${\mathcal D}^{I}(n)=
\begin{cases}
\Bbbk \1_0, & n=0,\\
A =\Bbbk \1_1\oplus \bigoplus_{j\in T} \Bbbk \delta_j,& n=1,\\
\bigoplus_{i\in \n, j\in T_n} \Bbbk \delta^{n}_{(i)j},& n\geq 2.
\end{cases}$$
One can check that ${\mathcal D}^{I}$ is a unitary, but
not 2-unitary, suboperad of ${\mathcal D}$.

A special case is when $\I_{\alpha}=\I$ for all $\alpha\geq 2$.
In this case, the above defined operad is denoted by ${\mathcal D}^{\I}$.
Suppose $T'$ is a subset of $T$ such that
$\{\delta_i \mid i\in T'\}$ is a $\Bbbk$-linear basis of $\I$.
Then
$${\mathcal D}^\I(n)=\begin{cases}
\Bbbk \1_0, & n=0,\\
A, & n=1,\\
\bigoplus_{i\in \n, j\in T'} \Bbbk \delta^{n}_{(i)j}, & n\geq 2.
\end{cases}$$
\end{enumerate}
\end{example}

\subsection{Some elementary operators on 2-unitary operads}
\label{xxsec2.2}
Let $s$ be an integer no more than $n$, and $I\subseteq \n$ a
subset consisting of $s$ elements. Clearly, there exists a unique
1-1 correspondence from $[s]$ to $I$
that preserves the ordering. Choosing $I\subset \n$ is equivalent
to giving an order preserving map
$$\overrightarrow{I}: [s]\longrightarrow I\subseteq \n.$$
Let $\chi_{I}$ be the characteristic function
of $I$, i.e. $\chi_{I}(x)=1$ for $x\in I$ and $\chi_{I}(x)=0$ otherwise.

We recall the following useful operators. Let $\ip$ be a
(2-)unitary operad. Consider the following {\it restriction operator}
\cite[Section 2.2.1]{Fr1}
\begin{equation}
\label{E2.3.4}\tag{E2.3.4}
\pi^{I}\colon \ip(n)\to\ip(s),\qquad \pi^{I}(\theta)=
\theta\circ(\1_{\chi_{I}(1)}, \cdots, \1_{\chi_{I}(n)})
\end{equation}
for all $\theta\in \ip(n)$. The {\it contraction operator} is defined
by $\Gamma^I=\pi^{\hat{I}}$ where $\hat{I}$ is the complement of $I$
in $\n$, or
\begin{equation}
\label{E2.3.5}\tag{E2.3.5}
\Gamma^{I}\colon \ip(n)\to \ip(n-s),\qquad \Gamma^{I}(\theta)=
\theta\circ(\1_{\chi_{{\hat I}}(1)}, \cdots, \1_{\chi_{{\hat I}}(n)})
\end{equation}
for all $\theta\in \ip(n)$.

Recall that $\bullet$ denotes the usual composition of two functions
that is omitted sometimes.

\begin{lemma}\cite[Lemma 2.2.4(1)]{Fr1}
\label{xxlem2.4}
Retain the above notation.
\begin{enumerate}
\item[(1)]
Let $I\subseteq \n$ with $|I|=s$ and $J\subseteq [s]$. Let
$\widetilde{J}:= \overrightarrow{I}(J)$ be the image of $J$ under
$\overrightarrow{I}$. Then $\pi^{\widetilde{J}}=\pi^J\bullet \pi^I$.
\item[(2)]
For each $W\subseteq \hat{I}$, $\pi^I=\Gamma^{\hat{I}}= \Gamma^{W'}
\bullet \Gamma^W$ for some subset $W'$ of $[n-|W|]$ with
$|W'|+|W| = n-|I|$.
\item[(3)]
If $k\not\in I$, then $\pi^I=\pi^{I'}\bullet \Gamma^k$ for some
$I'\subseteq [n-1]$ with $|I'| = |I|$.
\end{enumerate}
\end{lemma}

\begin{proof}
(1) This is \cite[Lemma 2.2.4(1)]{Fr1}. It follows from (OP2).

(2, 3) Easy consequences of part (1).
\end{proof}

If $\ip$ is  2-unitary, we can define another
operator as follows. The {\it extension operator}
$\D_{_I}\colon \ip(n) \to \ip(n+s)$ is defined by
\[\D_{_I}(\theta)= \theta\circ(\1_{\chi_{_I}(1)+1}, \cdots, \1_{\chi_{_I}(n)+1})\]
for all $\theta\in \ip(n)$. If $I=\{i_1,\cdots, i_s\}$ with
$i_1<i_2<\cdots<i_s$, then we also write $\pi^I $, $\Gamma^I$
and $\D_{_I}$ as $\pi^{i_1,\cdots, i_s}$, $\Gamma^{i_1,\cdots, i_s}$
and $\D_{i_1,\cdots, i_s}$ respectively.

Assume that $\ip$ is 2-unitary. For every $n\geq 3$, we define
inductively that
\begin{equation}
\label{E2.4.1}\tag{E2.4.1}
\1_{n}=\1_{2}\circ (\1_{n-1}, \1_1).
\end{equation}
Note that one might also define inductively
\begin{equation}
\label{E2.4.2}\tag{E2.4.2}
\1'_n=\1_2\circ (\1_1, \1'_{n-1})
\end{equation}
for all $n\geq 3$. By convention, $\1_{n}=\1'_{n}$ for $n=0,1,2$.
Unless $\ip$ is a quotient operad of $\As$,
it is not automatic that $\1'_n=\1_n$ for any $n\geq 3$.
In fact, $\1_3=\1'_3$ means that the binary operation
given by $\1_2$ is associative.

\begin{definition}
\label{xxdef2.5}
Let $\ip$ and $\iq$ be operads.
\begin{enumerate}
\item[(1)]
$\ip$ is called {\it $2a$-unitary} if $\ip$ is 2-unitary and
the 2-unit $\1_2\in \ip(2)$ is associative, or equivalently,
$\1_3=\1'_3$.
\item[(2)]
Let $\iq$ be a (unitary) operad.
We call $\ip$ {\it  $\iq$-augmented} if there are morphisms of operads
$f: \iq\to \ip$ and $g: \ip\to \iq$ such that $gf=\Id_{\iq}$.
\item[(3)]
$\ip$ is called {\it $\Com$-augmented} if there is a morphism from
$\Com\to \ip$.
\end{enumerate}
\end{definition}

It is easy to see that $\Com$-augmented operads are $2a$-unitary.
The 2a-unitary property in the above definition may be dependent on
choices of $\1_2$. For example, if $(\1_0, \1_1,\1_2)=(1_0, 1_1, 1_2)$
as elements in $\S_n$ for $n=0,1,2$, then $(\As, \1_0, \1_1,\1_2)$ is
a 2a-unitary operad. Suppose ${\rm{char}}\; \Bbbk\neq 2$.
If we set $(\1_0, \1_1,\1_2)=(1_0, 1_1, \frac{1}{2}(1_2+1_2\ast (12)))$,
$(\As, \1_0, \1_1,\1_2)$ is only 2-unitary, but not 2a-unitary.

\begin{lemma}
\label{xxlem2.6}
Let $\mathcal{P}$ be a $2a$-unitary operad, namely, $\1_3=\1'_3$.
Then the following hold.
\begin{enumerate}
\item[$(1)$]
For every $n\ge 3$, $\1_n=\1'_n$.
\item[$(2)$]
For every $n\ge 1$ and $k_1, \cdots, k_n\ge 0$,
$\1_n\circ (\1_{k_1}, \cdots, \1_{k_n})=\1_{k_1+\cdots+k_n}$.
\item[$(3)$]
There exists an operad morphism $\gamma\colon \As\to \ip$. As
a consequence, an algebra over $\ip$ has an associative algebra
structure.
\end{enumerate}
\end{lemma}

\begin{proof}
(1) Use induction on $n$. Assume that $\1_k=\1'_k$ for all
$3\le k\le n-1$. Then
\begin{align*}
\1_n=& \1_2\circ (\1_{n-1}, \1_1)
    =\1_2\circ (\1'_{n-1}, \1_1)
		=\1_2\circ (\1_2\circ (\1_1, \1'_{n-2}), \1_1\circ \1_1)\\
=& (\1_2\circ (\1_2, \1_1))\circ (\1_1, \1'_{n-2}, \1_1)
    =(\1_2\circ (\1_1, \1_2))\circ (\1_1, \1_{n-2}, \1_1)\\
=& \1_2\circ (\1_1, \1_2\circ (\1_{n-2}, \1_1))
    =\1_2\circ (\1_1, \1_{n-1})=\1_2\circ (\1_1, \1'_{n-1})\\
=& \1'_n.
\end{align*}

(2) This follows from induction.

(3) For any $\sigma\in \S_n$, we define
$\gamma(\sigma)=\1_n\ast \sigma$. Clearly,
for all $\sigma\in \S_n, \sigma_i\in \S_{k_i},
n\ge 1, k_i\ge 0, i=1, \cdots, n$, we have
\begin{align*}
\gamma(\sigma\circ (\sigma_1, \cdots, \sigma_n))
&= \gamma(\vartheta(\sigma, \sigma_1, \cdots, \sigma_n))\\
&= \1_{k_1+\cdots+k_n}\ast \vartheta(\sigma, \sigma_1, \cdots, \sigma_n)\\
&= (\1_n\ast \sigma)\circ (\1_{k_1}\ast \sigma_1, \cdots, \1_{k_n}\ast \sigma_n)\\
&= \gamma(\sigma)\circ (\gamma(\sigma_1), \cdots, \gamma(\sigma_n)).
\end{align*}
\end{proof}

By Lemma \ref{xxlem2.6}(3), $\As$ is the initial object in
the category of $2a$-unitary operads. It is easy to see that
Lemma \ref{xxlem2.2} holds for $2a$-unitary operads. If we consider
2-unitary operads that are not necessarily 2a-unitary, then the unitary
magmatic operad \cite[Section 13.8.2]{LV} is the initial object
in the category of 2-unitary operads. The unital magmatic operad encodes the
category of unital magmatic algebras \cite[Section 13.8.1]{LV}, namely,
unital non-associative algebras.

For any $l, r\ge 0$, we define the function
$\iota^l_r\colon \ip(n)\to \ip(l+n+r)$ by
\[\iota^l_r(\theta) = \1_3\circ(\1_l, \theta, \1_r).\]
We simply write $\iota_r=\iota^0_r$ and $\iota^l=\iota^l_0$.

\begin{lemma}
\label{xxlem2.7}
Retain the above notation. Let $\ip$ be a 2-unitary operad and let
$\theta\in \ip(n)$.
\begin{enumerate}
\item[(1)] $\pi^{I}(\1_n)=\1_{|I|}$ for all $I\subseteq \n$.
\item[(2)]
$\iota_r(\theta)=\1_2\circ (\theta, \1_r)$.
\item[(3)]
$\iota^l(\theta)=\1_2\circ (\1_l, \theta)$.
\item[(4)]
$\iota^l_r=\iota_r\bullet \iota^l$.
Moreover, $\iota_r^l=\iota^l \bullet \iota_r$ for all $l,r\geq 0$
if and only if $\ip$ is $2a$-unitary.
\end{enumerate}
\end{lemma}

\begin{proof} (1) This follows by induction on $n$.

(2) We compute
\begin{align*}
\iota_r(\theta)&=\1_3\circ (\1_0, \theta, \1_r)=(\1_2\circ (\1_2, \1_1))\circ
(\1_0, \theta, \1_r)\\
&=\1_2\circ (\1_2\circ(\1_0, \theta), \1_1\circ \1_r)\\
&=\1_2\circ (\theta, \1_r).
\end{align*}

(3) We compute
\begin{align*}
\iota^l(\theta)&=\1_3\circ (\1_l, \theta, \1_0)=(\1_2\circ (\1_2, \1_1))\circ
(\1_l, \theta, \1_0)\\
&=\1_2\circ (\1_2\circ(\1_l, \theta), \1_1\circ \1_0)\\
&=\1_2\circ (\1_2\circ(\1_l, \theta), \1_0)\\
&=\1_2\circ(\1_l, \theta).
\end{align*}

(4)  Using parts (2) and (3), we compute
\begin{align*}
\iota^l_r(\theta)&=\1_3\circ (\1_l, \theta, \1_r)=(\1_2\circ (\1_2, \1_1))\circ
(\1_l, \theta, \1_r)\\
&=\1_2\circ (\1_2\circ(\1_l, \theta), \1_1\circ \1_r)\\
&=\1_2\circ (\iota^l(\theta), \1_r)\\
&=\iota_r\bullet \iota^l(\theta).
\end{align*}

If $\iota^l\bullet \iota_r=\iota^l_r$, taking $r=l=1$, then
\[\1_3=\iota_1(\iota^1(\1_1))=\iota^1(\iota_1(\1_1))
=\1_2\circ (\1_1, \1_2)=\1'_3.\]
Conversely, if $\1_3=\1'_3$ (equivalently, if $\ip$ is
$2a$-unitary), then we have
\[
\iota^l(\iota_r(\theta))=\1_2\circ (\1_l, \1_2\circ (\theta, \1_r))
=(\1_2\circ (\1_1, \1_2))\circ (\1_l, \theta, \1_r)
=\1_3\circ (\1_l, \theta, \1_r)=\iota_r^l(\theta).\]
for all $\theta$.
\end{proof}

%

\begin{example}
\label{xxex2.8}
Let $\ip=\As$. Assume $n=5$, $I=\{2, 4\}$ and $\sigma=(14)(235)$.
Then $\Gamma^{I}(\sigma)=(123)\in\S_3$,
$\pi^{I}(\sigma)= (12)\in\S_2$,
$\D_{I}(\sigma)=(1624735)\in\S_7$ and $\iota^1_2(\sigma)= (25)(346)\in\S_8$.
Following the convention introduced in Section \ref{xxsec8.1}
the sequences corresponded to

\centerline{$\sigma$, $\qquad$ $\Gamma^{I}(\sigma)$, $\qquad$ $\pi^{I}(\sigma)$, $\qquad$
$\D_{_I}(\sigma)$ $\qquad$ and $\qquad$ $\iota^1_2(\sigma)$}

\noindent
are given by

\centerline{$(4,5,2,1,3)$, $(2,1)$,
$(3,1,2)$, $(5,6,7,2,3,1,4)$ and $(1,5,6,3,2,4,7,8)$,}

\noindent
respectively.
\end{example}

By an easy calculation, we have the following useful lemmas.

\begin{lemma}
\label{xxlem2.9}
Let $\ip$ be 2-unitary. Let $n, l, r\ge0$ be integers and
$i, j, i_1, \cdots, i_s\in \n$. Then the following hold.
\begin{enumerate}
\item[(1)]
Assume that $i_1<\cdots<i_s$, then
    \[\D_{i_1,\cdots,i_s}=\D_{i_s+s-1}\bullet\cdots\bullet\D_{i_2+1}\bullet\D_{i_1}
		= \D_{i_1}\bullet\cdots\bullet\D_{i_s}, \]
    and
    \[\Gamma^{i_1,\cdots,i_s}
		= \Gamma^{i_s-s+1}\bullet\cdots\bullet\Gamma^{i_2-1}\bullet\Gamma^{i_1}
		= \Gamma^{i_1}\bullet\cdots\bullet\Gamma^{i_s}.\]
\item[(2)]
$\Gamma^{i+1}\bullet\D_i =\id$.
\item[(3)]
$\Gamma^i\bullet\D_i =\id$.
\item[(4)]
        $\D_{j}\bullet \Gamma^i=
            \begin{cases}
            \Gamma^i\bullet\D_{j+1},\quad i\le j;\\
            \Gamma^{i+1}\bullet\D_{j},\quad i>j.
            \end{cases}
        $
\item[(5)]
$\Gamma^{l+i}\bullet\iota^l_r=
\iota^l_r\bullet\Gamma^i$, and $\D_{l+i}\bullet\iota^l_r=\iota^l_r\bullet\D_i$.
\item[(6)]
$\Gamma^1\bullet\iota^1= \id$, and
$\Gamma^{n+1}\bullet\iota_1|_{\ip(n)} = \id_{\ip(n)}$.
\end{enumerate}
\end{lemma}

\begin{proof} This follows from easy computations and (OP2).
\end{proof}

\begin{lemma}
\label{xxlem2.10} Let $\ip$ be 2-unitary.
Let $n, k_1, \cdots, k_n\ge0$ be integers. Then, for each
$\theta\in\ip(n)$,
\[\theta\circ(\1_{k_1},\cdots, \1_{k_n})
= ((\D_1)^{k_1-1}\bullet\cdots \bullet(\D_n)^{k_n-1})(\theta),\]
where, by convention, $(\D_i)^{-1}$ means $\Gamma^i$ in case $k_i=0$.
\end{lemma}

\begin{proof} We use (OP2) in the following computation.
If $k_s\ge 2$, we have
\begin{align*}
\theta\circ (\1_{k_1}, \cdots, \1_{k_s}, & \underbrace{\1_1,\cdots, \1_1}_{t})\\
=&\theta\circ (\1_1\circ \1_{k_1}, \cdots, \1_1\circ \1_{k_{s-1}}, \1_2\circ (\1_{k_{s}-1}, \1_1),
\underbrace{\1_1\circ \1_1,\cdots, \1_1\circ \1_1}_{t})\\
=& \D_s(\theta)\circ(\1_{k_1}, \cdots, \1_{k_{s-1}}, \1_{k_s-1},
\underbrace{\1_1, \1_1,\cdots, \1_1}_{t+1}).
\end{align*}
If $k_s=0$, then
\begin{align*}
\theta\circ (\1_{k_1}, \cdots, \1_{k_s}, &
\underbrace{\1_1,\cdots, \1_1}_{t})\\
=&\theta\circ (\1_1\circ \1_{k_1}, \cdots, \1_1\circ \1_{k_{s-1}}, \1_0\circ ( ),
\underbrace{\1_1\circ \1_1,\cdots, \1_1\circ \1_1}_{t})\\
=& \Gamma^s(\theta)\circ(\1_{k_1}, \cdots, \1_{k_{s-1}},
\underbrace{\1_1, \1_1,\cdots, \1_1}_{t}).
\end{align*}
Combining the above, we have
\begin{align}\label{E2.10.1}\tag{E2.10.1}
\theta\circ (\1_{k_1}, \cdots, \1_{k_s}, &\underbrace{\1_1,\cdots, \1_1}_{t})\\
&=    \begin{cases} \D_s(\theta)\circ(\1_{k_1}, \cdots, \1_{k_{s-1}}, \1_{k_s-1},
		\underbrace{\1_1, \1_1,\cdots, \1_1}_{t+1}) & \text{if } k_s\ge 2; \\
    \Gamma^s(\theta)\circ (\1_{k_1}, \cdots, \1_{k_{s-1}},
		\underbrace{\1_1, \cdots, \1_1}_t) & \text{if } k_s=0.
    \end{cases} \notag
\end{align}
The lemma follows by applying the formula \eqref{E2.10.1} iteratively.
\end{proof}

Note that lemmas \ref{xxlem2.9} and \ref{xxlem2.10} hold
for plain operads. We are ready to prove Proposition \ref{xxpro0.5}.

\begin{proof}[Proof of Proposition \ref{xxpro0.5}]
Assume that $\ip$ is not $\Com$.
Let $n=\min \{ m \mid \ip(m)\neq \Bbbk \1_m\}$.
Since $\ip$ is unitary, $n\geq 1$. Since $\ip(n-1)
=\Bbbk \1_{n-1}$, there is a nonzero element
$\theta\in \ip(n)$ such that $\pi^{I}(\theta)=0$
where $I=[n-1]$. For every $J\subseteq
\n$ such that $|J|=n-1$,
$$\pi^{\emptyset}\bullet \pi^{J}(\theta)=\pi^{\emptyset}(\theta)
=\pi^{\emptyset}\bullet \pi^{I}(\theta)=0.$$
Firstly since $\pi^{\emptyset}: \ip(n-1)\to \ip(0)$
is an isomorphism,
\begin{equation}
\label{E2.10.2}\tag{E2.10.2}
\pi^{J}(\theta)=0
\end{equation}
for all $J\subseteq \n$ with $|J|=n-1$. For every
$w\geq n+1$ and $0\leq i\leq w-n$,
let $\theta^{w}_{i}=\iota^i_{w-i-n}(\theta)$.
We claim that $\{\theta^w_{0},\theta^w_{1}\cdots, \theta^w_{w-n}\}$
are linearly independent. We prove this by induction
on $w$. The initial case is when $w=n+1$.
Suppose
\begin{equation}
\label{E2.10.3}\tag{E2.10.3}
a \theta^w_0+ b \theta^w_1=0.
\end{equation}
By \eqref{E2.10.2},
we have $\Gamma^1(\theta^w_0)=0$ and $\Gamma^1(\theta^w_1)=\theta$.
Thus $b\theta=0$ after applying $\Gamma^1$ to \eqref{E2.10.3}.
Hence $b=0$. Applying $\Gamma^w$ to \eqref{E2.10.3}, we obtain
that $a=0$. Therefore the claim holds for $w=n+1$. Now suppose
the claim holds for $w$, and we consider the equation
\begin{equation}
\label{E2.10.4}\tag{E2.10.4}
\sum_{s=0}^{w-n+1} a_s \theta^{w+1}_s=0.
\end{equation}
Since 
$\Gamma^{w+1}(\theta^{w+1}_s)=
\begin{cases} 
\theta^{w}_s, & s< w-n+1,\\
0, & s=w-n+1,\end{cases}$ we obtain that $\sum_{s=0}^w a_s \theta^{w}_s=0$
after applying $\Gamma^{w+1}$ to \eqref{E2.10.4}. By induction
hypothesis, $a_s=0$ for all $s=0,\cdots, w-n$. Using $\Gamma^1$ instead
of $\Gamma^{w+1}$, we obtain that $a_s=0$ for all $s=1,\cdots,w-n+1$.
Therefore we proved the claim by induction.

By the claim $\dim \ip(w)\geq w-n$ for all $w$, which implies that
$\gkdim \ip\geq 2$, a contradiction.
\end{proof}

Recall that $\ast$ denote the right action of $\S_{n}$ on $\ip(n)$.
The following lemma is easy.

\begin{lemma}
\label{xxlem2.11} Let $\ip$ be a unitary operad.
\begin{enumerate}
\item[(1)]
Let $n$ be a positive integer and $I$ a subset of $\n$. Then, for all
$\theta\in \ip(n), \sigma \in \S_n$,
\begin{align}
\label{E2.11.1}\tag{E2.11.1}
 \pi^{I}(\theta\ast \sigma)= \pi^{\sigma I }(\theta)\ast \pi^{I}(\sigma),
\end{align}
where $\sigma I =\{\sigma(i)\mid i\in I\}\subseteq \n$.
\item[(2)]
Let $\mu\in \ip(m)$, $\nu\in \ip(n)$ and
$1\leq i\leq m$. Then
\begin{equation}
\label{E2.11.2}\tag{E2.11.2}
\pi^{I}(\mu \underset{i}{\circ}  \nu)=\pi^{J}(\mu)\underset{j}{\circ}  \pi^{I'}(\nu)
\end{equation}
where
$$\begin{aligned}
J&= (I \cap [i-1]) \cup \{i\} \cup ((I\cap \{i+n,\cdots,m+n-1\})-(n-1)),\\
I'&= (I\cap \{i, \cdots, i+n-1\})-(i-1),\\
j&=|I \cap [i-1]|+1.
\end{aligned}
$$
\end{enumerate}
\end{lemma}
\begin{proof} (1)
First we recall (OP3). For all $k_1\cdots, k_n\ge 0$,
$\theta_i\in \ip(k_i), \sigma_i \in\S_{k_i}, 1\le i\le n$ and $\theta\in \ip(n),
\sigma\in\S_n$,
\begin{equation}
\label{E2.11.3}\tag{E2.11.3}
(\theta\ast \sigma)\circ(\theta_1\ast\sigma_1, \cdots, \theta_n\ast\sigma_n)
=(\theta\circ(\theta_{\sigma^{-1}(1)}, \cdots, \theta_{\sigma^{-1}(n)}))
\ast(\sigma\circ(\sigma_1, \cdots, \sigma_n)).
\end{equation}
Let $k=|I|$.
Take $\theta_i= \1_1\in \ip(1)$ and $\sigma_i =\1_1\in \S_1$ for $i\in I$ and
$\theta_i= \1_0\in \ip(0)$ and $\sigma_i =\1_0\in \S_0$ otherwise.
By \eqref{E2.11.3}, we obtain
\begin{align*} \pi^{I}(\theta\ast \sigma) = & (\theta\ast \sigma)
\circ(\theta_1\ast \sigma_1, \cdots, \theta_n\ast \sigma_n)\\
=&(\theta\circ(\theta_{\sigma^{-1}(1)}, \cdots, \theta_{\sigma^{-1}(n)}))
\ast(\sigma\circ(\sigma_1, \cdots, \sigma_n))\\
=&\pi^{\sigma I}(\theta)\ast \pi^{I}(\sigma).
\end{align*}

(2) This follows from the definition of $\pi$ and (OP2).
\end{proof}

\subsection{Some basic lemmas}
\label{xxsec2.3}
We show the following properties of ideals of $\ip$.

\begin{lemma}
\label{xxlem2.12} Let $\ip$ be 2-unitary.
Let $\I$ and $\I'$ be ideals of $\ip$.
\begin{enumerate}
\item[(1)]
For each integer $n\ge 0$ and each subset $I\subseteq \n$,
$\pi^{I}\colon \I(n)\to \I(|I|)$ is surjective.
\item[(2)]
If $\I(n)=\I'(n)$ for some $n$, then $\I(s)=\I'(s)$ for all $s\le n$.
\end{enumerate}
\end{lemma}

\begin{proof} (1)
Without loss of generality, we may assume that the complement $\hat{I}$ 
of $I$ is $\{i_1,\cdots, i_s\}$ with $1\le i_1<\cdots<i_s\le n$. Since
$\pi^I=\Gamma^{i_1,\cdots,i_s} = \Gamma^{i_1}\bullet\cdots\bullet\Gamma^{i_s}$
[Lemma \ref{xxlem2.9}(1)], it
suffices to prove that each $\Gamma^{i_{t}}\colon \I(n+t-s)\to \I(n+t-s-1)$ is
surjective, which follows from the fact $\Gamma^{i_t}\bullet\D_{i_t} =\id$
[Lemma \ref{xxlem2.9}(3)] or $\Gamma^{i_t}\bullet\D_{i_t-1} =\id$
[Lemma \ref{xxlem2.9}(2)]. The proof is completed.

(2) This is an easy consequence of part (1).
\end{proof}

Let $X$ be a subset of an operad $\ip$. The operadic ideal
of $\ip$ generated by $X$ is the unique minimal ideal of $\ip$ that contains $X$.
We denote by $\langle X\rangle$ the ideal generated by $X$. An ideal is said to
be {\it finitely generated} if it can be generated by a finite subset.

\begin{lemma}
\label{xxlem2.13}
Let $\ip$ be 2-unitary and $\I$ an ideal of $\ip$.
\begin{enumerate}
\item[(1)]
If $\I$ is finitely generated, then $\I$ is generated by a subset
$X\subseteq \ip(n)$ for some $n$.
\item[(2)]
Suppose $\ip$ is a quotient operad of $\As$. Then $\I$ is finitely
generated if and only if there exists some $n\ge0$ and some
$x\in \I(n)$, such that $\I=\langle x\rangle$.
\end{enumerate}
\end{lemma}

\begin{proof}
(1) Let $X$ be a finite generating set of $\I$. Then there exists
some $n$ such that $X\subset \bigcup_{0\le i\le n} \I(n)$.
Therefore we can take $X\subseteq \langle \I(n)\rangle$ by
Lemma \ref{xxlem2.12}(1).

(2) It suffices to show the ``only if " part. Without loss of
generality, we suppose $\ip=\As$. Then $\I(n)$ is a right
submodule and hence a direct summand of $\k\S_n$. Since a
direct summand of a cyclic module is always cyclic, we have
$\I(n)=x\cdot\k\S_n$ for some $x\in \I(n)$. Clearly, $x$ is the
desired  generator of $\I$, which completes the proof.
\end{proof}

\begin{theorem}
\label{xxthm2.14}
Let $\ip$ be $\As/\I$ for some ideal $I\subseteq \As$. Let 
$k\geq 0$ be an integer and $M$ a submodule of the 
$\Bbbk \S_k$-module $\ip(k)$.
\begin{enumerate}
\item[(1)]
As an $\S$-module,
$\langle M\rangle$ is generated by elements of the form
$\iota^l_r\bullet  \D_{i_1}\bullet \cdots \bullet \D_{i_s}
\bullet \pi^I(x)$,  $x\in M$.
More explicitly, for every $n\ge 0$, $\langle M\rangle(n)$
is a $\k\S_n$-submodule generated by
\[X_n=\left\{(\iota^l_r\bullet \D_{i_1}\bullet \cdots \bullet \D_{i_s}
\bullet \pi^I)(x){\bigg |}
{x\in M, I\subseteq [k], l, r\ge 0,
l+r+s+|I|=n, \atop 1\le i_t\le |I|+t-1, t=1, \cdots, s}\right\}.\]

\item[(2)]
If $\Gamma^i(M)=0$ for all $1\le i\le k$, then $\langle M\rangle(k)=M$
and $\langle M\rangle(n)=0$ for all $n<k$.
\end{enumerate}
\end{theorem}

\begin{proof}
(1) Let $\I(n)$ denote the $\k\S_{n}$-submodule of $\ip(n)$ generated
by the subset $X_n$ and let $X$ be $\bigcup_{n} X_n$. We claim that
$\I=({\I(n)})_{n\ge0}$ is an ideal of $\ip$.

By definition we need to show that
$\theta\circ(\theta_1,\cdots, \theta_n)\in \I$
provided that one of $\theta$, $\theta_1, \cdots, \theta_n$ is in $\I$.
By (OP2) or \eqref{E1.1.2}, it suffices to show that
\[\1_{s}\circ(\1_{k_1}, \cdots, \1_{k_{t-1}}, \theta, \1_{k_{t+1}},
\cdots, \1_{k_s})\in X,
\ \ {\rm and}\ \  \theta\circ(\1_{k_1},\cdots, \1_{k_n})\in X\]
if $\theta\in X$. Since $\ip$ is $2a$-unitary, we have $\1_3=\1'_3$
and $\iota_r\bullet \iota^l=\iota^l \bullet \iota_r$
[Lemma \ref{xxlem2.7}(4)]. It follows that
\[\1_{s}\circ(\1_{k_1}, \cdots, \1_{k_{t-1}}, \theta, \1_{k_{t+1}},
\cdots, \1_{k_s})
=\iota^{k_1+\cdots+k_{t-1}}_{k_{t+1}+\cdots+k_s}(\theta).\]
and therefore the former holds.
For the latter, Lemma \ref{xxlem2.10} shows that
$\theta\circ(\1_{k_1},\cdots, \1_{k_n})$ is obtained by applying
$\Gamma^i$'s and $\Delta_j$'s on $\theta$ iteratively. The
commutativity relations in Lemma \ref{xxlem2.9}(2-4) together with
Lemma \ref{xxlem2.10} imply that
$$\theta\circ(\1_{k_1},\cdots, \1_{k_n})\in X_{k_1+\cdots+k_n}.$$

Clearly $M\subseteq X_k\subseteq \I(k)$, and $\I\subseteq\langle M\rangle$. By
definition $\langle M\rangle$ is the minimal ideal containing $M$, 
which forces that $\langle M\rangle\subseteq \I$ and hence 
$\langle M\rangle=\I$.

(2) If $\Gamma^i(M)=0$ for all $1\le i\le k$, then $\pi^I(M)=0$ for any
$I\subset [k]$ with $|I|< k$, and the statement follows.
\end{proof}

\section{Truncation Ideals}
\label{xxsec3}

\subsection{The truncation ideals $^k\iu$}
\label{xxsec3.1}
We first recall the definition of truncation ideals \eqref{E0.0.2}
from the introduction. Let $\ip$ be a unitary operad (or a unitary plain
operad). For integers $k, n \ge0$, we use ${^k\iu_\ip}(n)$ to denote the
subspace of $\ip(n)$ defined by
\begin{equation}
\label{E3.0.1}\tag{E3.0.1}
^k\iu_\ip(n) =
\bigcap\limits_{I\subset \n,\, |I|\leq  k-1} \ker\pi^I
=\begin{cases}
\bigcap\limits_{I\subset \n,\, |I|= k-1} \ker\pi^I, & \text{if }n\geq  k;\\
\quad\ \ 0, & \text{otherwise}.
\end{cases}
\end{equation}

Clearly, ${^0\iu}_{\ip}=\ip$. It is easily deduced from Lemma \ref{xxlem2.11}
that $^k\iu_\ip(n)$ is an $\S_n$-submodule of $\ip(n)$. Therefore
we obtain an $\S$-submodule $^k\iu_\ip=(^k\iu_\ip(n))_{n\ge0}$ of $\ip$. If
no confusion, we write ${^k\iu}={^k\iu_\ip}$ for brevity.
For two ideals $\I$ and $\mathcal{J}$ of $\ip$, let
$\I\mathcal{J}$ denote the $\S$-module generated by all elements of the form
$\mu\underset{i}{\circ} \nu$ for all $\mu\in \I(m)$ and $\nu\in \mathcal{J}(n)$
and all $i$. It is easy to see that $\I\mathcal{J}$ is also an ideal of
$\ip$.

\begin{proposition}
\label{xxpro3.1}
Let $\ip$ be a unitary operad {\rm{(}}respectively, a unitary plain
operad{\rm{)}}.
\begin{enumerate}
\item[(1)]
${^k\iu}$ is an ideal of $\ip$ for any $k\ge 1$.
\item[(2)]
If $m, n\geq 1$, then
${^m \iu}{^n \iu}\subseteq {^{m+n-1} \iu}$, and
if $m=0$ or $n=0$, then
${^m \iu}{^n \iu}\subseteq {^{m+n} \iu}$.
\end{enumerate}
\end{proposition}

\begin{proof}
(1)
Let $n>0$, $k_1, \cdots, k_n\ge 0$ be integers, and $\theta\in \ip(n)$,
$\theta_i\in \ip(k_i)$ for $i=1, \cdots, n$. We need to show that
if either $\theta\in {^k\iu} (n)$ or $\theta_i\in {^k\iu} (k_i)$
for some $i\in \n$, then
$\theta\circ (\theta_1, \cdots, \theta_n)\in {^k\iu} (m)$, where
$m=k_1+\cdots+k_n$.

Let $I$ be an arbitrary subset of $[m]$ with $|I|=k-1$. Then
we have
\begin{align*}
\pi^{I}(\theta\circ (\theta_1, \cdots,& \theta_n))
= (\theta\circ (\theta_1, \cdots, \theta_n))\circ
  (\1_{\chi_I(1)}, \cdots, \1_{\chi_I(m)})\\
&= \theta\circ (\theta_1\circ (\1_{\chi(1)}, \cdots, \1_{\chi_I(k_1)}),
  \cdots, \theta_n\circ (\1_{\chi_I(k_1+\cdots+k_{n-1}+1)},
	\cdots, \1_{\chi_I(m)}))\\
&= \theta\circ (\pi^{I_1}(\theta_1), \cdots, \pi^{I_n}(\theta_n))
\end{align*}
where in the last equality, each $I_i$ is a subset of $[k_i]$
determined by $I$, with $|I_i|\le k_i$ and $\sum\limits_{i=1}^n
|I_i|=|I|=k-1$.

If $\theta_i\in {^k\iu}(k_i)$ for some $i\in \n$, then
$\pi^{I_i}(\theta_i)=0$ by Lemma \ref{xxlem2.4} and the fact that
$|I_i|\le k-1$. So $\pi^I(\theta\circ (\theta_1, \cdots, \theta_n))=0$.

We are left to show that if $\theta\in {^k\iu}(n)$, then
$\pi^I(\theta\circ (\theta_1, \cdots, \theta_n))=0$.
Set $J=\{i\in \n \mid I_i\neq \emptyset\}$.
By $\sum\limits_{i=1}^n |I_i|=k-1$, we have $s:=|J|\le k-1$.
Consequently, $\pi^{J}(\theta)=0$. Observe
that if $I_i=\emptyset$ and $\ip(0)=\Bbbk \1_0$, then
\[\pi^{I_i}(\theta_i)=\theta_i\circ (\1_0, \cdots, \1_0)=\la_i \1_0\]
for some $\la_i\in \Bbbk$. Therefore, we have
\begin{align*}
\pi^{I}(\theta\circ (\theta_1, \cdots, \theta_n))
&=\theta\circ (\pi^{I_1}(\theta_1), \cdots, \pi^{I_n}(\theta_n))\\
&= (\prod\limits_{i\notin J}\la_i)(\pi^J(\theta)) \circ
  (\pi^{I_{j_1}}(\theta_{j_1}), \cdots, \pi^{I_{j_{s}}}(\theta_{j_{s}}))\\
&= 0
\end{align*}
where $J=\{j_1, \cdots, j_s\}$ and $1\le j_1<j_2<\cdots<j_s\le n$.

(2) If $m=0$ or $n=0$, the assertion follows from part (1).
For the rest of the proof, we assume that $mn>0$.

Let $\mu\in {^m \iu}(m_0)$ and $\nu\in {^n\iu}(n_0)$ and
let $i\leq m_0$. It suffices to show that
$$\mu \underset{i}{\circ}  \nu\in
{^{m+n-1}\iu}(m_0+n_0-1)$$
for all $i$. Let $I\subseteq [m_0+n_0-1]$
such that $|I|\leq m+n-2$. By Lemma \ref{xxlem2.11} (2), we have
\begin{equation}
\label{E3.1.1}\tag{E3.1.1}
\pi^{I}(\mu \underset{i}{\circ}  \nu)=\pi^{J}(\mu)\underset{j}{\circ}  \pi^{I'}(\nu)
\end{equation}
where
\[\begin{aligned}
J&= (I \cap [i-1]) \cup \{i\} \cup ((I\cap \{i+n_0,i+n_0+1,\cdots,m_0+n_0-1\})-(n_0-1)),\\
I'&= I\cap \{i,i+1,\dots,i+n_0-1\}-(i-1),\\
j&=|I \cap [i-1]|+1.
\end{aligned}
\]
If $|I'|\leq n-1$, then $\pi^{I'}(\nu)=0$, whence
$\pi^{I}(\mu \underset{i}{\circ}  \nu)=0$ by \eqref{E3.1.1}. Otherwise,
$|I'|\geq n$ and then
$$|J|=|I|+1-|I'|\leq m+n-2+1-n=m-1.$$
In this case $\pi^{J}(\mu)=0$,
whence $\pi^{I}(\mu \underset{i}{\circ}  \nu)=0$ by \eqref{E3.1.1}.
Therefore $\mu \underset{i}{\circ}  \nu\in
{^{m+n-1}\iu}(m_0+n_0-1)$ as required.
\end{proof}

Note that for any operad $\ip$,  $\ip(1)$ is always an associative algebra;
and for a unitary operad $\ip$, $\ip(1)$ is an augmented algebra and
$^1\iu(1)$ is the augmented ideal of $\ip(1)$.

In later sections we will also use a modification of truncation ideals that 
we define now. Let $M$ be an $\S_k$-submodule of ${^k \iu}(k)$. We consider 
the following two conditions
\begin{enumerate}
\item[(E3.1.2)]
$\nu \circ m\in M$ for all $\nu\in \ip(1)$ and $m\in M$.
\item[(E3.1.3)]
$m\underset{i}{\circ} \nu\in M$ for all $\nu\in \ip(1)$, $m\in M$ and
$1\leq i\leq k$.
\end{enumerate}
Define ${^k \iu}^M$ by
$${^k \iu}^M(n)=\{ \mu\in {^k \iu}(n)\mid \pi^{I}(\mu)\in M, 
\; \forall \; I\subseteq \n, |I|=k\}.$$
We have the following proposition similar to Proposition \ref{xxpro3.1}.

\begin{proposition}
\label{xxpro3.2}
Let $\ip$ be a unitary operad. Let $M$ {\rm{(}}respectively, $N${\rm{)}} be an
$\S_m$ {\rm{(}}respectively, $\S_n${\rm{)}}-submodule of
${^m \iu}(m)$ {\rm{(}}respectively, ${^n \iu}(n)${\rm{)}}.
\begin{enumerate}
\item[(1)]
If {\rm{(E3.1.2)}} holds, then ${^m\iu}^M$ is a left ideal of $\ip$.
\item[(2)]
If {\rm{(E3.1.3)}} holds, then ${^m\iu}^M$ is a right ideal of $\ip$.
\item[(3)]
${^m \iu}^M {^n \iu}^N \subseteq {^{m+n-1} \iu}^{MN}$  where $MN$ is an
$\S_{m+n-1}$-submodule generated by elements of the form $\mu\underset{i}{\circ} \nu$
for all $\mu\in M$ and $\nu\in N$ and $1\leq i\leq m$.
\end{enumerate}
\end{proposition}

\begin{proof} (1) By Lemma \ref{xxlem2.11}, ${^m \iu}^M$ is an
$\S$-module. Next we show that ${^m \iu}^M$ is a left ideal.

For $\nu\in \ip(m_0)$ and $\mu\in {^m\iu}^M(n_0)$,
$I\subseteq [m_0+n_0-1]$ with $|I|=m$, by Lemma \ref{xxlem2.11} (2),
$$\begin{aligned}
\pi^I(\nu \underset{i}{\circ} \mu)&=
\pi^{J}(\nu) \underset{j}{\circ} \pi^{I'}(\mu)
\end{aligned}
$$
where $J$, $I'$ and $j$ are given as after \eqref{E3.1.1}.
If $|I'|\leq m-1$, then $\pi^{I'}(\mu)=0$, and
$\pi^I(\nu \underset{i}{\circ} \mu)=0$.
Otherwise $|I'|=m$ (which is maximum possible)
and $I\subset \{i, i+1, \cdots, i+n_0-1\}$, then $j=1$,
$|J|=1$, and
$$\pi^{J}(\nu) \underset{1}{\circ} \pi^{I}(\mu)\in M$$
by assumption of $M$. Thus $\pi^I(\nu \underset{i}{\circ} \mu)
\in M$.

(2) The proof is similar to the proof of part (1).
For $\mu\in {^m\iu}^M(m_0)$ and $\nu\in \ip(n_0)$,
$I\subseteq [m_0+n_0-1]$ with $|I|=m$, by Lemma \ref{xxlem2.11}(2),
$$\begin{aligned}
\pi^I(\mu \underset{i}{\circ} \nu)&=
\pi^{J}(\mu) \underset{j}{\circ} \pi^{I'}(\nu)
\end{aligned}
$$
where $J$, $I'$ and $j$ are given as after \eqref{E3.1.1}.
If $|J|\leq m-1$, then $\pi^J(\mu)=0$, and
$\pi^I(\nu \underset{i}{\circ} \mu)=0$. If $|J|=m$,
then $\pi^J(\mu)\in M$ and $\pi^{I'}(\nu)\in \ip(1)$,
and by the assumption on $M$, we obtain that
$\pi^{J}(\mu) \underset{j}{\circ} \pi^{I'}(\nu)\in M$.
If $|J|=m+1$ (maximal possible), then
$I'=\emptyset$ and $\pi^{I'}(\nu)\in \ip(0)$.
Then
$$\pi^{J}(\mu) \underset{j}{\circ} \pi^{I'}(\nu)
=\pi^{J}(\mu) \underset{j}{\circ} \pi^{\emptyset}(\nu) \1_0
=\pi^{J\setminus \{j'\}}(\mu) \pi^{\emptyset}(\nu)
\in M$$
for some $j'$. Combining these cases, we have
$\pi^I(\mu \underset{i}{\circ} \nu)\in M$. Therefore
${^m \iu}^M$ is a right ideal.

(3) Let $\mu\in {^m \iu}^M(m_0)$ and $\nu\in {^n\iu}^N(n_0)$ and
let $i\leq m_0$. It suffices to show that
$$\mu \underset{i}{\circ}  \nu\in
{^{m+n-1}\iu}^{MN}(m_0+n_0-1)$$
for all $i$. By Proposition \ref{xxpro3.1}(2), $\mu \underset{i}{\circ}  \nu\in
{^{m+n-1}\iu}(m_0+n_0-1)$.

Let $I\subseteq [m_0+n_0-1]$
such that $|I|= m+n-1$. It suffices to show that
$\pi^{I}(\mu \underset{i}{\circ}  \nu)\in MN$. By Lemma \ref{xxlem2.11}(2),
\begin{equation}
\notag
\pi^{I}(\mu \underset{i}{\circ}  \nu)=\pi^{J}(\mu)\underset{j}{\circ}  \pi^{I'}(\nu)
\end{equation}
where
$$\begin{aligned}
J&= (I \cap [i-1]) \cup \{i\} \cup ((I\cap \{i+n_0,i+n_0+1,\cdots,
m_0+n_0-1\})-(n_0-1)),\\
I'&= (I\cap \{i,i+1,\dots,i+n_0-1\})-(i-1),\; {\text{and}}\\
j&=|I \cap [i-1]|+1.
\end{aligned}
$$
In particular, $|I'|+|J|=m+n$. If $|I'|\leq n-1$ or
$|J|\leq m-1$, then $\pi^{I'}(\nu)=0$ or $\pi^{J}(\mu)=0$.
Hence $\pi^{I}(\mu \underset{i}{\circ} \nu)=0\in M$.
The remaining case is when $|I'|=n$ and $|J|=m$. Then, in this case,
$\pi^{J}(\mu)\in M$ and $\pi^{I'}(\nu)\in N$.
Hence $\pi^{J}(\mu)\underset{j}{\circ} \pi^{I'}(\nu)\in MN$
by definition. Combining all cases,
$\pi^{I}(\mu \underset{i}{\circ} \nu)\in MN$ as required.
\end{proof}

A version of Proposition \ref{xxpro3.2} holds for plain operads.
The following lemmas are clear.

\begin{lemma}
\label{xxlem3.3}
Let $f: \ip\to \iq$ be a morphism of unitary operads 
{\rm{(}}preserving $\1_0${\rm{)}}. Then,
for every $I\subseteq \n$ with $|I|=k-1$, we have a commutative
diagram
$$\begin{CD}
\ip(n) @>>> \iq(n)\\
@V \pi^I VV @VV \pi^I V \\
\ip(k-1) @>>> \iq(k-1).
\end{CD}$$
As a consequence, $f$ maps from ${^k \iu_{\ip}}$
to ${^k \iu_{\iq}}$ for all $k\geq 0$.
\end{lemma}

Recall from Definition \ref{xxdef1.1}(4) that
an operad $\ip$ is called {\it connected} if
$\ip(1)=\Bbbk\cdot \1_1\cong  \Bbbk$.

\begin{lemma}
\label{xxlem3.4}
Let $\ip$ be a connected unitary operad. Then
${^1 \iu}={^2\iu}$.
\end{lemma}

\begin{proof} In this case, $\pi^{\emptyset}: \ip(1)\to \ip(0)$
is an isomorphism. Then
$$\ker (\pi^i: \ip(n)\to \ip(1))=
\ker (\pi^{\emptyset}: \ip(n)\to \ip(0))$$
for all $i\leq n$. Therefore ${^1 \iu}={^2\iu}$.
\end{proof}

Recall that operads $\Com$ and $\Uni$ are defined before Lemma \ref{xxlem1.12}.

\begin{lemma}
\label{xxlem3.5}
Let $\ip$ be a unitary operad.
\begin{enumerate}
\item[(1)]
${^1 \iu}$ is the maximal
ideal of $\ip$ and $\ip/{^1\iu}$ is isomorphic to either
$\Com$ or $\Uni$.
\item[(2)]
If $\ip/{^1\iu}\cong \Com$ and $\ip$ is connected, then
$\ip$ is 2-unitary.
\item[(3)]
$\Uni\oplus {^1\iu}$ is a suboperad of $\ip$ and it is unitary,
but not 2-unitary.
\item[(4)]
If $\ip/{^1\iu}\cong \Uni$, then
$\ip= \Uni \oplus {^1\iu}$.
\end{enumerate}
\end{lemma}

\begin{proof} (1) Since $\ip$ is unitary, $\ip(0)=\Bbbk \1_0$. By definition,
$${^1\iu}(n)=\ker (\pi^{\emptyset}: \ip(n)\to \ip(0)).$$
Then $\dim (\ip/{^1 \iu})(n)$ is either 0 or $1$ for each $n$. If $(\ip/{^1 \iu})(2)=0$,
then one can check that $(\ip/{^1 \iu})(n)=0$ for all $n\geq 2$. Consequently,
$\ip/{^1 \iu}=\Uni$. If $(\ip/{^1 \iu})(2)\neq 0$,
then one can check that $\ip/{^1\iu}$ is 2a-unitary and $(\ip/{^1 \iu})(n)=\Bbbk \1_n$
for all $n$.  Consequently, $\ip/{^1 \iu}=\Com$.

(2) Since $\ip$ is connected, ${^1\iu}={^2\iu}$ by Lemma \ref{xxlem3.4}.
Since $\ip(2)\neq {^1 \iu}(2)$, there is an $f\in \ip(2)$ such that
$\pi^1(f)=\1_1$. Since $\pi^{\emptyset} \bullet \pi^2(f)=\pi^{\emptyset}(f)=
\pi^{\emptyset} \bullet \pi^1(f)=\1_0$, we obtain that
$\pi^2(f)=\1_1$. Thus $f$ is a 2-unit by definition.

(3) This follows from the fact that ${^1 \iu}$ is an ideal of $\ip$.

(4) This follows from part (3).
\end{proof}

Now we are ready to prove the first two parts of Theorem \ref{xxthm0.4}.

\begin{proof}[Proof of Theorem \ref{xxthm0.4} (1, 2)]
We only prove the results for symmetric operads. The proofs
for plain operads are similar.

(1) Since $\ip$ is reduced, $\ip_{\geq n}:=\bigoplus_{i\geq n} \ip(i)$
is an ideal for every $n$. Since $\ip$ is artinian and
$\{\ip_{\geq n}\}_{n=0}^{\infty}$ is a descending chain of ideals,
$\ip_{\geq n}=0$ for some $n$. Let $n$ be the largest integer such
that $\ip(n)\neq 0$. If $n\geq 2$, then $\ip$ being reduced implies
that $\ip(n)$ is an ideal such that $\ip(n)^2=0$. This
contradicts the hypothesis that $\ip$ is semiprime.
Therefore $\ip(n)=0$ for all $n\geq 2$. Let $\ip(1)=A$.
In this case the left (or right) ideals of $\ip$ coincide
with the left (or right) ideals of $A$. Thus $A$ is left or
right artinian and semiprime. This implies that $A$ is
semisimple as desired.

(2) In the proof of part (2), we need to use truncation
ideals ${^k\iu}$ of $\ip$. By definition, $\bigcap_{k\geq 1}{^k \iu}
=0$. Since $\ip$ is left or right artinian, ${^k \iu}=0$
for some $k$. Let $n$ be the largest integer such that
${^n \iu}\neq 0$. If $n\geq 2$, by Proposition \ref{xxpro3.1}(3),
$({^n \iu})^2\subseteq {^{2n-1}\iu}=0$. This contradicts the
hypothesis that $\ip$ is semiprime. Therefore ${^2 \iu}=0$.
Let $A=\ip(1)$. By Proposition \ref{xxpro3.2} (1, 2),
if $A$ is not left (respectively, right) artinian, then $\ip$ is
not left (respectively, right) artinian. Since $\ip$ is
left or right artinian, so is $A$. Let $N$ be an ideal
of $A$ such that $N^2=0$. By Proposition \ref{xxpro3.2} (1, 2),
${^1 \iu}^N$ is an ideal of $\ip$. By Proposition \ref{xxpro3.2} (3),
$$({^1 \iu}^N)^2\subseteq {^1 \iu}^{N^2}={^1 \iu}^{0}={^2 \iu}=0.$$
Since $\ip$ is semiprime, ${^1 \iu}^N=0$, consequently, $N=0$.
Thus $A$ is semiprime. Since $A$ is left artinian or right
artinian, $A$ is semisimple. It remains to show that $\ip(n)=0$
for all $n\geq 2$. If not, let $n\geq 2$ be the largest integer
such that $\ip(n)\neq 0$ (such $n$ exists since $\ip$ is bounded
above). For every element $0\neq \mu\in \ip(n)$, $x:=\pi^{i}(\mu)\neq 0$
for some $i$ as ${^2\iu}=0$. Let $I$ be the ideal of $A$ generated
by $x$. For every element $f\in I$, $f$ can be written as
$\sum_{s=1}^w a_s x b_s$ with $a_s, b_s\in A$. Let
$$g=\sum_{s=1}^w a_s \underset{1}{\circ} ( \mu \underset{i}{\circ} b_s)
\in \ip(n).$$
Then $f=\pi^i(g)$. Since $I$ is a nonzero ideal of a semisimple
ring, $I=e A=Ae$ for some idempotent $e\in I$. Hence we may assume that
$f=e$ is a nonzero idempotent. Let $\nu=\pi^{i,j}(g)$. Then
$f=\pi^1(\nu)$ or $f=\pi^2(\nu)$. By symmetry, we assume that
$f=\pi^1(\nu)$. Let $h=g\underset{i}{\circ} \nu\in \ip(n+1)$. Then
$$\pi^i(h)=\pi^i(g)\circ \pi^1(\nu)=f\circ f=f\neq 0$$
which contradicts the fact that $\ip(n+1)=0$. Therefore $\ip(n)=0$
for all $n\geq 2$ as required.
\end{proof}

Part (3) of Theorem \ref{xxthm0.4} will be proved in Section \ref{xxsec6}.

\begin{lemma}
\label{xxlem3.6}
Let $\ip$ be a 2-unitary operad
and $\I$ an ideal of $\ip$. Then for each $k\ge 1$, $\I(k-1)=0$
if and only if $\I\subset {^k\iu}$.
\end{lemma}

\begin{proof} ($\Leftarrow$) is obvious. Next we show the other 
implication ($\Rightarrow$). Suppose $\I(k-1)=0$ for some $k\ge 1$.

If $n\ge k-1$, then we have
\[\pi^I(\theta)\in \I(k-1)=0\]
for any $\theta\in \I(n)$ and any $I\subseteq \n$ with $|I|=k-1$, and
hence $\theta\in {^k\iu}(n)$.

If $n<k-1$, for every $\theta\in \I(n)$, we have
$$(\Delta_{i_1}\bullet \cdots \bullet \Delta_{i_{k-1-n}})(\theta)\in \I(k-1)=0,$$
for all possible $i_1, \cdots, i_{k-1-n}$. Since $\ip$ is
$2$-unitary, each $\Delta_{i}$ is injective
by Lemma \ref{xxlem2.9} (2) (or (3)).
It follows that $\theta=0$ and hence $\I\subset {^k\iu}$.
\end{proof}

\subsection{The unique maximal ideal of a quotient operad of $\As$}
\label{xxsec3.2}
In this subsection we assume that $\ip=\As/\mathcal{W}$ for some
ideal $\mathcal{W}$. We use $\Phi_n$ to denote the alternating sum
$\sum_{\sigma\in\S_n}\mathrm{sgn}(\sigma)\sigma$, where $\mathrm{sgn}(\sigma)= 1$
if $\sigma$ is an even permutation, and $\mathrm{sgn}(\sigma)= -1$ if $\sigma$
is an odd permutation. When applied to an associative algebra,
the operator $\Phi_2$ gives exactly the usual commutator.

\begin{lemma}
\label{xxlem3.7}
As an ideal of $\ip$, ${^1\iu}={^2\iu}=\langle \Phi_2\rangle$.
\end{lemma}

\begin{proof} We only consider the case $\ip=\As$.
By Lemma \ref{xxlem3.4}, ${^1\iu}={^2\iu}$.
Clearly, we have ${^2\iu}\supseteq\langle \Phi_2\rangle$ since
$\Phi_2\in {^2\iu}(2)$. It suffices to show that
${^2\iu}\subseteq\langle \Phi_2\rangle$.
By definition $\pi^i(\sigma) = 1_1$ for all $n\ge 1$, $\sigma \in\S_n$, and
$1\le i\le n$. Thus we have
\[{^2\iu}(n)=\left\{ \sum_{\sigma\in\S_n}\lambda_\sigma\sigma\mid
\sum_{\sigma\in \S_n}\lambda_\sigma=0, \lambda_\sigma\in \Bbbk\right\}.\]

It is well-known that ${^2\iu}(n)$ is generated by the set
$\{1_n-\sigma\mid \sigma\in\S_n, \sigma\ne 1_n\}$. (It may not be a basis unless
$\ip=\As$). For any $\sigma_1, \cdots, \sigma_s\in\S_n$, we write
\begin{align*}
1_n-\sigma_1\cdots\sigma_s = & (1_n-\sigma_s)
   + (\sigma_{s}-\sigma_{s-1}\sigma_s) +\cdots
	 +(\sigma_2\cdots\sigma_s - \sigma_1\cdots \sigma_s)\\
= & (1_n-\sigma_s) + (1_n-\sigma_{s-1})\sigma_s + \cdots
   +(1_n-\sigma_1)\sigma_2\cdots\sigma_s .
\end{align*}
Since $\{(12), (23), \cdots, (n-1,n)\}$ generates 
the group $\S_n$, the above equality implies that  ${^2\iu}(n)$
is generated by the elements $1_n-(12), 1_n-(23), \cdots,
1_n-(n-1,n)$ as a right $\S_n$-module.  For each
$i\ge 1$, we have $\1_n-(i,i+1)=\iota^{i-1}_{n-i-1}(\Phi_2)$, and hence
${^2\iu}(n)\subseteq \langle \Phi_2\rangle(n)$. The assertion follows.
\end{proof}

\begin{lemma}
\label{xxlem3.8}
Let $\I\subsetneq \ip$ be an ideal. Then
either $\I={^2\iu}$ or $\I\subseteq {^3\iu}$.
\end{lemma}

\begin{proof}
First we claim that $\I\subseteq {^2\iu}$. Otherwise, there exist $n\ge 1$
and $\theta\in\I(n)$ such that $\pi^i(\theta)\ne 0$ for some $1\le i\le n$.
It follows that $\1_1\in \I$, and hence $\ip\subseteq \I$, which leads to a
contradiction.

Next assume that $\I\nsubseteq {^3\iu}$. Then there exist $n\ge 2$
and $\theta\in\I(n)$, such that $\pi^{i,j}(\theta)\ne 0$ for some
$1\le i<j\le n$. Note that $\pi^{i,j}(\theta)\in {^2\iu}(2)$ and
hence $\pi^{i,j}(\theta)= \lambda \Phi_2$ for some $\lambda\ne 0$.
Now Lemma \ref{xxlem3.7} implies that $\I={^2\iu}$.
\end{proof}

\subsection{A descending chain of ideals}
\label{xxsec3.3}
In this subsection we assume that $\ip=\As$. By definition
and Lemma \ref{xxlem2.4} (1),  ${^{k+1}\iu}\subseteq{^{k}\iu}$
for all $k\ge 0$. Thus we obtain a descending chain of ideals
\[{^1\iu}={^2\iu}\supseteq {^3\iu}\supseteq {^4\iu}\supseteq \cdots\]
of $\As$. Then for any ideal $\I$ of $\As$, after taking intersections
with ${^k\iu}$'s, we also obtain a descending chain of subideals
\[\I\cap{^2\iu}\supseteq \I\cap{^3\iu}\supseteq \I\cap{^4\iu}\supseteq \cdots. \]
Before continuing, we introduce a useful lemma.

\begin{lemma}
\label{xxlem3.9}
Let $n\geq k\ge 0$ be integers, and $\theta$ be in $\As(n)$. Then
\begin{enumerate}
\item[$(1)$]
$\Phi_2\circ(\theta, 1_1)\in {^{k+1}\iu}(n+1)$
if and only if $\theta\in{^k\iu}(n)$.
\item[$(1')$]
$\Phi_2\circ(1_1,\theta)\in {^{k+1}\iu}(n+1)$
if and only if $\theta\in{^k\iu}(n)$.
\item[$(2)$]
$1_2\circ (\theta,\Phi_2)\in {^{k+2}\iu}(n+2)$
if and only if $\theta\in {^k\iu}(n)$.
\item[$(2')$]
$1_2\circ (\Phi_2,\theta)\in {^{k+2}\iu}(n+2)$
if and only if $\theta\in {^k\iu}(n)$.
\end{enumerate}
\end{lemma}

\begin{proof} To avoid confusion, we use $\tau$ to denote the 2-cycle
$(12)\in\S_2$. Thus $\Phi_2= 1_2-\tau$. We only prove $(1)$ and $(2)$,
and the argument for $(1')$ and $(2')$ are the same.

(1) $(\Leftarrow)$ First we assume that $\theta\in {^k\iu}(n)$.
Take any subset $I$ of $[n+1]$ with $|I|=k$.
We claim that $\pi^I(\Phi_2\circ(\theta, 1_1))=0$.
By definition,
\begin{align}
\label{E3.9.1}\tag{E3.9.1}
\pi^I(\Phi_2\circ(\theta, 1_1))
=  (\Phi_2\circ(\theta, 1_1))\circ (1_{\chi_{_I}(1)},
   \cdots, 1_{\chi_{_I}(n)}, 1_{\chi_{_I}(n+1)})
= \Phi_2\circ (\pi^{I_1}(\theta), 1_{\chi_{_I}(n+1)}),
\end{align}
where $I_1= I\cap\n$. There are two cases: $n+1\in I$ or
$n+1\not\in I$. If $n+1\in I$, then $|I_1|=k-1$, and hence
$\pi^{I_1}(\theta)=0$ by assumption. The claim follows in this case.
Now we assume that $n+1\not\in I$, i.e., $I=I_1\subseteq \n$.
Obviously one has $\tau\circ(\theta', 1_0) = \theta'$, and hence
$\Phi_2\circ(\theta', 1_0)=0$ for any $\theta'$. Now in both cases,
we have  $\pi^I(\Phi_2\circ(\theta, 1_1))=0$.
Therefore the claim holds and the ``if\;" part follows.

$(\Rightarrow)$ Next we prove the ``only if\;" part. Assume that
$\Phi_2\circ(\theta, 1_1)\in {^{k+1}\iu}(n+1)$. We need only to
show that $\pi^{I}(\theta)=0$ for every subset $I\subseteq\n$ with
$|I|=k-1$. Set $I'=I\cup\{n+1\}$. Clearly $I'$ is a subset of
$[n+1]$ with $|I'|=k$. By \eqref{E3.9.1}, we have
\[0=\pi^{I'}(\Phi_2\circ(\theta, 1_1))=
\Phi_2\circ(\pi^I(\theta), 1_1)=
1_2\circ(\pi^I(\theta), 1_1)- \tau\circ(\pi^I(\theta), 1_1).\]
Note that
\[\{1_2\circ(\sigma, 1_1)\mid \sigma\in \S_{k-1}\}\bigcup
\{\tau\circ(\sigma, 1_1)\mid \sigma\in \S_{k-1}\}\]
are linearly independent in $\k\S_{k}$.
It follows that $1_2\circ(\pi^I(\theta), 1_1)=0$ and
hence $\pi^I(\theta)=0$.

(2) For every $I\subset \n$, denote
by $\tilde{I}$ the set $I\cup\{n+1, n+2\}$.
Then we obtain a 1-1 correspondence between subsets of $\n$
and the ones of $[n+2]$ containing both $n+1$ and $n+2$.

$(\Leftarrow)$
Assume that $\theta\in {^k\iu}(n)$. Then for every
$J\subseteq [n+2]$
with $|J|=k+1$, we claim that
$$\pi^J(1_2\circ (\theta, 1_2-\tau)) = 0.$$
Easy calculations show that
\[\Gamma^{n+1}(1_2\circ (\theta, 1_2-\tau)) =
\Gamma^{n+2}(1_2\circ (\theta, 1_2-\tau)) = 0.\]
Thus, if $\{n+1, n+2\}\not\subseteq J$, then
$\pi^J (1_2\circ (\theta, 1_2-\tau))= 0$
since $\pi^J$ will factor through $\Gamma^{n+1}$ or
$\Gamma^{n+2}$ in this case. Now we may assume
that $J=\tilde I$ for some $I\subseteq \n$. Then
\begin{equation}
\label{E3.9.2}\tag{E3.9.2}
\pi^{\tilde I}(1_2\circ (\theta, 1_2-\tau))
= 1_2\circ (\pi^I(\theta), 1_2-\tau) =0.
\end{equation}
The ``if" part follows.

$(\Rightarrow)$
For the ``only if" part, again we use \eqref{E3.9.2} and
the fact that $1_2\circ (\pi^I(\theta),1_2-\tau))=0$
if and only if $\pi^I(\theta)=0$.
\end{proof}

The main result of this subsection is the following
separability property of the ideals ${^k\iu}$ of $\As$.

\begin{proposition}
\label{xxpro3.10}
\begin{enumerate}
\item[(1)]
Let $\I$ be a nonzero ideal of $\As$. Then $\I\cap{^k\iu}
\neq \I\cap{^{k+1}\iu}$ for all $k\gg 0$.
\item[(2)]
${^k\iu}\neq {^{k+1}\iu}$ for every $k\ge 2$.
\end{enumerate}
\end{proposition}

\begin{proof}
(1) Note that $\bigcap_{k\ge 0}{^k\iu}=0$ since ${^k\iu}(k-1)=0$ for all
$k\ge 1$. By Lemma \ref{xxlem3.8} and the assumption $\I\ne 0$, we have
$I\cap {^2\iu}\neq 0$. Thus $\I\cap{^{k_0}\iu} \ne \I\cap{^{k_0+1}\iu}$
for some $k_0\ge 1$. There exist some $k_0\ge1$, $n\ge k_0$, and
$\theta\in \I(n)$ such that $\theta\in {^{k_0}\iu}(n)$ while
$\theta\not\in {^{k_0+1}\iu}(n)$. By Lemma~\ref{xxlem3.9},
$\Phi_2\circ(\theta,1_1)\in {^{k_0+1}\iu}(n+1)$, and
$\Phi_2\circ(\theta,1_1)\not\in {^{k_0+2}\iu}(n+1)$, which implies that
$\I\cap{^{k_0+1}\iu} \ne \I\cap{^{k_0+2}\iu}$. By induction we may show
that $\I\cap{^k\iu}\neq \I\cap{^{k+1}\iu}$ for all $k\ge k_0$.

(2) The statement follows from the above proof and the fact that
$\Phi_2\in {^2\iu}$ and $\Phi_2\not\in {^3\iu}$.
\end{proof}

\begin{remark}
\label{xxrem3.11}
Recall that the descending chain condition (DCC, for short) for an object
$C$ means that any descending chain
$$C\supseteq C_1\supseteq C_2\supseteq C_3\supseteq \cdots$$
of subobjects of $C$ is stable, that is, $C_k = C_{k+1} =\cdots$ for
sufficiently large $k$. The proposition says that the DCC
does NOT hold for any nonzero ideal of $\As$ and $\As$ is not artinian.
\end{remark}

Let $\ip$ be a unitary operad and let ${^i \iu}$ be ${^i \iu}_{\ip}$.
Let $\iu$ denote the
$\S$-module
\begin{equation}
\label{E3.11.1}\tag{E3.11.1}
\bigoplus_{i=0}^{\infty} {^i \iu}(i).
\end{equation}

\begin{proposition}
\label{xxpro3.12}
Let $\ip$ be a unitary operad.
\begin{enumerate}
\item[(1)]
$\iu$ is closed under partial compositions.
\item[(2)]
$\Bbbk \1_1\oplus \iu$ is a unitary operad.
\item[(3)]
$\Bbbk \1_1\oplus \iu$ is $\Uni$-augmented.
\end{enumerate}
\end{proposition}

\begin{proof} (1) By the proof of Proposition \ref{xxpro3.1} (2),
${^m \iu}(m) {^n \iu}(n)\subseteq {^{m+n-1}\iu}(m+n-1)$
for all $m,n$. The assertion follows.

(2,3) These follow from part (1).
\end{proof}

\section{Dimension computation, basis theorem and categorification}
\label{xxsec4}

\subsection{Definitions of growth properties}
\label{xxsec4.1}
We collect some definitions.

\begin{definition}
\label{xxdef4.1}
Let $\M=(\M(n))_{n\ge 0}$ be an $\S$-module (or a $\Bbbk$-linear
operad).
\begin{enumerate}
\item[(1)]
The sequence $(\dim \M(0), \dim \M(1), \cdots)$ is called the
\emph{dimension sequence} (or simply \emph{dimension}) of $\M$.
We call $\M$ {\it locally finite} if $\dim_{\Bbbk} \M(n)<\infty$
for all $n$.
\item[(2)]
The {\it generating series} of $\M$ is defined to be
$$G_{\M}(t)=\sum_{n=0}^{\infty} \dim \M(n) t^n \in {\mathbb Z}[[t]].$$
The {\it exponential generating series}  of $\M$ is defined to be
$$E_{\M}(t)=\sum_{n=0}^{\infty} \frac{\dim \M(n)}{n!} t^n
\in {\mathbb Q}[[t]].$$
\item[(3)]
The {\it exponent} of $\M$ is defined to be
$$\exp(\M):=\limsup_{n\to \infty} (\dim \M(n))^{\frac{1}{n}}.$$
We say $\M$ has {\it exponential growth} if $\exp(\M)>1$.
We say $\M$ has {\it finite exponent} if
$\exp(\M)<\infty$.
%
\item[(4)]
We say that $\M$ has \emph{polynomial growth}
if there are $0< C, k< \infty$ such that
$\dim \M(n)< Cn^k$ for all $n>0$. The infimum of
such $k$ is called the
\emph{order of polynomial growth} and denoted by $o(\M)$.
\item[(5)]
We say $\M$ has {\it sub-exponential growth} if $\exp(\M)\leq 1$ and
if $\M$ does not have polynomial growth.
\item[(6)]
The {\it Gelfand-Kirillov dimension} ({\it GKdimension} for short)
of $\M$
is defined to be
$$\gkdim (\M)=\limsup_{n\to\infty}
\log_{n} \left(\sum_{i=0}^{n} \dim_{\Bbbk} \M(i)\right)$$
which is the same as \eqref{E0.0.3}.
\end{enumerate}
\end{definition}

When we talk about the growth of an operad $\ip$, we implicitly assume that
$\ip$ is locally finite.
It is easy to see that $\exp(\As)=\infty$, so $\As$ has (infinite) exponential
growth. And $\gkdim (\Com)=1$, so $\Com$ has polynomial growth.
We will see that for every
integer $k\geq 1$, there exists a quotient operad
$\ip/{^k\iu}$ has polynomial growth of order (no more than) $k$.
First we state a lemma for arbitrary unitary operads.

\begin{lemma}
\label{xxlem4.2}
Let $\ip$ be a $\k$-linear {\rm{(}}symmetric or plain{\rm{)}}
unitary operad. If ${^k\iu}=0$ for some $k$, then $\gkdim \ip\leq k$.
\end{lemma}

As usual
$${n \choose k}=\frac{n!}{k!\cdot (n-k)!}.$$

\begin{proof} Consider the restriction operator $\pi^I: \ip(n)\to \ip(k-1)$
for all $n\geq k-1$, which induces an injective map
$$(\pi^I)': \quad \ip(n)/\ker \pi^I \to \ip(k-1)$$
where $I\subseteq \n$ with $|I|=k-1$.
By hypothesis and definition,
$$0={^k \iu}(n)=\bigcap\limits_{I\subset \n,\, |I|= k-1} \ker\pi^I.$$
Hence we have an injective map
$$\ip(n)\xrightarrow{\cong} 
\frac{\ip(n)}{(\bigcap\limits_{I\subset \n,\, |I|= k-1} \ker\pi^I)}
\to \bigoplus\limits_{I\subset \n,\, |I|= k-1} \frac{\ip(n)}{\ker\pi^I}
\to \bigoplus\limits_{I\subset \n,\, |I|= k-1} \ip(k-1),$$
which implies that
$$\dim \ip(n)\leq \dim \ip(k-1) {n \choose k-1}$$
for all $n\geq k-1$. The assertion follows.
\end{proof}

Let $\ip$ be a unitary operad and $\I$ an ideal of $\ip$.
Let $d^k_{\I}(n)$  denote the codimension of $({^k\iu}\cap \I)(n)$ in $\I(n)$,
that is,
\begin{equation}
\label{E4.2.1}\tag{E4.2.1}
d^k_{\I}(n)=\dim_\k\dfrac{\I}{{^k\iu}\cap \I}(n) = \dim_\k \I(n)-\dim_\k({^k\iu}\cap \I)(n),
\end{equation}
where the second equality holds if $\ip$ is locally finite.
If $\I=\ip$, we have
\[d^k(n)=\dim_\k\dfrac{\ip}{{^k\iu}}(n) = \dim_\k \ip(n)-\dim_\k {^k\iu}(n),\]
where the second equation holds if $\ip$ is locally finite.

We do not assume that $d^k_{\I}(n)$ is finite. When $\ip$ is locally
finite, we will give a recursive formula for $d^k_{\I}(n)$. The key idea
is to find a basis for the quotient module
$\dfrac{{^{k}\iu}\cap \I}{{^{k+1}\iu}\cap \I}(n)$, so we can calculate
$\dim \dfrac{{^{k}\iu}\cap I}{{^{k+1}\iu}\cap \I}(n)$ for all $n$.

For every subset $I\subseteq \n$, we use $c_I$ to denote the element
in $\S_n$ which corresponds to the permutation
\begin{equation}
\label{E4.2.2}\tag{E4.2.2}
c_I:=(1,\cdots,i_1-1, i_1+1, \cdots, i_s-1, i_s+1, \cdots, n, i_1,\cdots, i_s),
\end{equation}
where $I=\{i_1,\cdots, i_s\}$ with $i_1<\cdots < i_s$.
Let $\ip$ be a $2$-unitary operad. By an easy
calculation we have
\[\Gamma^I((\1_2\circ (\1_{n-s}, \1_s))\ast c_I)=\1_{n-s}.\]
In fact, we have a more general result.

\begin{lemma}
\label{xxlem4.3}
Let $\ip$ be 2-unitary. Let $n\geq k$ be integers and set $s=n-k$.
\begin{enumerate}
\item[(1)]
Let $I\subseteq [n]$ be a subset with $|I|=s$. Then
$\Gamma^I((\1_2\circ(\theta, \1_{s}))\ast c_I) = \theta$
for all $\theta\in\ip(k)$.
\item[(2)]
Let $J\subseteq \n$ be a subset with $|J|=k$. Then for every
$\theta\in {^k\iu}(k)$ and every $\sigma\in\S_{n}$,
\[\pi^{J}((\1_2\circ(\theta, \1_{s}))\ast \sigma)=0\]
unless $J=\{\sigma^{-1}(1),\cdots, \sigma^{-1}(k)\}$.
\end{enumerate}
\end{lemma}

\begin{proof} (1) To avoid possible confusion, we use $1_{n}$ for
$\1_{\S_n}\in \S_n$ for all $n\geq 0$. Applying (OP3)
and using the fact that $\theta\ast 1_{k}=\theta$
for all $\theta\in \ip(k)$, we have
\begin{align*}
\Gamma^I(\1_2\circ(\theta, \1_s)\ast c_I)
= & ((\1_2\circ(\theta, \1_s))\ast c_I)\circ(\1_{\chi_{_{\hat I}}(1)},
   \cdots, \1_{\chi_{_{\hat I}}(n)})\\
= & ((\1_2\circ(\theta, \1_s))\ast c_I)\circ(\1_{\chi_{_{\hat I}}(1)}\ast 1_{\chi_{_{\hat I}}(1)},
   \cdots, \1_{\chi_{_{\hat I}}(n)}\ast 1_{\chi_{_{\hat I}}(n)})\\
= & [(\1_2\circ(\theta, \1_s))\circ (\1_{\chi_{_{\hat I}}(c_I^{-1}(1))}, \cdots,
	 \1_{\chi_{_{\hat I}}(c_I^{-1}(n))})]\ast
	 [c_I\circ(1_{\chi_{_{\hat I}}(1)}, \cdots, 1_{\chi_{_{\hat I}}(n)})]\\
=& [(\1_2\circ(\theta, \1_s))\circ
   (\underbrace{\1_1, \cdots, \1_1}_{k}, \underbrace{\1_0,\cdots,\1_0}_{s})]
	 \ast 1_{k} \\
=& \1_2\circ(\theta\circ(\underbrace{\1_1, \cdots, \1_1}_{k}),
   \1_s\circ(\underbrace{\1_0,\cdots,\1_0}_{s}))\\
=& \1_2\circ(\theta, \1_0) = \theta,
\end{align*}
where the second to last equality is Lemma \ref{xxlem2.7}(1) and
the last equality uses the hypothesis that $\ip$ is 2-unitary.

(2) We will consider the special case $\sigma=1_{n}\in \S_n$,
and the general case follows from Lemma \ref{xxlem2.11}.
If there exists some $r\in [k]$ such that $r\not\in J$,
then $\pi^J =\pi^{J'}\bullet \Gamma^r$ for some
$J'\subseteq [n-1]$ by Lemma \ref{xxlem2.4}(3). We
have
$$\Gamma^r(\1_2\circ (\theta, \1_{n-k})) = \1_2\circ
(\Gamma^r(\theta), \1_{n-k}) = 0,$$
as $\theta\in {^k\iu}(k)$. Hence $\pi^{J}(\1_2\circ (\theta, \1_{n-k}))=0$
as desired.
\end{proof}

As an immediate consequence of the above lemma we have the following.

\begin{corollary}
\label{xxcor4.4}
Let $I, I'\subseteq\n$ be subsets with $|I|=|I'|=n-k=:s$.
For $\theta\in {^k \iu}(k)$, we have
\begin{equation}\label{E4.4.1}\tag{E4.4.1}
\Gamma^{I'}((\1_2\circ (\theta, \1_s))\ast c_I)=
\begin{cases}
\theta, & \quad \text {if } I' = I;  \\
0, & \quad \text{otherwise.}
\end{cases}
\end{equation}
\end{corollary}

We are now in a position to give a recursive formula to
compute the dimension of ${^{k}\iu \cap \I}$.
By convention, ${{^0}\iu}=\ip$.
Let $$G_{\I}^{k}(t)=\sum_{n=0}^{\infty} d_{\I}^{k}(n) t^n$$
and
let $$f_{\I}(k)=d^{k+1}_{\I}(k)-d^{k}_{\I}(k)$$
for all $k\geq 0$.
Clearly, it follows from \eqref{E4.2.1} that
\begin{align*}
f_{\I}(k)= \dim_{\Bbbk} ({^{k}\iu}\cap \I)(k),\ \ {\rm and} \ \
f_{\ip}(k)=\dim_{\Bbbk}  {^{k}\iu} (k).
\end{align*}
Note that if $f_{\I}(k)$ is not finite, then it denotes a
cardinal.

\begin{theorem}
\label{xxthm4.5}
Let $\ip$ be 2-unitary and $\I$ an ideal of $\ip$.
Let $n\ge k\geq 0$ be integers.
\begin{enumerate}
\item[(1)]
Let
$\{\theta_i\mid 1\le i\le f_{\I}(k)\}$ be a basis of
$({^k \iu}\cap \I)(k)$.
Then
\[\{\1_2\circ (\theta_i, \1_{n-k})\ast c_I\mid 1\le i\le f_\I(k),
I\subseteq \n, |I|= n-k\} \]
forms a basis of
$(({^{k}\iu}\cap \I)/({^{k+1}\iu}\cap \I))(n)$. Consequently,
\[\dim_{\Bbbk} \dfrac{{^{k}\iu}\cap \I}{{^{k+1}\iu}\cap \I}(n)=f_\I(k){\tbinom{n}{k}}.\]
\item[(2)]
\[d^{k+1}_{\I}(n) = d^k_{\I}(n) + f_{\I}(k){\tbinom{n}{k}}.\]
Equivalently,
$$G_{\I}^{k+1}(t)-G_{\I}^{k}(t)=f_{\I}(k) \;\frac{t^{k}}{(1-t)^{k+1}}.$$
\item[(3)]
If $\I=\ip$, then
\[d^{k+1}(n) = d^k(n) + (d^{k+1}(k)-d^{k}(k)){\tbinom{n}{k}},\]
for all $n$.
\end{enumerate}
\end{theorem}

\begin{proof}
(1) Let $I, I'\subseteq\n$ be subsets with $|I|=|I'|=n-k$.
By Corollary \ref{xxcor4.4}, we have
\begin{equation}\label{E4.5.1}\tag{E4.5.1}
\Gamma^{I'}(\1_2\circ (\Gamma^I(\theta), \1_{n-k})\ast c_I)=
\begin{cases}  \Gamma^I(\theta), & \quad \text {if } I' = I,  \\
0, & \quad \text{otherwise,}
\end{cases}
\end{equation}
for all $\theta\in ({^k\iu}\cap \I)(n)$, because
$\Gamma^I(\theta)\in ({^k\iu}\cap \I)(k)$.
For each $\theta\in ({^{k}\iu}\cap \I)(n)$, we set
\begin{equation}
\label{E4.5.2}\tag{E4.5.2}
\theta' = \theta -
\sum\limits_{I\subseteq \n \atop |I|= n-k}
\1_2\circ (\Gamma^I(\theta), \1_{n-k})\ast c_I.
\end{equation}
Then \eqref{E4.5.1} implies that $\theta'\in ({^{k+1}\iu}
\cap \I)(n)$, and hence the image of the elements of the form
$$\1_2\circ (\theta_i, \1_{n-k})\ast c_I$$
span $\dfrac{{^{k}\iu}\cap \I}{{^{k+1}\iu}\cap \I}(n)$.

Next we show the linear independency. Assume that
\[\sum_{1\le i\le f_{\I}(k)\atop I\subseteq \n, |I|= n-k}
\lambda_{i,I} \1_2\circ (\theta_i, \1_{n-k})\ast c_I \in ({^{k+1}\iu}\cap \I)(n)\]
for some $\lambda_{i,I}\in\k$. Then, for each $I$, by applying
$\Gamma^I$ we obtain that
\[\sum_{1\le i\le f_{\I}(k)} \lambda_{i,I} \theta_i =0,\]
again we use Corollary \ref{xxcor4.4}
here. It follows that all $\lambda_{i,I}$'s must be zero.

(2,3) Easy consequences of part (1).
\end{proof}

\subsection{Basis Theorem}
\label{xxsec4.2}

As a consequence of Theorem \ref{xxthm4.5}(1), we have the
following result concerning a $\Bbbk$-linear basis of
$\ip$. In theorem below, if $z_k$ is not finite, then it denotes a
cardinal.

Recall that an operad $\ip$ is {\it finitely generated}
if there is a finite dimensional subspace $X$ such that every element
in $\ip$ is generated by $X$ by using operad composition
and $\S_n$-actions for $n\geq 0$.

\begin{theorem}
\label{xxthm4.6}
Suppose $\ip$ is a 2-unitary operad.
\begin{enumerate}
\item[(1)]{\rm{[Basis theorem]}}
For each $k\geq 0$, let
$$\Theta^k:=\{\theta^k_1,\cdots, \theta^k_{z_k}\}$$
be a $\Bbbk$-linear basis for $^k\iu(k)$. Let ${\mathbf B}_k(n)$
be the set
$$\{\1_2\circ (\theta^k_i, \1_{n-k})\ast c_I\mid 1\le i\le z_k,
I\subseteq \n, |I|= n-k\}.$$
Then $\ip(n)$ has a
$\Bbbk$-linear basis
$$\bigcup_{0\leq k\leq n} {\mathbf B}_k(n)=
\{\1_n\} \cup \bigcup_{1\leq k\leq n} {\mathbf B}_k(n),$$
and, for every $k\geq 1$,
${^k \iu}(n)$ has a
$\Bbbk$-linear basis $\bigcup_{k\leq i\leq n} {\mathbf B}_i(n)$.
\item[(2)]
$\ip$ is generated by $\{\1_0, \1_1, \1_2\} \cup \left[ \bigoplus_{k\geq 1}
{^k \iu(k)}\right]$.
\item[(3)]
If $\ip$ is locally finite and $^n\iu=0$ for
some $n$, then it is finitely generated.
\end{enumerate}
\end{theorem}

\begin{proof} (1) For each $n\geq 0$, $\ip(n)$ admits a decreasing filtration
$\{{^k \iu}(n)\}_{k=0}^{\infty}$. As a vector space, we have
$$\ip(n)\cong \bigoplus_{k=0}^{\infty} {^k \iu}(n)/{^{k+1} \iu}(n)
\cong \Bbbk \1_n \oplus \bigoplus_{k=1}^{\infty} {^k \iu}(n)/{^{k+1} \iu}(n).$$
By Theorem \ref{xxthm4.5} (1), ${\mathbf B}_k(n)$ is a
$\Bbbk$-linear basis of ${^k \iu}(n)/{^{k+1} \iu}(n)$.
Note that ${\mathbf B}_k(n)$ is empty if $k\geq n+1$.
The first assertion follows. The proof of the second assertion is
similar.

(2) This follows from part (1).

(3) This follows from part (2) and the fact that
${^k \iu(k)}=0$ for all $k\geq n$.
\end{proof}

%

As a consequence of the above basis theorem, we have the following
corollary. A morphism $f$ of operads is called a morphism of 2-unitary
operads if $f$ preserves $\1_i$ for $i=0,1,2$. Before we prove
the corollary, we need the following lemma.

\begin{lemma}
\label{xxlem4.7}
Let $\ip$ be a 2-unitary operad and $\I$ be an ideal of
$\ip$. Then
${^k \iu_{\ip/\I}}\cong{^k \iu_\ip}/({^k \iu_\ip}\cap \I)$.
\end{lemma}

\begin{proof}
Let $\iq=\ip/\I$. The canonical morphism $\varphi: \ip \to \iq$
induces a natural map $f: {^k \iu_{\ip}}\to {^k \iu_{\iq}}$ by
Lemma \ref{xxlem3.3}. Since $\I$ is the kernel
of $\varphi$, $f$ induces a natural injective morphism
$$g: {^k \iu_{\ip}}/({^k \iu_{\ip}}\cap \I)\to {^k \iu_{\iq}}.$$
It remains to show that $g$ is surjective, equivalently, to show
that, for each $n$,
$$\phi:  ({^k \iu_{\ip}}(n)+\I(n))/\I(n)\to {^k \iu_{\iq}}(n)$$
is surjective. For every $x\in {^k \iu_{\iq}}(n)$, let $\theta
\in \ip(n)$ such that $\varphi(\theta)=x$. Suppose
$\theta\in {^i \iu_{\ip}}$ for some $i$. We will use induction
to show that $i\geq k$ for some choice of $\theta$. There is 
nothing to be proved if $i\geq k$. Assume now that $i<k$. Then
$\Gamma^{J}(\theta)\in {^i \iu_{\ip}}(i)$ when $|J|=n-i$.
Let
$$\theta' = \theta -
\sum\limits_{J\subseteq \n \atop |J|= n-i}
\1_2\circ (\Gamma^J(\theta), \1_{n-i})\ast c_J
$$
which is similar to the element given in \eqref{E4.5.2}.
By Corollary \ref{xxcor4.4} or the proof of Theorem \ref{xxthm4.5} (1),
$\Gamma^J(\theta')=0$ for all $J\subseteq
\n$ with $|J|= n-i$. This means that $\theta'\in
{^{i+1} \iu}_{\ip}(n)$. For each $J$ as above, we have
$$\varphi(\Gamma^J(\theta))=\Gamma^{J}(\varphi(\theta))=\Gamma^{J}
(x)=0$$
as $x\in {^k \iu_{\iq}}(n)$ and $k>i$. Thus $\Gamma^{J}(\theta)\in \I(i)$
for all $J$.
Consequently,
$$\Omega:=\sum\limits_{J\subseteq \n \atop |J|= n-i}
\1_2\circ (\Gamma^J(\theta), \1_{n-i})\ast c_J\in \I(n).$$
Hence $\phi(\theta')
=\phi(\theta)=x$. Replacing
$\theta$ by $\theta'$ we move $i$ to $i+1$. The assertion follows
by induction.
\end{proof}

Recall from \eqref{E3.11.1} $\iu_\ip$ denote the $\S$-submodule 
$\bigoplus_{i=0}^\infty {^k\iu_\ip}(k)$.

\begin{corollary}
\label{xxcor4.8}
Suppose $\ip$ and $\iq$ are 2-unitary operads.
Let $f:\ip\to \iq$ be a morphism of 2-unitary operads.
\begin{enumerate}
\item[(1)]
$f$ is uniquely determined by $f\mid_{\iu_\ip}$.
\item[(2)]
$f$ is injective if and only if
$f\mid_{\iu_\ip}$ is.
\item[(3)]
$f$ is surjective if and only if
$f\mid_{\iu_\ip}$ is.
\item[(4)]
$f$ is an isomorphism if and only if
$f\mid_{\iu_\ip}$ is.
\end{enumerate}
\end{corollary}

\begin{proof} 
Since $f$ is a morphism of operads, it follows from
Lemma \ref{xxlem3.3} that $f$ maps ${^k \iu_\ip}$ to
${^k \iu}_{\iq}$ for every $k$. Consequently,
$f$ maps ${^k \iu_\ip}(k)$ to
${^k \iu_\ip}_{\iq}(k)$ for every $k$. This defines a
map $f:  \iu_\ip \to \iu_{\iq}$.
Since $f$ preserves
$\1_i$ for $i=0, 1, 2$, it preserves $\1_n$ for all $n$.
Therefore $f$ maps $\1_2\circ
(\theta, \1_{n-k})\ast c_{I}$ to $\1_2\circ
(f(\theta), \1_{n-k})\ast c_{I}$
for all $\theta\in {^k \iu_\ip}(k)$.

(1) Since $f$ is a morphism of 2-unitary operads,
$\ip$ is generated by elements in $\Theta^k$ for $k\geq 0$
by Theorem \ref{xxthm4.6} (1). The assertion follows.

(2) Suppose $f$ is not injective. Let $\I$ be the nonzero
kernel $\ker f$. Then $\I$ is an ideal of
$\ip$. Since $\I\neq 0$, $\I\cap {^k \iu_\ip}\neq \I\cap {^{k+1} \iu_\ip}$
for some $k$. Let $x\in (\I\cap {^k \iu_\ip})\setminus (\I\cap {^{k+1} \iu_\ip})$.
Then there is an $I$ with $|I|=k$ such that $0\neq \pi^I(x) \in
\I\cap {^k \iu_\ip}(k)$. So $f\mid_{\iu_\ip}$ is
not injective. The converse is easy.

(3) Suppose $f\mid_{\iu_\ip}$ is surjective.
Since $\iq$ is generated by $\bigoplus_{k\geq 1}{^k \iu_{\iq}(k)}$
by Theorem \ref{xxthm4.6} (2), $f$ is surjective.

Conversely, suppose that $f$ is surjective. Then $\iq$ is a quotient
operad of $\ip$. By Lemma \ref{xxlem4.7}, 
$f$ maps surjectively from ${^k \iu_\ip(k)}$ to ${^k\iu_{\iq}(k)}$ for each $k$.
The assertion follows.

(4) This is a consequence of parts (2) and (3).
\end{proof}

\subsection{Categorification of binomial coefficients}
\label{xxsec4.3}

Following the basis theorem [Theorem \ref{xxthm4.6} (1)], for each
$I\subseteq \n$ with $|I|= n-k$, we define a $\Bbbk$-linear
map
\begin{equation}
\label{E4.8.1}\tag{E4.8.1}
\Lambda_{I}^n: {^k \iu(k)}\to {^k \iu(n)}
\end{equation}
by
\begin{equation}
\label{E4.8.2}\tag{E4.8.2}
\Lambda_I^n(\theta)=\1_2\circ (\theta, \1_{n-k})\ast c_{I}.
\end{equation}

\begin{lemma}
\label{xxlem4.9}
Retain the above notation.
For every $n\geq k$ and every $\sigma\in \S_n$, the
following diagram is commutative in
the quotient space ${^k \iu}/{^{k+1} \iu}$
$$\begin{CD}
{^k\iu(k)} @>\Lambda_{I}^n >> {^k \iu(n)} \\
@V \ast \Gamma^{\sigma^{-1}(I)}(\sigma) VV @VV \ast \sigma V \\
{^k\iu(k)} @>>\Lambda_{\sigma^{-1}(I)}^n > {^k \iu(n)} .
\end{CD}$$
As a consequence, if ${^{k+1} \iu}=0$, then
$$\Lambda_I^n(\theta)\ast \sigma=
\Lambda^n_{\sigma^{-1}(I)}(\theta\ast \Gamma^{\sigma^{-1}(I)}(\sigma)).$$
\end{lemma}

\begin{proof} Let $\theta$ be an element in
${^k \iu(k)}$. For every $I'\subseteq \n$
with $|I'|=n-k$, by Lemma \ref{xxlem2.11} and Corollary
\ref{xxcor4.4},
$$\Gamma^{I'}(\Lambda_I^n(\theta)\ast \sigma)
=\Gamma^{\sigma(I')}(\Lambda_I^n(\theta))\ast \Gamma^{I'}(\sigma)
=\begin{cases} \theta\ast \Gamma^{I'}(\sigma), & \sigma(I')=I,\\
0,& \sigma(I')\neq I.\end{cases}$$
and
$$\Gamma^{I'}(\Lambda^n_{\sigma^{-1}(I)}(\theta\ast \Gamma^{\sigma^{-1}(I)}(\sigma)))
=\begin{cases}\theta\ast \Gamma^{\sigma^{-1}(I)}(\sigma), & \sigma(I')=I,\\
0,& \sigma(I')\neq I.\end{cases}$$
Thus $\Gamma^{I'}(\Lambda_I^n(\theta)\ast \sigma)=
\Gamma^{I'}(\Lambda^n_{\sigma^{-1}(I)}(\theta\ast \Gamma^{\sigma^{-1}(I)}(\sigma)))$
for all $I'$. Therefore
$$\Lambda_I^n(\theta)\ast \sigma=
\Lambda^n_{\sigma^{-1}(I)}(\theta\ast \Gamma^{\sigma^{-1}(I)}(\sigma))$$
modulo ${^{k+1}\iu}$. The assertion follows.
\end{proof}

Let $\rMod \S_n$ denote the category of right $\Bbbk \S_n$-modules.
Suggested by Lemma \ref{xxlem4.9}, we define the following functor
$${\mathcal C}^{n}_{k}: \rMod \S_{k}\to \rMod \S_{n}$$
for $n\geq k$ as follows. Let $T^n_k$ be the set $\{I\subset \n\, \mid |I|=n- k\}$.
Let $M$ be a right $\S_k$-module. Then ${\mathcal C}^n_{k}(M)$ is a right $\S_n$-module
such that
\begin{enumerate}
\item[(i)]
as a vector space, ${\mathcal C}^n_{k}(M)=\bigoplus\limits_{I\in T^n_k} M$,
elements in ${\mathcal C}^n_k(M)$ are linear combinations of $(m, I)$ for $m\in M$ and
$I\in T^n_k$;
\item[(ii)]
the $\S_n$-action on ${\mathcal C}^n_{k}(M)$ is determined by
$$ (m,I)\ast \sigma:=(m\ast \Gamma^{\sigma^{-1}(I)}(\sigma), \sigma^{-1}(I))$$
for all $(m,I)\in {\mathcal C}^n_{k}(M)$ and all $\sigma\in \S_n$.
\end{enumerate}

\begin{lemma}
\label{xxlem4.10}
Retain the notation as above.
\begin{enumerate}
\item[(1)]
If $M$ is a right $\S_{k}$-module, then ${\mathcal C}^n_k(M)$ is a
right $\S_n$-module.
\item[(2)]
Let $A$ be an algebra.
If $M$ is an $(A,\S_{k})$-bimodule, then ${\mathcal C}^n_k(M)$ is an
$(A,\S_n)$-bimodule.
\item[(3)]
The functor ${\mathcal C}^{n}_{k}(-)$ is equivalent to the tensor
functor $-\otimes_{\S_k} {\mathcal C}^{n}_{k}(\S_k)$.
\end{enumerate}
\end{lemma}

\begin{proof}
(1) For $\sigma,\tau\in \S_n$, and $(m,I)\in {\mathcal C}^n_k(M)$,
$$\begin{aligned}
((m,I)\ast \sigma)\ast \tau
&=(m\ast \Gamma^{\sigma^{-1}(I)}(\sigma), \sigma^{-1}(I))\ast \tau\\
&=((m\ast \Gamma^{\sigma^{-1}(I)}(\sigma))\ast
\Gamma^{\tau^{-1}\sigma^{-1}(I)}(\tau), \tau^{-1}(\sigma^{-1}(I)))\\
&=((m\ast \Gamma^{\sigma^{-1}(I)}(\sigma))\ast
\Gamma^{(\sigma\tau)^{-1}(I)}(\tau), (\sigma\tau)^{-1}(I))\\
&=(m\ast (\Gamma^{\tau(\sigma\tau)^{-1}(I)}(\sigma)\ast
\Gamma^{(\sigma\tau)^{-1}(I)}(\tau)), (\sigma\tau)^{-1}(I))\\
&=(m\ast (\Gamma^{(\sigma\tau)^{-1}(I)}(\sigma\ast \tau)), (\sigma\tau)^{-1}(I))\\
&=(m\ast (\Gamma^{(\sigma\tau)^{-1}(I)}(\sigma \tau)), (\sigma\tau)^{-1}(I))\\
&=(m,I)\ast (\sigma \tau).
\end{aligned}
$$

(2) This follows from the definition and part (1).

(3) This follows from the Watts Theorem and the fact that
${\mathcal C}^n_k$ is exact.
\end{proof}

\section{Binomial transform of generating series}
\label{xxsec5}
In this section we study 2-unitary operads of finite
Gelfand-Kirillov dimension. One tool is binomial transform
\cite{Ku, Pr, SS} of generating series that is closely related
to truncation ideals of 2-unitary operads.

\subsection{Binomial transform}
\label{xxsec5.1}
First of all, there are at least two versions of binomial 
transforms, We will use the following version. We also list 
some facts without proofs.

Let $a:=\{a_0,a_1,a_2,\cdots\}$ be a sequence of numbers.
Its generating series is denoted by
$$G_{a}(t)=\sum_{i=0}^{\infty} a_i t^i$$
and
its exponential generating series is
$$E_{a}(t)=\sum_{i=0}^{\infty} \frac{a_i}{i!} t^i.$$
The {\it binomial transform} of $a$ is a sequence
$b:=\{b_0,b_1,b_2,\cdots,\}$ defined by
\begin{equation}
\label{E5.0.1}\tag{E5.0.1}
b_i=\sum_{k=0}^{i} a_k (-1)^{i-k} \tbinom{i}{k}
\end{equation}
for all $i\geq 0$. It is well-known (see \cite{Pr}) that
\begin{equation}
\label{E5.0.2}\tag{E5.0.2}
a_i=\sum_{k=0}^{i} b_k \tbinom{i}{k}
\end{equation}
for all $i\ge 0$, and
\begin{equation}
\label{E5.0.3}\tag{E5.0.3}
G_{a}(t)=\frac{1}{1-t} \; G_{b}(\frac{t}{1-t}), \qquad
G_{b}(t)=\frac{1}{1+t} \; G_{a}(\frac{t}{1+t})
\end{equation}
and
\begin{equation}
\label{E5.0.4}\tag{E5.0.4}
E_{a}(t)=e^{t} E_{b}(t), \qquad E_{b}(t)=e^{-t}E_{a}(t).
\end{equation}

Note that \eqref{E5.0.3} is equivalent to
\begin{equation}
\label{E5.0.5}\tag{E5.0.5}
\sum_{k=0}^{\infty} a_k t^k=\sum_{k=0}^{\infty} b_k \; \frac{t^{k}}{(1-t)^{k+1}}.
\end{equation}
We also write
$$\tt(\{a_i\})=\{b_i\}, \quad {\text{ and }} \quad
\tt^{-1}(\{b_i\})=\{a_i\},$$
or
$$\tt(\sum_{k=0}^{\infty} a_k t^k)= \sum_{k=0}^{\infty} b_k
t^k,
\quad {\text{ and }} \quad \tt^{-1}(\sum_{k=0}^{\infty} b_k t^k)=
\sum_{k=0}^{\infty} a_k t^k,$$
where $\{a_k\}_{k\geq 0}$ and $\{b_k\}_{k\geq 0}$ are determined
by each other via \eqref{E5.0.1}-\eqref{E5.0.2}, and in this case we call
$a=\{a_i\}$ the {\it inverse binomial transform} of $b=\{b_i\}$.
For a sequence of non-negative numbers (called {\it a
non-negative sequence}) $a=\{a_i\}$, define the {\it exponent} of
$a$ to be
\begin{equation}
\label{E5.0.6}\tag{E5.0.6}
\exp(a):=\limsup_{n\to \infty} a_n^{\frac{1}{n}}.
\end{equation}
When $\{a_n\}$ is a sequence of non-negative integers with infinitely
many nonzero $a_n$'s, then by \cite[Lemma 1.1(1)]{StZ},
\begin{equation}
\label{E5.0.7}\tag{E5.0.7}
\exp(a)=\limsup_{n\to \infty} \left(\sum_{i=0}^n a_i\right)^{\frac{1}{n}}.
\end{equation}

\begin{lemma}
\label{xxlem5.1}
Let $b:=\{b_n\}$ be a non-negative sequence with $b_0=1$ and
$a:=\{a_n\}={\mathcal T}^{-1}(\{b_n\})$.
\begin{enumerate}
\item[(1)]
$\exp(a)=\exp(b)+1$.
\item[(2)]
If $b_n=0$ for $n\gg 0$, then
$\exp(a)=1$.
\item[(3)]
For every real number $r\geq 2$, let $b=\{ \lfloor (r-1)^n \rfloor\}$,
then $\exp(a)=r$.
\end{enumerate}
\end{lemma}

\begin{proof} First of all $\exp(a)\geq 1$ since $a_n\geq 1$ for each $n$.
From calculus, the radius of convergence of the power series
$G_{a}(t)$ is $ r_a:=\exp(a)^{-1}$. The same is true for $b$.

(1) By \eqref{E5.0.3}, $r_a^{-1}=r_b^{-1}+1$. The assertion follows.

(2) Since $\exp(b)=0$, this is a special case of (1).

(3) Clearly $\exp(b)=r-1$. The assertion follows from part (1).
\end{proof}

Next we apply binomial transform to operads.
Let $\ip$ be a 2-unitary operad and let $\I$ be an ideal of
$\ip$ or $\I=\ip$. Let $^{n}\iu$ be defined as \eqref{E3.0.1} and
let $^{0}\iu=\ip$. Let
\begin{equation}
\label{E5.1.1}\tag{E5.1.1}
G_{\I}^{w}(t)=\sum_{n=0}^{\infty} \dim_{\k}(\frac{\I}{{^w \iu}\cap \I}(n)) t^n=
\sum_{n=0}^{\infty} d_{\I}^w(n)t^n
\end{equation}
and
\begin{equation}
\label{E5.1.2}\tag{E5.1.2}
G_{\I}(t)=\sum_{n=0}^{\infty} \dim_{\k}(\I(n)) t^n.
\end{equation}

\begin{lemma}
\label{xxlem5.2}
Let $\ip$ be a 2-unitary operad and let $\I$ be an ideal of $\ip$ or $\I=\ip$.
Then $G_{\I}^w(t)$ and $G_{\I}(t)$ are
\begin{equation}
\label{E5.2.1}\tag{E5.2.1}
G_{\I}^w(t)=\sum_{k=0}^{w-1} f_{\I}(k) \; \frac{t^{k}}{(1-t)^{k+1}}
\end{equation}
for all $w$ and
\begin{equation}
\label{E5.2.2}\tag{E5.2.2}
G_{\I}(t)=\sum_{k=0}^{\infty} f_{\I}(k) \; \frac{t^{k}}{(1-t)^{k+1}}
\end{equation}
where $f_{\I}(k)=d_{\I}^{k+1}(k)-d_{\I}^{k}(k)$ for all $k$.
\end{lemma}

\begin{proof} Since $G_{\I}(t)=\lim_{w\to \infty} G_{\I}^w(t)$,
\eqref{E5.2.2} is a consequence of \eqref{E5.2.1}.
So we only prove \eqref{E5.2.1}.

By Theorem \ref{xxthm4.5}(2), we have
$$G_{\I}^{w}(t)=G_{\I}^{w-1}(t) + f_{\I}(w-1) \frac{t^{w-1}}{(1-t)^{w}}$$
for all $w\geq 1$. When $w=1$, the above equation becomes $0=0+0$
where $\I\neq \ip$, or $\sum_{i=0}^{\infty} t^i=0+\frac{1}{1-t}$ where $\I=\ip$,
both of which hold clearly. We have
$$
G_{\I}^w(t)=\sum_{k=1}^{w} (G_{\I}^{k}(t)-G_{\I}^{k-1}(t))
=\sum_{k=0}^{w-1} f_{\I}(k) \frac{t^{k}}{(1-t)^{k+1}}.
$$
\end{proof}

Lemma \ref{xxlem5.2} tells us that the sequence $\{f_\I(n)\}$ is the
binomial transform of $\{\dim_\I(n)\}$. By Definition \ref{xxdef4.1}(6)
and \eqref{E5.0.5}, we immediately get
\begin{equation}
\label{E5.2.3}\tag{E5.2.3}
\gkdim \I=\max\{k\mid  f_{\I}(k)\neq 0\}+1.
\end{equation}

\subsection{Proofs of some main results}
\label{xxsec5.2}
We are ready to prove Theorem \ref{xxthm0.1}.

\begin{proof}[Proof of Theorem \ref{xxthm0.1}]
If ${^k \iu}=0$, then $\ip$ has finite GKdimension by
Lemma \ref{xxlem4.2}. Conversely, we assume that
$\gkdim \ip<\infty$. By Lemma \ref{xxlem5.2},
\begin{equation}
\label{E5.2.4}\tag{E5.2.4}
G_{\ip}(t)=\sum_{n=0}^{\infty} f_{\ip}(n) \;
\frac{t^{n}}{(1-t)^{n+1}}.
\end{equation}
By definition, $f_{\ip}(n)\geq 0$ for all $n$. Since
$\gkdim \ip<\infty$, there is an $N\in \N$ such that
$f_\ip(n)=0$ for all $n\geq N$ where
$f_\ip(n)=d^{n+1}_\ip(n)-d_\ip^{n}(n)=\dim {^{n}\iu}(n)$.
This implies that ${^n \iu}(n)=0$ for all $n\geq N$.
By Theorem \ref{xxthm4.6}(1), ${^n \iu}=0$ for all $n\geq
N$. In this case,
$$G_{\ip}(t)=\sum_{n=0}^{N-1} f_{\ip}(n) \;
\frac{t^{n}}{(1-t)^{n+1}}$$
which is rational. It is clear and follows from \eqref{E5.2.3} that
\begin{equation}
\label{E5.2.5}\tag{E5.2.5}
\gkdim \ip=\max\{ n\mid f_{\ip}(n)\neq 0\}+1
=\max\{n \mid {^n \iu}\neq 0\}+1.
\end{equation}
Therefore assertions in parts (1) and (2) follow.
\end{proof}

\begin{proof}[Proof of Corollary \ref{xxcor0.2}]
By Theorem \ref{xxthm0.1} (2) and Lemma \ref{xxlem4.7},
$\gkdim \ip\leq k$ if and only if ${^k \iu}_{\ip}=0$
if and only if $\I\supseteq {^k \iu}_{\As}$.

By definition,
$$G_{\As}(t)=\sum_{n=0}^{\infty} f_{\As}(n) \frac{t^n}{(1-t)^{n+1}},$$
and
\begin{equation}
\label{E5.2.6}\tag{E5.2.6}
G_{\As/{^k \iu}}(t)=G^k_{\As}(t)=\sum_{n=0}^{k-1} f_{\As}(n) \frac{t^n}{(1-t)^{n+1}}.
\end{equation}
Since $G_{\As}(t)=\sum_{k=0}^{\infty} k! t^k$, by
\eqref{E5.0.1},
$$f_{\As}(n)=\sum_{s=0}^{n} (-1)^{n-s} s! \tbinom{n}{s} .$$
It is easy to check that $f_{\As}(n)\neq 0$ for all $n\neq 1$.
The assertion concerning the GKdimension of
$\As/{^k \iu}$ follows from \eqref{E5.2.6}.
\end{proof}

Let $\sum_{k} a_k t^k$ and $\sum_k b_k t^k$ be two power series.
If $a_k\leq b_k$ for all $k$, then we write $\sum_{k} a_k t^k
\leq \sum_k b_k t^k$.

\begin{lemma}
\label{xxlem5.3}
Let $\I$ be an ideal of $\ip$. Then
$\tt(G_{\I}(t))\leq \tt(G_{\ip}(t))$.
As a consequence, if ${^n \iu}=0$
for some $n$, the set $\{G_{\I}(t)\mid \I\subseteq \ip\}$
is finite.
\end{lemma}

\begin{proof}
By Theorem \ref{xxthm4.5}(1) and Lemma \ref{xxlem5.2},
$$G_{\I}(t)=\sum_{k=0}^{\infty} f_{\I}(k) \; \frac{t^{k}}{(1-t)^{k+1}}$$
where $f_{\I}(k)=\dim ({^{k}\iu}\cap \I)(k)$. Since
$$0\leq f_{\I}(k)=\dim ({^{k}\iu}\cap \I)(k)\leq \dim {^{k}\iu}(k)=
f_{\ip}(k)$$
for all $k$, we have
$$0\leq \tt(G_{\I}(t))=\sum_{k=0}^{\infty} f_{\I}(k) t^k\leq
\sum_{k=0}^{\infty} f_{\ip}(k) t^k=\tt(G_{\ip}(t)).$$

If ${^n \iu}=0$, then there are only finitely many
nonzero $f_{\ip}(k)$. Therefore there are only
finitely many possible choices $\{f_{\I}(k)\}_{k\geq 0}$.
The assertion follows.
\end{proof}

\begin{proof}[Proof of Theorem \ref{xxthm0.3}]
(1)
Let \[\I_1\subseteq \I_2 \subseteq \cdots \subseteq
\I_n\subseteq \cdots\]
be an ascending chain of ideals of $\ip$. Then we have
\[G_{\I_1}(t) \leq G_{\I_2}(t) \leq \cdots \leq
G_{\I_n}(t) \leq \cdots\]
Since $\gkdim\ip<\infty$, we have ${^k\iu}=0$ for some $k$.
By Lemma \ref{xxlem5.3}, $\{G_{\I_i}\mid i=1, 2, \cdots \}$
is finite and therefore the sequence $\{G_{\I_i}(t)\}_{i\ge 1}$ stabilizes.
This implies that the sequence of ideals $\{\I_i\}_{i\ge 1}$ stabilizes.

(2) ($\Rightarrow$) The proof is similar to the proof of
part (2) above.
Let \[\I_1\supseteq \I_2 \supseteq \cdots \supseteq \I_n\supseteq \cdots\]
be a descending chain of ideals of $\ip$. Since $\gkdim\ip<\infty$, we have
${^k\iu}=0$ for some $k$. By Lemma \ref{xxlem5.3}, $\{G_{\I_i}(t)\mid i=1, 2, \cdots \}$
is finite and therefore the sequence $\{G_{\I_i}(t)\}_{i\ge 1}$ stabilizes.
This implies that the sequence of ideals $\{\I_i\}_{i\ge 1}$ stabilizes.

($\Leftarrow$)
By Proposition \ref{xxpro3.1} and Lemma \ref{xxlem2.4}(1), we have
a descending chain of ideals
\[ {^0\iu} \supseteq {^1\iu} \supseteq {^2\iu} \supseteq
\cdots \supseteq {^k\iu} \supseteq \cdots\]
 of $\ip$. If $\ip$ is artinian, then this chain is stable.
On the other hand, we have $\cap_{k\ge 0} {^k\iu}=0$ since
${^k\iu}(k'-1)=0$ for all $k\ge k'\geq 0$.
It follows that ${^k\iu}=0$ for some sufficiently large $k$
and hence $\ip$ has finite GKdimension.

(3) This is a consequence of parts (1) and (2).
\end{proof}

Next we prove Theorem \ref{xxthm0.9}. The Gelfand-Kirillov 
dimension (or GKdimension) is an important tool in the
study of noncommutative algebra \cite{McR, KL}. Similar to
associative algebras, we introduce the notion of the
GKdimension for algebras over any operad.

\begin{definition}
\label{xxdef5.4}
Let $\ip$ be an operad and $A$ a $\ip$-algebra.
\begin{enumerate}
\item[(1)]
Let $X$ be a subset of $A$. We say that $X$ is a
\emph{set of generators} of $A$ if $A= \sum_{n\ge 0}\gamma_n(X)$,
where $\gamma_n(X)$ denotes the image $\ip(n)\otimes (\k X)^{\ot n} \to A$.
\item[(2)]
We say $A$ is \emph{finitely generated} if it has a set of generators
which is finite.
\item[(3)]
The \emph{GKdimension} of $A$ is defined to be
\[\gkdim(A) = \sup_{V}
\left\{\limsup_{n\to \infty}\log_n\left(\dim\sum_{i=0}^n\gamma_i(V)\right)\right\},\]
where the sup is taken over all finite dimensional subspaces $V\subseteq A$.
\end{enumerate}
\end{definition}

\begin{remark}
\label{xxrem5.5}
If $A$ is an associative algebra, then the above defined notions of
generators and GKdimension coincide with the standard ones in
\cite{McR, KL}.
\end{remark}

The next result is Theorem \ref{xxthm0.9}. In this theorem, $\ip$
might not be $2$-unitary.

\begin{theorem}
\label{xxthm5.6}
Let $\ip$ be an operad with order of the growth $o(\ip)$ {\rm{(}}see
Definition \ref{xxdef4.1}(4){\rm{)}} and $A$
an algebra over
$\ip$ with a finite set $X$ of generators. Then $A$ has finite
GKdimension, precisely
\[\gkdim(A) \le o(\ip) + r,\]
where $r= |X|$ is the cardinality of $X$.
If, the generating series of $\ip$ is rational, then
$$\gkdim (A)\le \gkdim (\ip)-1+r.$$
\end{theorem}

\begin{proof}
First we claim that
$$\dim (\ip(n)\otimes_{\S_n} (\k X)^{\ot n}) \le \dim\ip(n)\cdot
\tbinom{n+r-1}{r-1}.$$
In fact, assume $X=\{x_1, \cdots, x_r\}$. We define a total ordering on
$X$ by $x_1<x_2<\cdots<x_r$. For any $x_{i_1}, \cdots, x_{i_n}$ with
$1\le i_1, \cdots, i_n\le r$, there exists some $\sigma\in\S_n$ such
that $i_{\sigma(1)}\le \cdots \le i_{\sigma(n)}$, thus
$$\theta\ot (x_{i_1}\ot\cdots \ot x_{i_n}) =
(\theta\ast \sigma)\ot \sigma^{-1}\ast(x_{i_1}\ot\cdots \ot x_{i_n})
=(\theta\ast \sigma)\ot (x_{i_{\sigma(1)}}\ot \cdots\ot  x_{i_{\sigma(n)}}).$$
By a standard argument we obtain the desired inequality.

Consequently, we have
\begin{equation}
\label{E5.6.1}\tag{E5.6.1}
\dim(\gamma_n(\Bbbk X))\le \dim\ip(n)\cdot \tbinom{n+r-1}{r-1}.
\end{equation}
By Definition \ref{xxdef4.1}(4), for any arbitrary small positive
number $\epsilon$, there is a positive number $C$ such that
\begin{equation}
\label{E5.6.2}\tag{E5.6.2}
\dim \ip(n) \leq C n^{o(\ip)+\epsilon}
\end{equation}
for all $n\geq 1$.
Now let $V$ be any finite dimensional subspace of $A$. Since
$X$ is a generating set of $A$,
$V\subseteq \sum_{i=0}^m \gamma_i(\Bbbk X)$
for some integer $m\ge 1$. Then we have
\begin{align*}
\gamma_n(V)\subseteq & \sum\limits_{0\le i_1, \cdots, i_n\le m}
\gamma(\ip(n)\otimes ((\gamma_{i_1}(\Bbbk X)^{i_1})
\otimes \cdots (\gamma_{i_n}(\Bbbk X)^{i_n})))\\
=& \sum\limits_{0\le i_1, \cdots, i_n\le m}
\gamma((\ip(n)\circ( \ip(i_1), \cdots, \ip(i_n)))
\otimes (\Bbbk X)^{i_1+\cdots+i_n})\\
\subseteq & \sum\limits_{0\le i_1, \cdots, i_n\le m}
\gamma_{i_1+\cdots+i_m} (\Bbbk X)
\subseteq  \sum\limits_{i=0}^{mn} \gamma_i(\Bbbk X),
\end{align*}
and hence
$\sum_{i=0}^n \gamma_i(V) \subseteq \sum_{i=0}^{mn}\gamma_i(\Bbbk X)$.
Combining with \eqref{E5.6.1} and \eqref{E5.6.2}, we have
$$\begin{aligned}
\dim (\sum_{i=0}^{n} \gamma_i(V))
&\leq \dim (\sum_{i=0}^{mn}\gamma_i(\Bbbk X))
\leq \sum_{i=0}^{mn} \dim \ip(i)  \tbinom{i+r-1}{r-1}\\
&\leq \sum_{i=0}^{mn} C i^{o(ip)+\epsilon} \tbinom{i+r-1}{r-1}
\leq \sum_{i=0}^{mn} C_1 i^{o(\ip)+\epsilon+r-1}\\
&\leq C_2 n^{o(\ip)+\epsilon+r}
\end{aligned}
$$
for some constants $C_1$ and $C_2$. Thus
$\gkdim (A)\leq o(\ip)+r$ when taking $\epsilon$ arbitrary small.

When $\ip$ has rational generating series,
one can easily check that $\gkdim \ip=o(\ip)+1$. Thus
$\gkdim (A)\leq \gkdim (\ip)-1+r$.
\end{proof}

\begin{example}
\label{xxex5.7}
Consider the operad $\Com$. It is easy to check that
$\gkdim\Com=1$, and consequently, any commutative
algebra generated in $n$-elements has GKdimension no greater
than $n$. Notice that the free algebra generated in $n$ elements
over $\Com$ is the polynomial algebra in $n$ variables,
and has GKdimension $n$.
\end{example}

Recall that $(GK\leq k)\rad(\ip)$ is defined in Definition
\ref{xxdef1.10}(2).

\begin{proposition}
\label{xxpro5.8}
Let $\ip$ be a 2-unitary operad. Then
$(GK\leq k)\rad(\ip)={^{k} \iu}$.
\end{proposition}

\begin{proof} Let $\I$ be an ideal of $\ip$ and let
$\ip'=\ip/\I$. If $\gkdim \ip' \leq k$, then
$^{k} \iu_{\ip'}=0$ by Theorem \ref{xxthm0.1}(2).
By Lemma \ref{xxlem4.7}, $^{k}\iu_{\ip}\subseteq \I$.
By definition, $^{k}\iu_{\ip} \subseteq
(GK\leq k)\rad(\ip)$. For the other inclusion,
note that,
$$\gkdim \ip/{^{k}\iu}\leq k$$
by Lemmas \ref{xxlem4.2} and \ref{xxlem4.7}. Hence
${^{k}\iu}\supseteq (GK\leq k)\rad(\ip).$
\end{proof}

\section{Signature of a 2-unitary operad}
\label{xxsec6}

In this section we introduce the notion of the
signature of an unitary operad. Then we prove
Theorems \ref{xxthm0.6}, \ref{xxthm0.7} and \ref{xxthm0.8}.
Note that we do not usually
assume that $\ip$ is locally finite.

\subsection{Definition of the signature}
\label{xxsec6.1}

\begin{definition}
\label{xxdef6.1}
Let $\ip$ be a unitary operad. The {\it signature} of
$\ip$ is defined to be the sequence
$${\mathcal S}(\ip):=\{d_1,d_2,d_3,\cdots\}$$
where
$$d_{k}=\dim_{\Bbbk} {^{k}\iu}(k)$$
for all $k\geq 1$. We leave out $d_0=
\dim_{\Bbbk} {^{0}\iu}(0)$ because it is always 1.
\end{definition}

We borrow the word ``signature'' from a paper of Brown-Gilmartin
\cite[Definition 5.3(1)]{BG}. There are some similarities
between the signature of a connected Hopf algebra in the
sense of \cite{BG} and the signature of a 2-unitary operad
defined above.

The signature of $\Com$ is $\{0,0,0,\cdots\}$.
Let $\ip$ be a 2-unitary operad of GKdimension $k$. By
\eqref{E5.2.5}, we have the signature of $\ip$ is of form
\[\{f_\ip(1), \cdots, f_\ip(k-1), 0, 0, \cdots \}\]
where $f_\ip(k-1)\neq 0$, and
\[\dim \ip(n)=\sum\limits_{i=0}^{k-1} f_\ip(i)\tbinom{n}{i},\]
where ${n\choose i}=0$ if $n<i$. Thus the signature of
$\ip$ is uniquely determined by the Hilbert series
of $\ip$, and vice versa.

\subsection{Proofs of other main results}
\label{xxsec6.2}
Theorem \ref{xxthm0.6} classifies all 2-unitary operads with
signature $(d,0,\cdots)$ for any $d\geq 0$. We start with the
following lemma.

\begin{lemma}
\label{xxlem6.2}
Let $\ip$ be a 2-unitary operad or a 2-unitary plain operad.
Suppose that ${^2 \iu}=0$. Then
\begin{enumerate}
\item[(1)]
$\ip$ is $2a$-unitary.
\item[(2)]
$\ip$ is $\Com$-augmented, namely, there is a morphism
from $\Com\to \ip$.
\end{enumerate}
\end{lemma}

\begin{proof}
Note that ${^2\iu}=0$ means that, for each $\theta\in \ip(n)$,
if $\pi^i(\theta)=0$ for all $i\in \n$, then $\theta=0$.

(1) It is easily seen that
$\pi^i(\1_3)=\1_1=\pi^i(\1'_3)$ for $i=1, 2, 3$.
So $\1_3-\1'_3\in \ker \pi^i$ and $\1_3-\1'_3\in {^2\iu}(3)=0$.
The assertion follows.

(2) We claim that $\1_2*\tau=\1_2$ where $\tau=(12)\in \S_2$.
The proof is similar to the proof of part (1) by using the
fact that $\pi^i(\1_2*\tau) =\1_1=\pi^i(\1_2)$ for $i=1,2$.
It follows by induction on $n$ that, for every $n\geq 1$,
$\1_n*\sigma=\1_n$ for all $\sigma\in \S_n$. Thus there is an
operad morphism from $\Com$ to $\ip$ by sending $\1_n\in \Com$
to $\1_n\in \ip$.
\end{proof}

\begin{proof}[Proof of Theorem \ref{xxthm0.6}]
If $\ip$ is a 2-unitary operad of $\gkdim\leq 2$, then, by
Theorem \ref{xxthm0.1}, $^2 \iu=0$. Hence Lemma \ref{xxlem6.2}
can be applied. In particular, two categories in Theorem
\ref{xxthm0.6}(2) and (3) are the same. We now show that
two categories in Theorem \ref{xxthm0.6}(1) and (2) are
equivalent.

Suppose $A$ is a finite dimensional augmented algebra.
By Example \ref{xxex2.3}(1), one can construct an operad,
denoted by ${\mathcal D}_A$. It is routine to check
that every algebra morphism $f: A\to A'$ induces
a natural morphism of operads ${\mathcal D}_{A}
\to {\mathcal D}_{A'}$. Thus $F: A\to {\mathcal D}_{A}$
is a functor from the category of finite dimensional
augmented algebras to the category of 2-unitary operads
of $\gkdim\leq 2$.

Let $\mathcal{D}$ be a 2-unitary operad of GKdimension $\leq 2$.
Observe that $\mathcal{D}(1)$ is an associative $\Bbbk$-algebra
with identity $\1_1$. The map $\pi^{\emptyset}\colon {\mathcal D}(1)\to
{\mathcal D}(0)=\Bbbk$ shows that $\mathcal{D}(1)$ is augmented.
The restriction $G: {\mathcal D}\to {\mathcal D}(1)$ defines
a functor from the category of 2-unitary operads
of $\gkdim\leq 2$ to the category of finite dimensional
augmented algebras. It is clear that $G\bullet F\cong Id$.
It remains to show that $F\bullet G\cong Id$.

If $\gkdim \mathcal{D}=1$, then $\mathcal{D}=\Com$ by Proposition
\ref{xxpro0.5}. In this case we have $F\bullet G(\mathcal{D})=\mathcal{D}$.
For the rest of the proof we assume that $\gkdim \mathcal{D}=2$.
By Theorem 0.1, we have that $f_{\mathcal{D}}(1)\neq 0$ (or
${^1\iu_\mathcal{D}}\neq 0$) and $f_{\mathcal{D}}(n)=0$ (or
${^n\iu_\mathcal{D}}=0$) for all $n\ge 2$. Recall that
$f_\mathcal{D}(0)=\dim{^{0}\iu_{\mathcal{D}}(0)}=\dim\mathcal{D}(0)=1$.
Suppose
$f_\mathcal{D}(1)=\dim{^1\iu_\mathcal{D}(1)}=\dim\mathcal{D}(1)-1=d>0$.
Then by Lemma \ref{xxlem5.2}, we know that the generating series
of $\mathcal{D}$ is
\begin{align*}
G_\mathcal{D}(t)= f_\mathcal{D}(0)\dfrac{1}{1-t}+f_{\mathcal{D}}(1)\dfrac{t}{(1-t)^2}
= \sum_{n=0}^\infty (1+nd)t^n.
\end{align*}

Since ${^1\iu}(1)$ is the kernel of
the $\Bbbk$-linear map $\pi^{\emptyset}\colon {\mathcal D}(1)\to
{\mathcal D}(0)$
(sending $\theta\mapsto \theta\circ \1_0$), we can choose a
$\Bbbk$-basis $\1_1, \delta_1, \cdots, \delta_d$ for
$\mathcal{D}(1)$ with $\delta_i\circ \1_0=0$ for all $i=1,
\cdots, d$.

The claim that $F\bullet G\cong \Id$ is equivalent to the claim
that $\mathcal{D}$ is naturally isomorphic to the operad introduced
in Example \ref{xxex2.3}(1). We separate the proof into
several steps.

\noindent
{\bf Step 1:}
Denote $\delta^{n}_{(i)j}\colon =\1_n\underset{i}{\circ} \delta_j$.
We claim that $\{\1_n, \delta^n_{(i)j}\mid i\in \n,
j\in \d \}$ is a basis for $\mathcal{D}(n)$.
In fact, since $\dim\mathcal{D}(n)=1+nd$, we only need show
$\{\1_n, \delta^n_{(i)j} \mid i\in \n, j\in \d \}$ are
linearly independent. Assume that there exist $\{\la_0, \la_{ij}
\in \Bbbk
\mid i\in \n, j\in \d \}$ such that
$\la_0\1_n+\sum\limits_{i, j}\la_{ij}\delta^n_{(i)j}=0$.
Then we have
\begin{align*}
0=\pi^k(\la_0\1_n+\sum\limits_{i, j}\la_{ij}\delta^n_{(i)j})
=\la_0\1_1+\sum\limits_{j}\la_{kj}\delta_j
\end{align*}
since $\pi^k(\delta^n_{(i)j})=
\begin{cases} \delta_j, & i=k \\ 0, & i\neq k.\end{cases}$
It follows that $\la_0=0$ and $\la_{ij}=0$ for all $i, j$.
Therefore we proved our claim.

\noindent
{\bf Step 2:}
For consistency, we set $\delta_0=\1_1$, and
$\delta_{(i)0}^n=\1_n$ for any $n\ge 1$ and any
$i\in \n$.
For other $i\in \n, 0\le j\le d, n\ge 1$, we have
$\delta^{n}_{(i)j} =\1_n\underset{i}{\circ} \delta_j$ by definition.
Next, we compute $\delta^m_{(s)t}\underset{i}{\circ}
\delta^n_{(k)l}$ for all possible $m, s, t, i, n, k, l$.

Case 1: $t\ge 1$ and $l= 0$. We consider the special case $m=1$.
Suppose that $\delta_t\circ \1_n=\la_0^t\1_n+\sum
\limits_{1\le i\le n\atop 1\le j\le d} \la_{ij}^t \delta^n_{(i)j}$.
Then for any $k\in \n$, we have
\begin{align*}
\delta_t
= & (\delta_t \circ \1_n)\circ (\1_0, \cdots, \1_0,
\underset{k}{\1_1}, \1_0, \cdots, \1_0)\\
=& (\la_0^t\1_n+\sum\limits_{1\le i\le n\atop 1\le j\le d}
\la_{ij}^t \delta^n_{(i)j})\circ (\1_0, \cdots, \1_0,
\underset{k}{\1_1}, \1_0, \cdots, \1_0)\\
=& \la_0^t\1_1+\sum\limits_{1\le i\le n\atop 1\le j\le d}
\la_{ij}^t (\1_n\circ (\1_1, \cdots, \1_1,
\underset{i}{\delta_j}, \1_1, \cdots, \1_1))\circ
(\1_0, \cdots, \1_0, \underset{k}{\1_1}, \1_0, \cdots, \1_0)\\
=& \la_0^t\1_1+\sum\limits_{j} \la_{kj}^t \delta_j.
\end{align*}
It follows that $\la_0^t=0$ and
$\la_{ij}^t=\begin{cases} 1, & j=t, i\in \n,
\\ 0, & {\rm otherwise}.\end{cases}$
Therefore,
\begin{equation}\label{E6.2.1}\tag{E6.2.1}
\delta_t\circ \1_n=\sum\limits_{1\le i\le n} \delta^n_{(i)t}.
\end{equation}
In general,
\begin{align*}
\delta_{(s)t}^m \underset{i}{\circ} \1_n
=&(\1_m \underset{s}{\circ} \delta_t)\underset{i}{\circ} \1_n\\
=&\begin{cases}
(\1_m\underset{i}{\circ} \1_n)\underset{s}{\circ} \delta_t, & \; i>s,\\
\1_m\underset{s}{\circ} (\delta_t\underset{1}{\circ} \1_n), & \;  i=s,\\
(\1_m\underset{i}{\circ} \1_n)\underset{s+n-1}{\circ} \delta_t, & \; i<s
\end{cases}\\
=&\begin{cases}
\delta^{m+n-1}_{(s)t}, & \qquad\quad i>s,\\
\sum\limits_{k=1}^n \delta^{m+n-1}_{(s+k-1) t}, & \qquad\quad i=s,\\
\delta^{m+n-1}_{(s+n-1)t}, & \qquad\quad i<s.
\end{cases}
\end{align*}

Case 2: $t=0$ and $l\ge 1$.
\[\1_m \underset{i}{\circ} \delta_{(k)l}^n
=\1_m  \underset{i}{\circ} (\1_n\underset{k}{\circ}\delta_l)
=(\1_m\underset{i}{\circ} \1_n)\underset{i+k-1}{\circ} \delta_l
=\1_{m+n-1} \underset{i+k-1}{\circ} \delta_l
=\delta^{m+n-1}_{(i+k-1) l}.\]

Case 3: $t\ge 1$, $l\ge 1$ and $n=1$.

For any $1\le i<i'\le m$, $1\le j, j'\le d$, we have
\[\pi^k(\1_m\circ (\1_1, \cdots, \1_1,
\underset{i}{\delta_{j}}, \1_1, \cdots, \1_1,
\underset{i'}{\delta_{j'}}, \1_1, \cdots, \1_1))=0\]
for any $k\in [m]$. So
$\1_m\circ (\1_1, \cdots, \1_1, \underset{i}{\delta_{j}},
\1_1, \cdots, \1_1, \underset{i'}{\delta_{j'}}, \1_1, \cdots, \1_1)=0$.
It follows that
\begin{equation}\label{E6.2.2}\tag{E6.2.2}
\delta^{m}_{(s)t}\underset{i}{\circ} \delta_{l}=0
\end{equation}
for any $1\le s\neq i\le m$.

\noindent
{\bf Step 3:}
Next we consider the multiplication of $\mathcal{D}(1)$.
Suppose that
$\delta_j\circ \delta_{j'}
=\Omega_{jj'}^0\1_1+\sum\limits_{k=1}^d \Omega_{jj'}^k \delta_k$,
where $\Omega_{jj'}^k$ ($k=0,1, \cdots, d$) are the structure
constants of the associative algebra $\mathcal{D}(1)$
associated to the basis $\{\1_1, \delta_1, \cdots, \delta_d\}$.
By \eqref{E6.2.1}, we have
$\delta_j\circ \1_2=\1_2\circ (\delta_j, \1_1)
+\1_2\circ (\1_1, \delta_j)$, and
by \eqref{E6.2.2}, we have
$\1_2\circ (\delta_j, \delta_{j'})=0$ for any $1\le j, j'\le d$.
It follows that
\begin{align*}
(\delta_j\circ \delta_{j'})\circ \1_2
=\1_2\circ (\delta_j\circ \delta_{j'}, \1_1)
+\1_2\circ (\1_1, \delta_j\circ \delta_{j'})
\end{align*}
and hence $\Omega_{jj'}^0=0$, which means that
$\mathcal{D}(1)=\Bbbk \1_1\oplus \overline{\mathcal{D}(1)}$
is an augmented algebra with
$\overline{\mathcal{D}(1)}=\bigoplus\limits_{j=1}^d\Bbbk \delta_j$.

\noindent
{\bf Step 4:}
We now consider general $\delta_{(s)t}^m \underset{i}{\circ} \delta_{(k)l}^n$
for $t,l\geq 1$.

By \eqref{E6.2.2}, we have
$\delta_{(s)t}^m \underset{i}{\circ} \delta_{(k)l}^n=0$ for any $i\neq s$.
If $s=i$, we have
\begin{align*}
\delta_{(s)t}^m \underset{i}{\circ} \delta_{(k)l}^n
=& (\1_m\underset{s}{\circ} \delta_t) \underset{i}{\circ}
   (\1_n \underset{k}{\circ}\delta_l)=
   \1_m\underset{s}{\circ} (\delta_t\underset{1}{\circ}
	 (\1_n \underset{k}{\circ} \delta_l))\\
=& \1_m\underset{s}{\circ} ((\delta_t\underset{1}{\circ} \1_n)
   \underset{k}{\circ} \delta_l)\\
=& \1_m\underset{s}{\circ}
   ((\sum\limits_{u=1}^n \delta^{n}_{(u)t})
   \underset{k}{\circ} \delta_l) \qquad {\text{by \eqref{E6.2.1}}}\\
=& \1_m\underset{s}{\circ} (\delta^{n}_{(k)t} \underset{k}{\circ}\delta_l)
\qquad \qquad  {\text{by \eqref{E6.2.2}}}\\
=& \1_m\underset{s}{\circ} (\1_n \underset{k}{\circ}
   (\delta_t\underset{1}{\circ}\delta_l)) \\
=& (\1_m\underset{s}{\circ} \1_n) \underset{k+s-1}{\circ}
   (\delta_t\underset{1}{\circ}\delta_l)\\
=&\sum\limits_{v=1}^d \Omega_{tl}^v \delta^{m+n-1}_{(k+s-1) v}.
\end{align*}

The first 4 steps show that \eqref{E2.3.3} holds.

\noindent
{\bf Step 5:}
Finally it follows from ${^2\iu}_{\mathcal D}=0$ that $\delta^n_{(i)j}\ast \sigma
=\delta^n_{(\sigma^{-1}(i))j}$ for all $\sigma \in \S_n$.

As above, we have shown that a 2-unitary operad $\mathcal{D}$ is isomorphic
to an operad introduced in Example \ref{xxex2.3}(1) with $A=\mathcal{D}(1)$
and that $\mathcal{D}$ is uniquely determined by an augmented algebra $\mathcal{D}(1)$.
This implies that $F\bullet G \cong Id$, as required.
\end{proof}
\def\I{\mathcal{I}}

\begin{proof}[Proof of Theorem \ref{xxthm0.7}]
Let $\ip=\As/\I$ be a quotient operad of $\As$ of
GKdimension $n$.
Let ${^k\iu_\ip}$ and ${^k\iu}$ be the truncations
of $\ip$ and $\As$, respectively.
Since $\ip$ is unitary, $f_\ip(0)=\dim {^{0}\iu_\ip}(0)
=\dim \ip(0)=1$.

(1) This is Proposition \ref{xxpro0.5}.

(2) Since $\ip$ is a quotient of
$\As$, $f_\ip(1)=\dim {^1\iu_\ip}(1)=0$.
By \eqref{E5.2.3}, $\gkdim \ip$ is either 1 or at least 3.

(3) $\gkdim \ip=3$.  From Corollary \ref{xxcor0.2}
and Lemma \ref{xxlem3.8},
it immediately follows that $\I={^3\iu}$ and
$\ip=\As\slash {^3\iu}$.

(4) $\gkdim\ip=4$. Then $\dim\ip(0)=\dim\ip(1)=1$.
By Lemma \ref{xxlem3.8}, $\I \subseteq {^3 \iu}$.
Hence $\dim\ip(2)=2$, and consequently by \eqref{E5.2.4},
$\dim f_{\ip}(1)=0$ and $\dim f_\ip(2)=1$.
Hence we have
\[G_\ip(t)=\sum_{n=0}^\infty (1+\tbinom{n}{2}+f_\ip(3)\tbinom{n}{3})t^n.\]
Observe that
$\I(3)$ must be a $\Bbbk\S_3$-submodule of
$${^3\iu}(3):=
\Bbbk((1, 2, 3)-(2, 1, 3)-(3, 1, 2)+(3, 2, 1))+\Bbbk((1, 3, 2)-(2, 1, 3)-(3, 1, 2)+(2, 3, 1)),$$
where the permutations is written by the convention introduced Appendix 8.1.
Since ${^3\iu}(3)$ above is a simple $\Bbbk \S_3$-module,
we have either $\I(3)=0$ or $\I(3)={^3\iu}(3)$.

If $\I(3)={^3\iu}(3)$, then $f_\ip(3)=0$,
which is impossible. The only possibility is $\I(3)=0$.
In this case, $\dim \ip(3)=6$ and $f_\ip(3)=2$. So we have
\[\dim \ip(n)=1+\tbinom{n}{2}+2\tbinom{n}{3}=\dim (\As\slash {^4\iu})(n),\]
and consequently,
$$\dim \I(n)=\dim {^4\iu}(n).$$
On the other hand, we have
${^4\iu}\subseteq W$. Therefore, we have $\I(n)= {^{4}\iu}(n)$
for all $n\ge 4$. It follows that $\I={^4\iu}$ and
$\ip=\As\slash {^4\iu}$.

(5) It is easy to see that $\dim \As/{^4\iu}(4)=15$
(for example, by the proof of part (4)). Hence
$\dim {^4 \iu}(4)=4!-15=9$. Thus there is a nonzero $\Bbbk \S_4$-submodule
$M\subsetneq {^4\iu}(4)$. Since $\As(1)=\Bbbk$,
both (E3.1.2) and (E3.1.3) hold trivially for $M$. By Proposition
\ref{xxpro3.2}(1,2), ${^4 \iu}^M$ is an ideal of $\As$. By the
choice of $M$, we have
$${^5 \iu}\subsetneq {^4\iu}^M\subsetneq {^4 \iu}$$
which implies that
$$\gkdim \As/{^5 \iu}=5=\gkdim \As/{^4\iu}^M.$$
Since the Hilbert series of $\As/{^5 \iu}$ and $\As/{^4\iu}^M$
are different, these two operads are non-isomorphic.
\end{proof}

We make a remark.

\begin{remark}
\label{xxrem6.3}
The proof of Theorem \ref{xxthm0.6} works for  non-locally
finite operads when the use of the generating function is replaced by
the basis theorem \ref{xxthm4.6} (1). Therefore, for non-locally finite
2-unitary operads, there are natural equivalences between the following
categories:
\begin{enumerate}
\item[(1I)]
the category of $\Bbbk$-algebras not necessarily having unit;
\item[(1I')]
the category of unital augmented $\Bbbk$-algebras;
\item[(2I)]
the category of 2-unitary operads with $^2 \iu=0$;
\item[(3I)]
the category of $2a$-unitary operads with $^2 \iu=0$.
\end{enumerate}
Every operad in one of the above categories is isomorphic to one 
given in Example \ref{xxex2.3}(1).
\end{remark}

For the rest of this section we consider $\Com$-augmented
operads. Let $\OP_{\Com}$ denote the category of
$\Com$-augmented operads. For every $\ip$ in $\OP_{\Com}$,
there is a natural decomposition
$$\ip=\Com \oplus \iu$$
of $\S$-module, where $\iu:={^1 \iu}=\ker(\ip\to \Com)$.

\begin{definition}
\label{xxdef6.4}
Let $\{\ip_i\}_{i\in I}$ be a family of operads in $\OP_{\Com}$.
The {\it $\Com$-augmented sum} of $\{\ip_i\}_{i\in I}$ is
defined to be
\begin{equation}
\label{E6.4.1}\tag{E6.4.1}
\bigoplus_{i\in I} \ip_i:=
\Com \oplus \bigoplus_{i\in I} \iu_{\ip_i}
\end{equation}
with relations, for all homogeneous element
$\theta_k$ in $\Com \cup \bigcup_{i\in I}  \iu_{\ip_i}$,
\begin{equation}
\label{E6.4.2}\tag{E6.4.2}
\theta_0 \circ (\theta_1,\cdots, \theta_n)=0
\end{equation}
whenever at least two of $\theta_0, \cdots, \theta_n$ are in different $\iu_{\ip_j}$.
If all $\theta_k$'s are in the same $\ip_j$,
then the composition in $\bigoplus_{i\in I}
\ip_i$ agrees with the composition in
$\ip_j$.
\end{definition}

\begin{lemma}
\label{xxlem6.5}
Let $\{\ip_i\}_{i\in I}$ be a family of operads in
$\OP_{\Com}$.
\begin{enumerate}
\item[(1)]
$\ip:=\bigoplus_{i\in I} \ip_i$ is an operad in $\OP_{\Com}$.
\item[(2)]
${^k \iu_{\ip}}=\bigoplus_{i\in I} {^k \iu_{\ip_i}}$ for all
$k\geq 1$.
\item[(3)]
${\mathcal S}(\ip)=\sum_{i\in I} {\mathcal S}(\ip_i)$.
\item[(4)]
For each subset $I'\subseteq I$, $\bigoplus_{i\in I'}\iu_{\ip_i}$
is an ideal of $\ip$. As a consequence, if there are infinitely
many $i$ such that $\iu_{\ip_i}\neq 0$, then $\ip$ is neither
artinian nor noetherian.
\end{enumerate}
\end{lemma}

\begin{proof} (1) We need to show (OP1), (OP2), (OP3)
in Definition \ref{xxdef1.1}. Since all maps are linear
or multilinear, we only need to consider elements in
$\Com$, $\iu_{\ip_i}$, for $i\in I$. Using the relations
in \eqref{E6.4.2}, it amounts to verify (OP1), (OP2) and (OP3)
for elements in $\Com\cup \iu_{\ip_i}$ for one $i$.
In this setting (OP1), (OP2), (OP3) hold since $\ip_i$
is an operad. Therefore $\bigoplus_{i\in I} \ip_i$ is an
operad. It is clear from \eqref{E6.4.1} that we can
define a morphism from $\Com \to \bigoplus_{i\in I} \ip_i$.
So the assertion follows.

(2) Let $\ip$ be $\bigoplus_{i\in I} \ip_i$.
It is clear from the definition that
$${^1 \iu_{\ip}}=\bigoplus_{i\in I} {^1\iu_{\ip_i}}.$$
Inside this ideal, we have $\pi^I_{\ip}=\bigoplus_{i\in I}
\pi^I_{\ip_i}$ for restriction maps defined in \eqref{E2.3.4}.
The assertion follows easily from this fact.

(3) This is an consequence of part (2).

(4) It is easy to show and the proof is omitted.
\end{proof}

To prove Theorem \ref{xxthm0.8}(1), we need to construct
an operad with signature $\{0,\cdots, 0, d_w, 0,\cdots\}$
for a positive number $d_w$ in the $w$th position of this
sequence.

\begin{example}
\label{xxex6.6}
Fix $w\geq 1$ and $d\geq 1$. In this example, we construct
a $\Com$-augmented operad of signature $\{0,\cdots, 0, d, 0,\cdots\}$,
where $d$ is in $w$th position.

Let $V$ be an $\S_w$-module of dimension $d$ and let
$\{\delta_1,\cdots,\delta_d\}$ be a $\Bbbk$-linear basis of $V$.
If $w=1$, we further assume that the multiplication
$\delta_i \delta_j=0$ for all $i,j$.
Let ${\mathcal C}^n_w$ be defined as before Lemma \ref{xxlem4.10}.
Define
$$\ip(n)=
\begin{cases}
\Bbbk \1_n, & n<w,\\
\Bbbk \1_w \bigoplus V & n=w,\\
\Bbbk \1_n \bigoplus {\mathcal C}^n_{w}(V)
& n>w.
\end{cases}$$
We recall the following notation.
For $n=w+s$, where $s>0$, and for every $I\subseteq \n$
such that $|I|=s$ and for $j\in \d$, let
$$(\delta, I):=\1_2\circ(\delta, \1_{s})\ast c_{I}$$
for all $\delta\in V$.
As a vector space, $\ip(n)$ has a basis $\{\1_n\}\cup
\{(\delta_i, I)\mid i\in \d, I\subseteq \n, |I|=s\}$.

Assuming first that $\ip$ is an operad, we would like
to derive some defining equations.
By Corollary \ref{xxcor4.4}, if $I'\subseteq \n$ such that
$|I'|=n-w$, then
$$\Gamma^{I'}((\delta, I))=
\begin{cases}
\delta, & I=I',\\
0,& I\neq I',
\end{cases}$$
or, for $J\subseteq \n$ with $|J|=w$,
\begin{equation}
\label{E6.6.1}\tag{E6.6.1}
\pi^{J}((\delta,I))=
\begin{cases}
\delta, & I\cup J=\n,\\
0,& I\cup J\neq \n.\end{cases}
\end{equation}
Following Lemma \ref{xxlem4.9}, we set
$$(\delta, I)\ast \sigma =(\delta\ast \Gamma^{\sigma^{-1}(I)}, \sigma^{-1}(I))$$
for all $\delta\in V$ and $I$. Together with the trivial $\S_n$ on
$\Bbbk \1_n$, this defines $\S_n$-module structure on $\ip(n)$.

Next we consider partial compositions. Similar to
Example \ref{xxex2.3} (1), we set
$$(\delta, I)\underset{s}{\circ} (\delta', I')=0$$
because, for every $|J|=w$,
$$\pi^{J} ((\delta, I)\underset{s}{\circ} (\delta', I'))
=0.$$
Write $I=\{i_1,\cdots, i_{n-w}\}\subseteq \n$. Define
$$\1_m\underset{s}{\circ} \1_n=\1_{m+n-1},$$
$$\1_m\underset{s}{\circ} (\delta, I)=(\delta, I'),$$
where $I'=\{1,\cdots, s-1, I+(s-1), n+s,\cdots, n+m-1\}$,
and
$$(\delta, I) \underset{s}{\circ} \1_m=
\begin{cases} (\delta, {\bar{I}}), & s\in I,\\
\sum_{u=1}^m (\delta, I_u), & s\not\in I,
\end{cases}$$
where
$$\bar{I}=\{i_1,\cdots, i_{f-1}, s, s+1,\cdots, s+m-1,i_{f+1}+m-1,\cdots, i_{n-w}+m-1\}$$
when $s=i_f$ for some $f$,
and where
$$I_u=\{i_1,\cdots,i_{f-1}, s, \cdots, \widehat{s+u-1}, \cdots, s+m-1,
i_{f}+m-1, i_{f+1}+m-1,\cdots, i_{n-w}+m-1\}$$
when $i_{f-1}<s<i_{f}$.
Now it is routine to check that $\ip$ is a 2-unitary operad with
given signature.
\end{example}

\begin{theorem}
\label{xxthm6.7}
Let $w\geq 2$.
Every $\Com$-augmented operad of signature
$\{0,\cdots, 0, d_w, 0,\cdots\}$ is of the following form
given in Example \ref{xxex6.6}.
\end{theorem}

\begin{proof} The proof is similar to the proof of Theorem
\ref{xxthm0.6}. We omit the proof due to its length.
\end{proof}

\begin{proof}[Proof of Theorem \ref{xxthm0.8}]
(1) For each $d_w$ for $w\geq 1$, pick a trivial $\S_w$-module
$V_w$ of dimension $d_w$. By Example \ref{xxex6.6},
there is a $\Com$-augmented (thus 2-unitary)
operad $\ip_w$ with signature $\{0,\cdots, 0, d_w, 0,\cdots\}$.
By Lemma \ref{xxlem6.5}(3), $\bigoplus_w \ip_w$ has the
required signature.

(2) Take a sequence ${\mathcal S}(\ip)$
with $\exp({\mathcal S}(\ip))=\infty$, then $\exp(\ip)
=\infty$. One such example is $\ip=\As$.

We know that $\exp(\Com)=1$. Let $\ip$ be an
2-unitary operad with ${\mathcal S}(\ip)=\{b_1, \cdots, b_w,\cdots\}$.
If $b_n=0$ for all $n\gg 0$, then $\exp({\mathcal S}(\ip))=0$
and by Lemma \ref{xxlem5.1}(1), $\exp(\ip)=\exp(\{\dim \ip(n)\}_{n\geq 0})=1$ since
$\{\dim \ip(n)\}_{n\geq 0}$ is the inverse binomial transform
of ${\mathcal S}(\ip)$, see \eqref{E5.2.4}.
Otherwise, $\exp({\mathcal S}(\ip))=1$ and by Lemma \ref{xxlem5.1}(1),
$\exp(\ip)\geq 2$.

It remains to show that for each $r\geq 2$, there is a 2-unitary
operad (in fact, a $\Com$-augmented operad) $\ip$ such that
$\exp(\ip)=r$. Let $d_w=\lfloor (r-1)^w \rfloor$ for each $w\geq 1$.
By part (1), there is a $\Com$-augmented operad $\ip$ such that
${\mathcal S}(\ip)=\{d_1, d_2, \cdots, d_w, \cdots\}$. Thus
$\exp({\mathcal S}(\ip))=r-1$. By Lemma \ref{xxlem5.1}(1),
$\exp(\ip)=r$ as required.
\end{proof}

\begin{proof}[Proof of Theorem \ref{xxthm0.4} (3)]
The proof of this part is similar to the proof of
Theorem \ref{xxthm0.4}(2).

Let ${^k\iu}$ be the truncation ideals of $\ip$. By definition,
$\bigcap_{k\geq 1}{^k \iu} =0$. Since $\ip$ is left or right artinian,
${^k \iu}=0$ for some $k$. Let $n$ be the largest integer such that
${^n \iu}\neq 0$. If $n\geq 2$, by Proposition \ref{xxpro3.1} (2),
$({^n \iu})^2\subseteq {^{2n-1}\iu}=0$. This contradicts the
hypothesis that $\ip$ is semiprime. Therefore ${^2 \iu}=0$.

Let $A=\ip(1)$. By Proposition \ref{xxpro3.2}(1,2),
if $A$ is not left (respectively, right) artinian, then $\ip$ is
not left (respectively, right) artinian. Since $\ip$ is
left or right artinian, so is $A$. Let $N$ be an ideal
of $A$ such that $N^2=0$. By Proposition \ref{xxpro3.2}(1,2),
${^1 \iu}^N$ is an ideal of $\ip$. By Proposition \ref{xxpro3.2}(3),
$$({^1 \iu}^N)^2\subseteq {^1 \iu}^{N^2}={^1 \iu}^{0}={^2 \iu}=0.$$
Since $\ip$ is semiprime, ${^1 \iu}^N=0$, consequently, $N=0$.
Thus $A$ is semiprime. Since $A$ is left artinian or right
artinian, $A$ is semisimple.

By Remark \ref{xxrem6.3}, the operad $\ip$ is given as in Example
\ref{xxex2.3}(1).

If, further, $\ip(1)$ is finite dimensional, then by
Theorem \ref{xxthm4.6}(1) $\ip$ is locally finite.
Since ${^2\iu}=0$, $\gkdim \ip\leq 2$. If $\gkdim \ip=1$,
then $\ip=\Com$ by Proposition \ref{xxpro0.5}.
Otherwise $\gkdim \ip=2$. The rest of assertion follows.
\end{proof}

We conclude this section with an easy corollary.

\begin{corollary}
\label{xxcor6.8}
Let ${\mathbf d}:=\{d_i\}_{i\geq 1}$ be any sequence of non-negative integers.
Then there is a unitary operad $\ip$ such that $G_{\ip}(t)=1+(d_1+1)t+\sum_{i=2}^{\infty}
d_i t^i$.
\end{corollary}

\begin{proof} By Theorem \ref{xxthm0.8}(1),
there is a 2-unitary operad $\iq$ such that ${\mathcal S}(\iq)={\mathbf d}$.
Let $\ip=\Bbbk 1_1\oplus \bigoplus_{i=0}^{\infty} {^i \iu}_{\iq}(i)$.
By Proposition \ref{xxpro3.12}(2), $\ip$ is a unitary operad. By the
definition of signature, we see that
$$G_{\ip}(t)=1+(d_1+1)t+\sum_{i=2}^{\infty} d_i t^i.$$
\end{proof}

\section{Truncatified operads}
\label{xxsec7}

The truncation of a unitary operad $\ip$ defines a descending filtration
on $\ip$ which induces an associated operad, called a truncatified
operad, as we will define next.

\begin{definition}
\label{xxdef7.1}
A unitary operad $\ip$ is called {\it truncatified} if the following
hold.
\begin{enumerate}
\item[(1)]
For each $n$, $\ip(n)$ has a decomposition of $\S_n$-submodules,
$$\ip(n)=\bigoplus_{i=0}^{n} \ip(n)_i.$$
\item[(2)]
For all $k$ and all $n\geq k$,
$${^k \iu}(n)=\bigoplus_{i=k}^n \ip(n)_i.$$
\item[(3)]
Let $\mu\in \ip(n)_{n_0}$ and $\nu\in \ip(m)_{m_0}$. Suppose $1\leq i\leq n$.
\begin{enumerate}
\item[(3a)]
If $n_0, m_0\geq 1$, then
$$\mu\underset{i}{\circ} \nu\in \ip(n+m-1)_{n_0+m_0-1}.$$
\item[(3b)]
If $m_0=0$ or $n_0=0$, then
$$\mu\underset{i}{\circ} \nu\in \ip(n+m-1)_{m_0+n_0}.$$
\end{enumerate}
\end{enumerate}
\end{definition}

\begin{remark}
\label{xxrem7.2}
A truncatified operad in the above definition may be called a
{\it truncated} operad since it is induced by the truncation
(see also Lemma \ref{xxlem7.3}).
However, the notion of a {\it truncated} operad has been defined
in \cite[Definition 4.2.1]{GNPR} and been used in some other
papers \cite{We}. To avoid possible confusions, we create a new
word, ``truncatified'', in Definition \ref{xxdef7.1}. Note that
every truncatified operad is either $\Com$-augmented or
$\Uni$-augmented.
\end{remark}

It is easy to check that the operads in Examples \ref{xxex2.3}
and \ref{xxex6.6} are truncatified. Truncatified operads can be
constructed from a non-truncatified operad.

\begin{lemma}
\label{xxlem7.3}
Let $\iq$ be a unitary operad and $\{{^i \iu}_{\iq}\}_{i\geq 0}$
be the truncation of $\iq$. For each $n\geq 0$, let $\ip(n)$
denote the $\Bbbk$-linear space $\bigoplus_{i=0}^{\infty}
{^i \iu}_{\iq}(n)/{^{i+1} \iu}_{\iq}(n)$. Then
$\ip:=\{\ip(n)\}_{n\geq 0}$ is a truncatified operad.
\end{lemma}

\begin{proof} Let $\ip(n)_i:={^i \iu}_{\iq}(n)/{^{i+1} \iu}_{\iq}(n)$
for all $n,i$. For the rest of the proof, $i,j,k$, $m,n$ and $s$ are
non-negative integers. Assume that $1\leq s\leq m$.

Let $\mu\in \ip(m)_i$ and $\nu\in \ip(n)_j$.
Then $\mu$ is the image of some $a\in {^i \iu}_{\iq}(m)$
and $\nu$ is the image of some $b\in {^j \iu}_{\iq}(n)$.
Define $\mu\underset{s}{\circ} \nu$ to be the image of
$a\underset{s}{\circ} b$ in
$\ip(m+n-1)_{i+j-1}:={^{i+j-1} \iu}_{\iq}(m+n-1)/{^{i+j} \iu}_{\iq}(m+n-1)$
(or in $\ip(m+n-1)_{i+j}$ if
either $i$ or $j$ is zero). It is routine to check that
$\ip$ is a unitary operad using the partial definition Definition
\ref{xxdef1.2}.

Next we show (1), (2) and (3) in Definition \ref{xxdef7.1}.

(1) Since ${^i \iu}_{\iq}(n)=0$ for all $i>n$,  we have
$$\ip(n)=\bigoplus_{i=0}^{\infty}
{^i \iu}_{\iq}(n)/{^{i+1} \iu}_{\iq}(n)=\bigoplus_{i=0}^{n}
{^i \iu}_{\iq}(n)/{^{i+1} \iu}_{\iq}(n)
=\bigoplus_{i=0}^{n} \ip(n)_i.$$
Since each $\ip(n)_i$ is clearly an $\S_n$-module, (1) holds.

(2) Denote $T^n_{n-k}=\{K\subset \n \mid |K|=k\}$.
(Note that $T^n_{k}$ is defined before Lemma \ref{xxlem4.10}.)

Let $\theta$ be an element in $\up{i}_\iq(n)$ such that
$\theta\notin \up{i+1}_\iq(n)$. If $k<i$, by definition, we
have $\pi_\iq^K(\theta)=0$ for all $K\in T^n_{n-k}$.
If $k\geq i$, we have $\pi_\iq^K(\theta)\in \up{i}_\iq(k)$ for
all $K\in T^n_{n-k}$, and there exists some $K_0\in T^n_{n-k}$
such that $\pi_\iq^{K_0}(\theta)\notin \up{i+1}_\iq(k)$. In fact,
since $\theta\notin \up{i+1}_\iq(n)$, there exists some
$I\in T^n_{n-i}$ such that $\pi_\iq^I(\theta)\neq 0$.
Then for every $K_0$ with $I\subseteq K_0 \in T^n_{n-k}$, we have
$\pi_\iq^{K_0}\notin \up{i+1}_\iq(k)$.

Return to consider the restricted operator
$\pi_\ip^I\colon \ip(n) \to \ip(|I|)$. Pick any nonzero
element $\mu$ in $\up{i}_\iq(n)/\up{i+1}_\iq(n)$ and write
it as $\mu=\theta+\up{i+1}_\iq(n)\neq \bar 0$ for
some $\theta\in \up{i}_\iq(n)$. If $i>k-1$, then, for every
$I\in T^n_{n-(k-1)}$,
we have
\[\pi_\ip^I(\theta+\up{i+1}_\iq(n))=\bar 0\in \ip(k-1).\]
This implies that, for any $i\geq k$,
$\ip(n)_i=\up{i}_\iq(n)/\up{i+1}_\iq(n)\subset \up{k}_\ip(n)$.

On the other hand, if $i<k$, then, for every nonzero
element $\mu:=\theta+\up{i+1}_\iq(n)\in \ip(n)_i$, there
exists $I_0\in T^n_{n-(k-1)}$ such that
\[\pi_\ip^{I_0}(\mu)=\pi_\ip^{I_0}(\theta+\up{i+1}_\iq(n))
=\pi^{I_0}_\iq(\theta)+\up{i+1}_\iq(k-1)\neq \bar 0.\]
in $\ip(k-1)$. It follows that
$\up{k}_\ip(n)\subset \bigoplus_{i=k}^n \ip(n)_i$.

(3) Note that (3a) and (3b) follow from the proof of Proposition
\ref{xxpro3.1}(2).
\end{proof}

In the setting of Lemma \ref{xxlem7.3}, we say that $\ip$ is the
{\it associated truncatified operad} of $\iq$, and denoted it by ${\rm{trc}} \iq$.
The process from $\iq$ to ${\rm{trc}} \iq$ is called {\it truncatifying}.

It follows from Lemma \ref{xxlem7.3} that a unitary operad $\ip$
is truncatified if and only if $\ip\cong {\rm trc} (\ip)$.
As a consequence, ${\rm trc}({\rm trc}(\ip))\cong {\rm trc}(\ip)$
for all unitary operads $\ip$.

Next we show that $\Pois$ is the associated truncatified operad of $\As$.
For any unitary operad $\ip$, let $\ip_{\geq 1}$ be the non-unitary version
of $\ip$, namely,
$$\ip_{\geq 1}(n)=\begin{cases} 0 & n=0, \\
\ip(n) & n\geq 1.\end{cases}$$
Note that $\Pois_{\geq 1}$ agrees with the non-unitary version of
the Poisson operad, and $\Pois_{\geq 1}$ is denoted by $\Pois$
in \cite[Section 1.2.12]{Fr1} and \cite[Section 13.3.3]{LV}.
On the other hand, the unitary version of the Poisson operad
(namely, our $\Pois$) is denoted by $\Pois_{+}$ in the book \cite{Fr1}.

\begin{lemma}
\label{xxlem7.4}
Let $\As$ be the operad for the unital associative algebras and $\Pois$
be the operad for unital commutative Poisson algebras. Then
${\rm{trc}}\As\cong \Pois$.
\end{lemma}

\begin{proof}
Denote by $\up{k}$ the $k$-th truncation ideal of $\As$.
By Lemma 3.4, we have $\up{1}=\up{2}$. By definition, we
have ${\rm{trc}}\As(2)=\Bbbk \bar 1_2\oplus \Bbbk \bar\Phi_2$,
where $\bar 1_2=1_2+\up{1}(2)$ and
$\bar\Phi_2\in {\rm{trc}}\As(2)$ is the corresponding
element of $(1_2-(21))\in \up{2}(2)$. Clearly,
$\bar 1_2\ast (21)=\bar 1_2$ and $\bar\Phi_2\ast (21)=-\bar\Phi_2$,
and they satisfy the following relations
\begin{align}
\label{E7.4.1}\tag{E7.4.1}
\bar 1_2 \underset{1}{\circ} \bar 1_2=
& \bar 1_2 \underset{2}{\circ} \bar 1_2, \\
\label{E7.4.2}\tag{E7.4.2}
\bar\Phi_2\underset{1}{\circ} \bar 1_2=
& \bar 1_2 \underset{2}{\circ} \bar\Phi_2
  +(\bar 1_2 \underset{2}{\circ} \bar\Phi_2)\ast (213),\\
\label{E7.4.3}\tag{E7.4.3}
\bar\Phi_2 \underset{2}{\circ} \bar\Phi_2 =
& \bar\Phi_2 \underset{1}{\circ} \bar\Phi_2
  +(\bar\Phi_2 \underset{2}{\circ} \bar\Phi_2)\ast (213)
\end{align}
which are exactly the defining relations of $\Pois_{\geq 1}$,
see \cite[Section 1.2.12]{Fr1}.
Observe that ${\rm trc}\As$ is generated by
$\Bbbk \bar 1_2\oplus \Bbbk \bar\Phi_2$.
In fact, from Theorem \ref{xxthm4.6}(1), we know
$\up{k}(n)/\up{k+1}(n)$ admits a $\Bbbk$-linear basis
\[\textbf{\it B}_k(n)=
\{1_2\circ (\theta_i^k, 1_{n-k})\ast c_I\mid 1\le i\le z_k,
I\in T^n_{k}\},\]
where $\{\theta_1^k, \cdots, \theta_{z_k}^k\}$ is a
$\Bbbk$-basis of $\up{k}(k)$. Furthermore, for every
$k\ge 3$, we have $\up{k}(k)\subset \up{2}(k)$.
By Lemma \ref{xxlem3.7}, $\up{2}(2)$ is generated by $\Phi_2$.
Note that $1_2$ generates $1_n$ for all $n\geq 2$.
By the proof of Lemma \ref{xxlem3.7},
for every $k\geq 3$, $\up{k}(k)$ is generated by $\{1_n\}_{n\geq 2}$
and $\Phi_2$.  Therefore $\up{k}(n)/\up{k+1}(n)$ can be
generated by $\bar 1_2$ and $\bar \Phi_2$ for any $n\ge k\ge 2$.
It follows that $({\rm{trc}}\As)_{\geq 1}$ can be
generated by $\bar 1_2$ and $\bar \Phi_2$.
The above argument shows that there is a canonical
epimorphism $\ip:=\mathscr{T}(E)/(R) \to ({\rm{trc}}\As)_{\geq 1}$,
where $\mathscr{T}(E)/(R)$ be the quotient operad of the
free operad $\mathscr{T}(E)$ on the $\Bbbk\S$-module
$E=(0, 0, \Bbbk \bar 1_2\oplus \Bbbk \bar\Phi_2, 0, \cdots)$
modulo relations \eqref{E7.4.1}-\eqref{E7.4.3}. By
\cite[Section 1.2.12]{Fr1}, $\ip\cong \Pois_{\geq 1}$.
By the fact that
$$\dim \ip(n)=\dim \Pois(n)=n!=\dim \As(n)=\dim {\rm{trc}}\As(n)$$
for all $n\geq 1$ \cite[Section 13.3.3]{LV},
we have $({\rm{trc}}\As)_{\geq 1}=\ip$,
which is isomorphic to the Poisson operad $\Pois_{\geq 1}$.
Therefore we obtain that $({\rm{trc}}\As)_{\geq 1}=
\Pois_{\geq 1}$. It remains to verify that 0-ary
operations of ${\rm{trc}}\As$ and $\Pois$ agree.
We can easily see that, in ${\rm{trc}}\As$,
$${\bar 1_2}\underset{i}{\circ} {\bar 1_0}={\bar 1_1}$$
and
$${\bar \Phi_2}\underset{i}{\circ} {\bar 1_0}=0$$
for $i=1,2$. This is also how we define the unitary
Poisson operad $\Pois$. This finishes the proof.
\end{proof}

\begin{remark}
\label{xxrem7.5} We make some comments about the
above lemma.

\begin{enumerate}
\item[(1)]
The result in Lemma \ref{xxlem7.4} may be well-known,
possibly in a different language. Similar ideas
appeared in \cite{LL, Br, MaR}.
\item[(2)]
By Livernet-Loday \cite{LL}, $\As$ is a deformation of
$\Pois$,  in the sense that there is a family of operads,
denoted by ${\mathcal {LL}}_q$, such that $\Pois\cong
{\mathcal {LL}}_0$ and that $\As\cong {\mathcal {LL}}_q$
for any $q\neq 0$. Further study in this direction can be
found in \cite{Br, MaR} and \cite[Section 13.3.4]{LV}.
Lemma \ref{xxlem7.4} gives an explanation why $\As$ is
a deformation of $\Pois$.
We refer to \cite[Example 4 and Theorem 5]{MaR} for
some interesting connections with
deformation quantization.
\item[(3)]
Related to combinatorics, the dimension of ${^k \iu}(k)$
of either $\As$ and $\Pois$ is the number of derangements
of a set of size $k$.
\item[(4)]
It would be interesting to determine associated
truncatified operads of other unitary operads.
\end{enumerate}
\end{remark}

\section{Appendix}
\label{xxsec8}

In this part, we mainly rewrite some conventions and facts on 
operads, see \cite{LV} or \cite{Fr1, Fr2}.

\subsection{Symmetric groups, permutations and block permutations}
\label{xxsec8.1}
 We use $\S_n$ to denote the symmetric group,
namely, the set of bijections, on the set $\n$. Note that both $\S_0$ and
$\S_1$ are isomorphic to the trivial group with one element.

Following convention in the book \cite{LV}, we identify $\S_n$ with the
set of permutations of $\n$ by assigning each $\sigma\in \S_n$ the sequence
$(\sigma^{-1}(1),\sigma^{-1}(2),\cdots,\sigma^{-1}(n))$. This assignment
is convenient when we use other convention such as \eqref{E1.1.4}.
Equivalently, each permutation $(i_1, i_2, \cdots, i_n)$ of $\n$ corresponds
to the $\sigma\in\S_n$ given by $\sigma(i_k)=k$ for all $1\le k\le n$.

Let $n>0$, $k_1, k_2\cdots, k_n\ge 0$ be integers. For simplicity we write
$m=k_1+k_2+\cdots+ k_n$, $m_1=0$, and $m_i=k_1+\cdots +k_{i-1}$ for
$2\le i\le n$. We may divide $(1,2,\cdots, m)$ into $n$-blocks
$(B_1, B_2, \cdots, B_n)$, where $B_i= (m_i+1, \cdots, m_i+k_i)$ for
$1\le i\le n$. Now each $\S_{k_i}$ acts on the block $B_i$, and each element
in $\S_n$ acts on $[m]$ naturally by permuting the blocks. More
precisely, we have the following natural map
\begin{align*}
\vartheta_{n; k_1,\cdots, k_n}\colon  \S_n \times\S_{k_1}\times\cdots
        \times \S_{k_n}  & \to \S_m,\\
(\sigma, \sigma_1,\cdots, \sigma_n)& \mapsto (\tilde{B}_{\sigma^{-1}(1)},
\cdots, \tilde{B}_{\sigma^{-1}(n)})
\end{align*}
for all $\sigma\in \S_n$ and $\sigma_i\in\S_{k_i}$ for $1\le i \le n$,
where each
\begin{equation}
\label{E1.0.1}\tag{E8.0.1}
\tilde{B}_i=m_i+(\sigma_i^{-1}(1), \cdots, \sigma_i^{-1}(k_i))
=(m_i+\sigma_i^{-1}(1), \cdots, m_i+\sigma_i^{-1}(k_i))
\end{equation}
is the sequence corresponding to $\sigma_i$.

The following lemma is easy.

\begin{lemma}
\label{xxlem8.1}
Retain the above notation.
\begin{enumerate}
\item[(1)]
We have
\begin{align}\label{E8.1.1}\tag{E8.1.1}
\vartheta_{n; k_1,\cdots, k_n}
&(\tau\sigma, \tau_1\sigma_1, \cdots, \tau_n\sigma_n)\\
\notag
=&\vartheta_{n; k_{\sigma^{-1}(1)},\cdots, k_{\sigma^{-1}(n)}}
(\tau, \tau_{\sigma^{-1}(1)}, \cdots, \tau_{\sigma^{-1}(n)})
\vartheta_{n; k_1, \cdots, k_n}(\sigma, \sigma_1, \cdots, \sigma_n)
\end{align}
for all $\sigma, \tau\in\S_n$, and $\sigma_i,\tau_i\in \S_{k_i}$, $1\le i\le n$.
\item[(2)]
In particular,
\begin{align}
\vartheta_{n; k_1,\cdots, k_n}
&(\sigma, \sigma_1, \cdots, \sigma_n)\nonumber \\
=& \vartheta_{n; k_1,\cdots, k_n}(\sigma, 1, \cdots, 1)
\vartheta_{n; k_1, \cdots, k_n}(1, \sigma_1, \cdots, \sigma_n)
\label{E8.1.2}\tag{E8.1.2}\\
=& \vartheta_{n; k_{\sigma^{-1}(1)}, \cdots, k_{\sigma^{-1}(n)}}
(1, \sigma_{\sigma^{-1}(1)}, \cdots, \sigma_{\sigma^{-1}(n)})
\vartheta_{n; k_1,\cdots, k_n}(\sigma, 1, \cdots, 1)\nonumber
\end{align}
where $1$ in different positions represents the identity map of $[k_i]$
or $[n]$.
\end{enumerate}
\end{lemma}

\begin{proof}
We first prove  part (2). For any $\sigma\in \mathbb{S}_n$,
$\sigma_i\in \mathbb{S}_{k_i}$, $1\le i\le n$, using notation
in \eqref{E1.0.1}, we have
\begin{align*}
\vartheta_{n; k_1, \cdots, k_n}(\sigma, 1, \cdots, 1)&
\vartheta_{n; k_1, \cdots, k_n}(1, \sigma_1, \cdots, \sigma_n)\\
&=  \vartheta_{n; k_1, \cdots, k_n}(\sigma, 1, \cdots, 1)
(\tilde{B}_1, \cdots, \tilde{B}_n)\\
&=  (\tilde{B}_{\sigma^{-1}(1)}, \cdots, \tilde{B}_{\sigma^{-1}(n)})\\
&=  \vartheta_{n; k_1, \cdots, k_n}(\sigma, \sigma_1, \cdots, \sigma_n),
\end{align*}
and
\begin{align*}
\vartheta_{n; k_{\sigma^{-1}(1)}, \cdots, k_{\sigma^{-1}(n)}}
&(1, \sigma_{\sigma^{-1}(1)}, \cdots, \sigma_{\sigma^{-1}(n)})
\vartheta_{n; k_1, \cdots, k_n}(\sigma, 1, \cdots, 1)\\
&= \vartheta_{n; k_{\sigma^{-1}(1)}, \cdots, k_{\sigma^{-1}(n)}}
(1, \sigma_{\sigma^{-1}(1)}, \cdots, \sigma_{\sigma^{-1}(n)})
(B_{\sigma^{-1}(1)}, \cdots, B_{\sigma^{-1}(n)})\\
&= (\tilde{B}_{\sigma^{-1}(1)}, \cdots, \tilde{B}_{\sigma^{-1}(n)})\\
&= \vartheta_{n; k_1, \cdots, k_n}(\sigma, \sigma_1, \cdots, \sigma_n).
\end{align*}

For part (1), by part (2), we have
\begin{align*}
\vartheta_{n; k_1,\cdots, k_n}
&(\tau\sigma, \tau_1\sigma_1, \cdots, \tau_n\sigma_n)\\
&= \vartheta_{n; k_1,\cdots, k_n}(\tau\sigma, 1, \cdots, 1)
\vartheta_{n; k_1,\cdots, k_n}(1, \tau_1\sigma_1, \cdots, \tau_n\sigma_n)\\
&= \vartheta_{n; k_{\sigma^{-1}(1)},\cdots, k_{\sigma^{-1}(n)}}
(\tau, 1, \cdots, 1)\vartheta_{n; k_1,\cdots, k_n}(\sigma, 1, \cdots, 1)\\
&\qquad
\vartheta_{n; k_1,\cdots, k_n}(1, \tau_1, \cdots, \tau_n)
\vartheta_{n; k_1,\cdots, k_n}(1, \sigma_1, \cdots, \sigma_n)\\
&= \vartheta_{n; k_{\sigma^{-1}(1)},\cdots, k_{\sigma^{-1}(n)}}
(\tau, 1, \cdots, 1)
\vartheta_{n; k_{\sigma^{-1}(1)},\cdots, k_{\sigma^{-1}(n)}}
(1, \tau_{\sigma^{-1}(1)}, \cdots, \tau_{\sigma^{-1}(n)})\\
&\qquad
\vartheta_{n; k_1,\cdots, k_n}(\sigma, 1, \cdots, 1)
\vartheta_{n; k_1,\cdots, k_n}(1, \sigma_1, \cdots, \sigma_n)\\
&=\vartheta_{n; k_{\sigma^{-1}(1)},\cdots, k_{\sigma^{-1}(n)}}
(\tau, \tau_{\sigma^{-1}(1)}, \cdots, \tau_{\sigma^{-1}(n)})
\vartheta_{n; k_1, \cdots, k_n}(\sigma, \sigma_1, \cdots, \sigma_n)
\end{align*}
for all $\tau, \sigma\in \mathbb{S}_n$, $\tau_i, \sigma_i\in
\mathbb{S}_{k_i}$, $1\le i\le n$.
\end{proof}

For convenience, we introduce the following maps obtained from
$\vartheta_{n; k_1,\cdots, k_n}$:
\begin{align}
\notag
\vt_{k_1,\cdots, k_n}& \colon \S_n\to \S_{m},
 \sigma\mapsto \vartheta_{n; k_1,\cdots, k_n}(\sigma, 1, \cdots, 1),\\
\label{E8.1.3}\tag{E8.1.3}
\vt_{k_1,\cdots, k_n}^i & \colon \S_{k_i}\to \S_{m},
\sigma_i\mapsto \vartheta_{n; k_1,\cdots, k_n}
(1,1,\cdots, 1, \sigma_i, 1,\cdots, 1).
\end{align}

Note that $\S_{k_1}\times\cdots\times \S_{k_n}$ can be viewed
as a subgroup of $\S_m$ via the embedding maps $\vt_{k_1,\cdots, k_n}^i$.
While in general, $\vartheta_{k_1,\cdots, k_n}$ is not an embedding
of groups. It is the case if and only if all blocks have the same
size, that is, $k_1=k_2=\cdots =k_n$.

\subsection{multilinear maps, compositions and symmetric group action}
\label{xxsec8.2}
Let $V$ be a vector space and $n>0$ an integer. Denote by $V^{\ot n}$
the tensor space $V\ot V\ot\cdots\ot V$ with $n$ factors. For any
$v_1,\cdots ,v_n\in V$, we simply denote $v_1\ot v_2\ot\cdots \ot v_n$
by $[v_1, v_2,\cdots, v_n]$. Let $\mathbf{B}\subset V$ be a $\Bbbk$-linear
basis of $V$, then $V^{\ot n}$ has a $\Bbbk$-linear basis
$$\{[v_1,v_2,\cdots, v_n] \; \mid \; v_i\in \mathbf{B},\ 1\le i\le n\}.$$
For consistency, we set $V^{\ot0}=\k$,
and denote by $[\ ]$ a fixed basis element of $V^{\ot 0}$. Under the
map $[v_1,\cdots, v_i]\ot [\ ]\ot [v_{i+1},\cdots, v_{i+j}]\mapsto
[v_1,\cdots, v_{i+j}]$, we may identify $V^{\ot i}\ot V^{\ot 0}
\ot V^{\ot j}$ with $V^{\ot (i+j)}$.

Let $\End_V(n)$ denote the $\k$-vector space $\Hom_\k(V^{\ot n}, V)$
of multilinear operators on $V$. Clearly, $\End_V(0)\cong V$ under
the mapping $f\mapsto f([\ ])$.

It is standard that $\S_n$ acts on $V^{\ot n}$ on the left by
permuting the factors, more precisely
\begin{equation}
\label{E8.1.4}\tag{E8.1.4}
\sigma\cdot[x_1, x_2,\cdots, x_n] =
[x_{\sigma^{-1}(1)}, x_{\sigma^{-1}(2)},\cdots,x_{\sigma^{-1}(n)}]
\end{equation}
for all $\sigma\in\S_n$, and $x_1, x_2\cdots, x_n\in V$
\cite[p. xxiv and p.164]{LV}.
This convention could be different from the one used by some researchers.
This action induces a right action of $\S_n$ on
$\End_V(n)$ by
$$(f\ast \sigma)(X) = f(\sigma X)$$ for all $\sigma\in \S_n$,
$f\in \End_V(n)$ and $X\in V^{\ot n}$. Here $\ast$ denotes the
(right) $\S_n$-action.

Consider the composition map
\begin{align}\label{E8.1.5}\tag{E8.1.5}
\circ \colon \End_V(n)\ot \End_V(k_1)\ot \cdots\ot \End_V(k_n)
&\To \End_V(k_1+\cdots+k_n),\\
(f, f_1, \cdots, f_n)\mapsto f\circ(f_1,\cdots, f_n)
&:=f\bullet(f_1\ot\cdots \ot f_n),\notag
\end{align}
where
\begin{align}\label{E8.1.6}\tag{E8.1.6}
f\bullet(f_1\ot \cdots\ot f_n)&([x_{1,1},\cdots, x_{1,k_1},
\cdots, x_{n,1},\cdots, x_{n,k_n}])\\
&= f(f_1([x_{1,1},\cdots, x_{1,k_1}])\ot \cdots\ot
f_n([x_{n,1},\cdots, x_{n,k_n}])), \notag
\end{align}
for all $f\in\End_V(n)$, $f_i\in\End_V(k_i)$ and $x_{ij}\in V$. Here
$\bullet$ denotes an ordinary composition of two functions and
$\circ$ denotes the composition map of an operad. The composition
map $\circ$ is compatible with the symmetric group actions. The following
is clear.

\begin{lemma}
\label{xxlem8.2}
Keep the above notation. Then
\begin{align}\label{E8.2.1}\tag{E8.2.1}
&(f\ast \sigma)\circ(f_{1} \ast \tau_{1},
\cdots, f_{n} \ast \tau_{n})\\
=& (f\circ(f_{\sigma^{-1}(1)},\cdots, f_{\sigma^{-1}(n)}))
\ast\vartheta_{n; k_1, \cdots, k_n}
(\sigma, \tau_1, \cdots, \tau_n)\notag
\end{align}
for all $\sigma\in \S_n$, and $\tau_i\in \S_{k_i}$, $1\le i\le n$.
\end{lemma}

\begin{proof}
We write $m_1=0$, $m_i=k_1+\cdots+k_{i-1}$ for $2\le i\le n$, and
$m=k_1+\cdots+k_n$. Then
\begin{align*}
(f\ast \sigma)\circ &(f_{1} \ast \tau_{1}, \cdots, f_{n} \ast
   \tau_{n})[x_1, \cdots, x_n]\\
&= (f\ast \sigma) ([(f_1\ast\tau_1)([x_1, \cdots, x_{k_1}]),
   \cdots, (f_n\ast\tau_n)([x_{m_n+1}, \cdots, x_{m_n+k_n}])])\\
&= f([(f_{\sigma^{-1}(1)}\ast \tau_{\sigma^{-1}(1)})
   ([x_{m_{\sigma^{-1}(1)}+1}, \cdots,
	 x_{m_{\sigma^{-1}(1)}+k_{\sigma^{-1}(1)}}]), \cdots, \\
&\quad\quad (f_{\sigma^{-1}(n)}\ast \tau_{\sigma^{-1}(n)})
   ([x_{m_{\sigma^{-1}(n)}+1}, \cdots, x_{m_{\sigma^{-1}(n)}+k_{\sigma^{-1}(n)}}])])\\
&= f([f_{\sigma^{-1}(1)}([x_{m_{\sigma^{-1}(1)}+
   \tau_{\sigma^{-1}(1)}^{-1}(1)}, \cdots, x_{m_{\sigma^{-1}(1)}
	 +\tau_{\sigma^{-1}(1)}^{-1}(k_{\sigma^{-1}(1)})}]),\cdots,\\
&\quad\quad f_{\sigma^{-1}(n)}([x_{m_{\sigma^{-1}(n)}
   +\tau_{\sigma^{-1}(n)}^{-1}(1)}, \cdots, x_{m_{\sigma^{-1}(n)}
	 +\tau_{\sigma^{-1}(n)}^{-1}(k_{\sigma^{-1}(n)})} ])])\\
&= ((f\circ(f_{\sigma^{-1}(1)},\cdots, f_{\sigma^{-1}(n)}))\ast
   \vartheta_{n; k_1, \cdots, k_n}
   (\sigma, \tau_1, \cdots, \tau_n))([x_1, \cdots, x_m])
\end{align*}
This completes the proof.
\end{proof}

Moreover, denote by $\1\in\End_V(1)$ the identity map on $V$.
Clearly, the composition $\circ$ satisfies the following
coherence axioms:
\begin{enumerate}
\item (Identity)
$$f\circ(\1,\1,\cdots,\1) = f =\1\circ f;$$
\item (Associativity)
\begin{align*}
&f\circ(f_1\circ(f_{1,1},\cdots, f_{1,k_1}),
\cdots, f_n\circ(f_{n,1},\cdots, f_{n,k_n}))\\
=& (f\circ(f_1,\cdots,f_n))\circ(f_{1,1},
\cdots, f_{1,k_1},\cdots, f_{n,1}, \cdots, f_{n,k_n}).
\end{align*}
\end{enumerate}

\subsection{Associative algebras and the operad $\As$}
\label{xxsec8.3}
Recall that an \emph{associative algebra} (over $\k$) is a $\k$-vector space
$A$ equipped with a binary operation,
\[\mu\colon A\ot A\to A,\qquad \mu(a,b)=ab\]
satisfying the associative law $\mu\circ(\mu\ot \id_A) = \mu\circ(\id_A\ot\mu)$.
If moreover, there exists a linear map $u\colon \k\to A$ such that
$\mu\circ(u\ot\id_A) = \id_A = \mu\circ(\id_A\ot u)$, then $A$ is said to be
\emph{unital}.

The famous operad $\As$ encodes the category of unital
associative algebras, namely, unital associative algebras are exactly
$\As$-algebras. Recall that, for each $n\ge0$,
$\As(n) =\k\S_n$ as a right $\S_n$-module, and the composition $\circ$ is
given by
$$\sigma\circ (\sigma_1,\cdots, \sigma_n) =
\vt_{n; k_1, \cdots, k_n}(\sigma, \sigma_1,\cdots, \sigma_n)$$
for all $n>0$, $k_1,\cdots, k_n\ge0$, and $\sigma\in\S_n$ and
$\sigma_i\in\S_{k_i}$ for $1\le i\le n$.
It is direct to verify that $\As$ is an operad with the identity
$\1_1 := 1_{\S_1}\in \As(1)$ \cite[Section 9.1.3]{LV}. From
now on, we denote $1_{\S_n}$ by $1_n$ (or $\1_n$) for short for all $n\ge0$.

Let $(A, \gamma)$ be an $\As$-algebra. Clearly $\mu:=\gamma(1_{2})$
gives a binary operator on $A$, which is associative since
\begin{equation}
\label{E8.2.2}\tag{E8.2.2}
1_{\S_2}\circ(1_{2}, 1_{1}) =1_{3} = 1_{2}\circ(1_{1}, 1_{2}).
\end{equation}
Moreover, $1_{0}$ gives a linear map $u:=\gamma(1_{0})\colon \k\to A$,
and the fact that
\begin{equation}
\label{E8.2.3}\tag{E8.2.3}
1_{2}\circ(1_{0}, 1_{1}) = 1_{1} = 1_{2}\circ(1_{1}, 1_{0})
\end{equation}
means that $u$ is the unit map of $A$. Thus $(A, \mu, u)$ is a unital algebra.
Conversely, for every unital associative algebra $(A, \mu, u)$, we may define
$\gamma\colon \As\to \End_A$ as follows. By definition, $\gamma(1_{0}) = u$,
and for each $n>0$ and each $\sigma\in \S_n$, $\gamma(\sigma)$ is given by
$$\gamma(\sigma)\colon A^{\ot n}\to A,\quad a_1\ot\cdots\ot a_n\mapsto
a_{\sigma^{-1}(1)}a_{\sigma^{-1}(2)}\cdots a_{\sigma^{-1}(n)}$$
for all $a_1, \cdots, a_n \in A$, where the right hand side in the above formula
means the multiplication in $A$. It is direct to check that $\gamma$ is a morphism
of operads (these are standard facts in the book \cite{LV}).

\noindent\textbf{Acknowledgments}.
We would like to thank Marcelo Aguiar, Jonathan Beardsley,
Zheng-Fang Wang, Xiao-Wei Xu and Zhi-Bing Zhao for many useful
conversations. Y.-H. Bao and Y. Ye was supported by NSFC (Grant
No. 11871071, 11431010 and 11571329) and J. J. Zhang by the US National Science
Foundation (Grant No. DMS-1700825).


\begin{thebibliography}{9999}


\bibitem[AM1]{AM1}
M. Aguiar and S. Mahajan,
Monoidal functors, species and Hopf algebras,
with forewords by Kenneth Brown and Stephen Chase and Andr{\'e} Joyal.
CRM Monograph Series, 29. American Mathematical Society, Providence, RI, 2010.

\bibitem[AM2]{AM2}
M. Aguiar and S. Mahajan,
Hopf monoids in the category of species.
{\it Hopf algebras and tensor categories}, 17--124, Contemp. Math.,
{\bf 585}, Amer. Math. Soc., Providence, RI, 2013.


\bibitem[Am]{Am}
S. A. Amitsur, The identities of PI-rings,
Proc. Amer. Math. Soc. {\bf 4}(1) (1953), 27--34.



\bibitem[BV]{BV}
J.M. Boardman and R.M. Vogt,
Homotopy Invariant Algebraic Structures on Topological Spaces ,
Lecture Notes in Math., vol. {\bf 347}, Springer-Verlag,
Berlin $\cdot$ Heidelberg $\cdot$ New York, 1973.

\bibitem[Br]{Br}
M.R. Bremner,
One-parameter deformations of the diassociative and dendriform operads,
Linear Multilinear Algebra {\bf 65} (2017), no. 2, 341--350.


\bibitem[BG]{BG}
K.A. Brown and P. Gilmartin,
Quantum homogeneous spaces of connected Hopf algebras.
J. Algebra {\bf 454} (2016), 400--432.

\bibitem[DK]{DK}
V. Dotsenko and A. Khoroshkin,
Gr$\ddot{\rm o}$bner bases for operads,
Duke Math. J. {\bf 153}(2) (2010), 363--396.

\bibitem[Cu]{Cu}
P.-L. Curien,
Operads, clones, and distributive laws.
(English summary) Operads and universal algebra, 25--49,
Nankai Ser. Pure Appl. Math. Theoret. Phys., 9, World Sci. Publ.,
Hackensack, NJ, 2012.


\bibitem[Dr]{Dr}
V. Drensky,
Gelfand-Kirillov dimension of PI-algebras.
Methods in ring theory (Levico Terme, 1997), 97--113,
Lecture Notes in Pure and Appl. Math., {\bf 198},
Dekker, New York, 1998.



\bibitem[Fr1]{Fr1}
B. Fresse,
Homotopy of operads and Grothendieck-Teichm\"uller groups.
Part 1. The algebraic theory and its topological background.
Mathematical Surveys and Monographs, {\bf 217}.
American Mathematical Society, Providence, RI, 2017.


\bibitem[Fr2]{Fr2}
B. Fresse,
Homotopy of operads and Grothendieck-Teichm\"uller groups.
Part 2. The applications of (rational) homotopy theory methods.
Mathematical Surveys and Monographs, {\bf 217}. American Mathematical
Society, Providence, RI, 2017.

\bibitem[GZ1]{GZ1}
A. Giambruno and M. Zaicev,
PI-algebras and codimension growth.
Methods in ring theory (Levico Terme, 1997), 115--120,
Lecture Notes in Pure and Appl. Math., {\bf 198}, Dekker, New York, 1998.

\bibitem[GZ2]{GZ2}
A. Giambruno and M. Zaicev,
On Codimension Growth of Finitely Generated Associative Algebras,
Adv. Math. {\bf 140} (2) (1998), 145--155.

\bibitem[GZ3]{GZ3}
A. Giambruno and M. Zaicev,
Exponential Codimension Growth of PI Algebras: An Exact Estimate,
Adv. Math. {\bf 142}(2) (1999),  221--243.

\bibitem[GMZ]{GMZ}
A. Giambruno, S. Mishchenko and M. Zaicev,
Codimensions of algebras and growth functions.
Adv. Math. {\bf 217} (2008), no. 3, 1027-1052.



\bibitem[GK]{GK}
V. Ginzburg and M. Kapranov,
Koszul duality for operads,
Duke Math. J. {\bf 76}(1) (1994), 203--272.

\bibitem[Gir]{Gir}
S. Giraudo, Combinatorial operads from monoids,
J. Algebr. Comb. {\bf 41} (2015), 493--538.

\bibitem[Go]{Go}
K. R. Goodearl,
A Dixmier-Moeglin equivalnce for Poisson algebras with torus actions,
Algebra and its applications, 131--154, Contemp. Math.,
419, Amer. Math. Soc., Providence, RI, 2006.

\bibitem[GNPR]{GNPR}
S.F. Guill{\'e}n, V. Navarro, P. Pascual and A. Roig,
Moduli spaces and formal operads,
Duke Math. J. {\bf 129} (2005), no. 2, 291--335.

\bibitem[Jo]{Jo}
A. Joyal,
Une th{\'e}orie combinatoire des s{\'e}ries formelles.
(French) [A combinatorial theory of formal series]
Adv. in Math. {\bf 42} (1981), no. 1, 1--82.

\bibitem[Kle]{Kle}
A. A. Klein, PI-algebras satisfying identities of degree 3,
Trans. Amer. Math. Soc. {\bf 201} (1975), 263--277.

\bibitem[KP]{KP}
A. Khoroshkin and D. Piontkovski,
On generating series of finitely presented operads,
J. Algebra {\bf 426} (2015), 377--429.

\bibitem[KR]{KR}
A. A. Klein and A. Regev,
The codimensions of a PI-algebra,
Israel J. Math. {\bf 12} (1972), 421--426.

\bibitem[KL]{KL}
G. R. Krause and T. H. Lenagan,
\emph{Growth of algebras and Gelfand-Kirillov dimension},
Research Notes in Mathematics, Pitman Adv. Publ. Program,
{\bf 116} (1985).

\bibitem[Ku]{Ku}
D.E. Knuth,
\emph{The Art of Computer Programming}, Vol. III: Sorting and Searching,
Addison-Wesley, Reading, MA, second ed., 1998.


\bibitem[LL]{LL}
M. Livernet, J.-L. Loday,
The Poisson operad as a limit of associative operads,
unpublished preprint, March 1998.

\bibitem[LP]{LP}
M. Livernet and F. Patras,
Lie theory for Hopf operads,
{\it J. Algebra} {\bf 319} (2008), no. 12, 4899--4920.


\bibitem[LV]{LV}
J.L. Loday, and B. Vallette,
Algebraic operads,
Grundlehren der Mathematischen Wissenschaften [Fundamental Principles of Mathematical Sciences],
{\bf 346}, Springer, Heidelberg, 2012.

\bibitem[MaR]{MaR}
M. Markl and E. Remm,
Algebras with one operation including Poisson and other Lie-admissible algebras,
J. Algebra {\bf 299} (2006), no. 1, 171--189.


\bibitem[MSS]{MSS}
M. Markl, S. Shnider, and  J. Stasheff,
Operads in Algebra, Topology and Physics, Math. Surveys Monogr.,
vol. {\bf 96}, Amer. Math. Soc., Providence, RI, 2002.

\bibitem[Ma]{Ma}
J.P. May, The Geometry of Iterated Loop Spaces,
Lecture Notes in Math., vol. {\bf 271}, Springer-Verlag, Berlin $\cdot$ Heidelberg $\cdot$ New York, 1972.

\bibitem[McR]{McR}
J.C. McConnell, and J.C. Robson,
Noncommutative Noetherian rings With the cooperation of L. W. Small.
Pure and Applied Mathematics (New York). A Wiley-Interscience Publication.
John Wiley \& Sons, Ltd., Chichester, 1987.

\bibitem[Pr]{Pr}
H. Prodinger,
Some information about the binomial transform. Fibonacci Quart.
{\bf 32} (1994), no. 5, 412--415.

\bibitem[SS]{SS}
M.Z. Spivey and L.L. Steil,
The $k$-binomial transforms and the Hankel transform,
J. Integer Seq. 9 (2006), no. 1, Article 06.1.1, 19 pp.

\bibitem[StZ]{StZ}
D.R. Stephenson and J.J. Zhang,
Growth of graded Noetherian rings,
Proc. Amer. Math. Soc. {\bf 125} (1997), no. 6, 1593--1605.

\bibitem[Sz]{Sz}
A. Szendrei,
Clones in universal algebra.
S{\'e}minaire de Math{\'e}matiques Sup{\'e}rieures
[Seminar on Higher Mathematics], {\bf 99}
Presses de l'Universit{\'e} de Montr{\'e}al, Montreal, QC, 1986. 166 pp.



\bibitem[We]{We}
M.S. Weiss,
Truncated operads and simplicial spaces,
Tunisian Journal of Mathematics,
{\bf 1}, (2019) No. 1, 109--126.
dx.doi.org/10.2140/tunis.2019.1.109



\end{thebibliography}
\end{document}